\documentclass{birkmult}
\usepackage[T1]{fontenc}
\usepackage[latin1]{inputenc}
\usepackage{color}
\usepackage[english]{babel}
\usepackage{verbatim}
\usepackage{amsthm}
\usepackage{amsmath}
\usepackage{amssymb}
\usepackage[unicode=true,pdfusetitle,
 bookmarks=true,bookmarksnumbered=false,bookmarksopen=false,
 breaklinks=false,pdfborder={0 0 1},backref=false,colorlinks=false]
 {hyperref}

\makeatletter

\pdfpageheight\paperheight
\pdfpagewidth\paperwidth

\numberwithin{equation}{section}
\numberwithin{figure}{section}
\theoremstyle{plain}
\newtheorem{thm}{\protect\theoremname}[section]
  \theoremstyle{plain}
  \newtheorem{cor}[thm]{\protect\corollaryname}
  \theoremstyle{definition}
  \newtheorem{example}[thm]{\protect\examplename}
  \theoremstyle{definition}
  \newtheorem{defn}[thm]{\protect\definitionname}
  \theoremstyle{remark}
  \newtheorem{rem}[thm]{\protect\remarkname}
  \theoremstyle{plain}
  \newtheorem{prop}[thm]{\protect\propositionname}

\usepackage[T1]{fontenc}    

\usepackage{a4wide}        
\usepackage{amsmath}
\usepackage{enumerate}
\usepackage{euscript}
\usepackage{mathtools}
\allowdisplaybreaks[1] 
\usepackage{amsfonts}

\newcommand*{\trace}{\operatorname{trace}}

\newcommand*{\dive}{\operatorname{div}}

\newcommand*{\curl}{\operatorname{curl}}

\newcommand*{\Grad}{\operatorname{Grad}}
\newcommand*{\Dive}{\operatorname{Div}}

\newcommand*{\grad}{\operatorname{grad}}

\newcommand*{\ii}{\mathrm{i}}
\DeclareMathOperator*{\esssup}{ess-sup}

\newcommand{\1}{{\chi}}
\newcommand{\X}{\1}

\DeclareMathAccent{\Circ}{\mathalpha}{operators}{"17}

\renewcommand{\Im}{\operatorname{\mathfrak{Im}}}
\renewcommand{\Re}{\operatorname{\mathfrak{Re}}}

\newcommand{\oi}[2]{\left]#1,#2 \right[}

\newcommand{\lci}[2]{\left[#1,#2 \right[}

\renewcommand{\hat}{\widehat}
\renewcommand{\tilde}{\widetilde}
\renewcommand*{\epsilon}{\varepsilon}
\renewcommand*{\theta}{\vartheta}
\renewcommand*{\rho}{\varrho}

\theoremstyle{definition}
\newtheorem{hyp}[thm]{Hypotheses}

\makeatother

  \providecommand{\corollaryname}{Corollary}
  \providecommand{\definitionname}{Definition}
  \providecommand{\examplename}{Example}
  \providecommand{\propositionname}{Proposition}
  \providecommand{\remarkname}{Remark}
\providecommand{\theoremname}{Theorem}

\begin{document}

\title{Well-posedness via Monotonicity -- An Overview.}
\author{Rainer Picard}

\email{rainer.picard@tu-dresden.de}

\author{Sascha Trostorff}
\email{sascha.trostorff@tu-dresden.de}

\author{Marcus Waurick}

\email{marcus.waurick@tu-dresden.de}

\address{%
Institut f\"ur Analysis\\ Fachrichtung Mathematik\\
Technische Universit\"at Dresden\\
Zellescher Weg 12-14\\
D-01062 Dresden\\
Germany}
\subjclass{3502,35D30,35F16,35F25,35F61,35M33,35Q61,\\
35Q79,35Q74,74Q05,74Q10,74C10,74B05,78M40}

\keywords{positive definiteness, monotonicity, material laws, coupled systems, multiphysics}

\date{\today}

\begin{abstract}
The idea of monotonicity is shown to be the central theme of the solution theories associated with problems of mathematical physics. A ``grand unified'' setting is surveyed covering a comprehensive class of such problems. We illustrate on the applicability of this setting with a number examples. A brief discussion of stability and homogenization issues within this framework is also included.
\end{abstract}

\maketitle









%


\setcounter{section}{-1}
\newpage

\section{Introduction}

In this paper we shall survey a particular class of problems, which
we like to refer to as ``evolutionary equations'' (to distinguish
it from the class of explicit first order ordinary differential equations
with operator coefficients predominantly considered under the heading
of \emph{evolution equations}). This problem class is spacious enough
to include not only classical evolution equations but also partial
differential algebraic systems, functional differential equations
and integro-differential equations. Indeed, by thinking of elliptic
systems as time dependent, for example as constant with respect to
time on the connected components of $\mathbb{R}\setminus\left\{ 0\right\} $,
they also can be embedded into this class. The setting is -- in its
present state -- largely limited to a Hilbert space framework. As
a matter of convenience the discussion will indeed be set in a complex
Hilbert space framework. For the concept of monotonicity it is, however,
more appropriate to consider complex Hilbert spaces as real Hilbert
spaces, which can canonically be achieved by reducing scalar multiplication
to real numbers and replacing the inner product by its real part.
So, a binary relation $R$ in a complex Hilbert space $H$ with inner
product $\left\langle \,\cdot\,|\,\cdot\,\right\rangle _{H}$ would
be called strictly monotone if 
\[
\Re\left\langle x-y|u-v\right\rangle _{H}\geq\gamma\left\langle x-y|x-y\right\rangle _{H}
\]
for all $\left(x,u\right),\left(y,v\right)\in R$ holds and $\gamma$
is some positive real number. In case $\gamma=0$ the relation $R$
would be called monotone.

The importance of strict monotonicity, which in the linear operator
case reduces to strict positive definiteness%
\footnote{We use the term \emph{strict positive definiteness} for a linear operator
$A$ in a real or complex Hilbert space $X$ in the sense naturally
induced by the classification of the corresponding quadratic form
$Q_{A}$ given by $u\mapsto\left\langle u|Au\right\rangle _{X}$ on
its domain $D\left(A\right)$. So, if $Q_{A}$ is non-negative (mostly
called positive semi-definite), positive definite, strictly positive
definite, then the operator $A$ will be called non-negative (usually
called positive), positive definite, strictly positive definite, respectively.
If $X$ is a complex Hilbert space it follows that $A$ must be Hermitean.
Note that we do \emph{not} restrict the definition of non-negativity,
positive definiteness, strict positive definiteness to Hermitean or
symmetric linear operators.%
}, is of course well-known from the elliptic case. By a suitable choice
of space-time norm this key to solving elliptic partial differential
equation problems also allows to establish well-posedness for dynamic
problems in exactly the same fashion.

The crucial point for this extension is the observation that the one
dimensional derivative itself, acting as the time derivative%
\footnote{We follow here the time-honored convention that physicists practice
by labeling the partial time derivative by index zero.%
} $\partial_{0}$ (on the full time line $\mathbb{R}$), can be realized
as a maximal strictly positive definite operator in an appropriately
exponentially weighted $L^{2}$-type Hilbert space over the real time-line
$\mathbb{R}$. It is in fact this strict positive definiteness of
$\partial_{0}$ which opens access to the problem class we shall describe
later.

Indeed, \emph{$\partial_{0}$ }simply turns out to be a\emph{ normal
}operator with $\Re\partial_{0}$ being just multiplication by a positive
constant. Moreover, this time-derivative $\partial_{0}$ is continuously
invertible and, as a normal operator, admits a straightforwardly defined
functional calculus, which can canonically be extended to operator-valued
functions. Indeed, since we have control over the positivity constant
via the choice of the weight, the norm of $\partial_{0}^{-1}$ can
be made as small as wanted. This observation is the Hilbert space
analogue to the technical usage of the exponentially weighted sup-norm
as introduced by D. Morgenstern, \cite{morgenstern1952beitraege},
and allows for the convenient inclusion of a variety of perturbation
terms.

Having established time-differentiation $\partial_{0}$ as a normal
operator, we are led to consider evolutionary problems as operator
equations in a space-time setting, rather than as an ordinary differential
equation in a spatial function space. The space-time operator equation
perspective implies to deal with sums of unbounded operators, which,
however, in our particular context is -- due to the limitation of
remaining in a Hilbert space setting and considering only sums, where
one of the terms is a function of the normal operator $\partial_{0}$
-- not so deep an issue. For more general operator sums or for a Banach
space setting more sophisticated and powerful tools from the abstract
theory of operator sums initiated by the influential papers by da
Prato and Grisvard, \cite{Prato1975}, and Brezis and Haraux, \cite{zbMATH03505695},
may have to be employed. In these papers operator sums $\frac{d}{dt}+A$
typically occurring in the context of explicit first order differential
equations in Banach spaces are considered as applications of the abstract
theory, compare also e.g. \cite[Chapter 2, Section 7]{zbMATH03302617}.
The obvious overlap with the framework presented in this paper would
be the Hilbert space situation in the case $\mathcal{M}=1$. We shall,
however, not pursue to explore how the strategies developed in this
context may be expanded to include more complicated material laws,
which indeed has been done extensively in the wake of these ideas,
but rather stay with our limited problem class, which covers a variety
of diverse problems in a highly unified setting. Naturally the results
available for specialized cases are likely to be stronger and more
general. 

For introductory purposes let us consider the typical linear case
of such a space-time operator equation
\begin{equation}\label{eq:evo_1}
\partial_{0}V+AU=f,
\end{equation}
where $f$ are given data, $A$ is a -- usually -- purely spatial
-- prototypically \emph{skew-selfadjoint}%
\footnote{\label{fn:paradigm}Note that in our canonical reference situation
$A$ is \emph{skew}-selfadjoint rather than selfadjoint and so we
have $\Re\left\langle u|Au\right\rangle _{H}=0$ for all $u\in D\left(A\right)$
and coercitivity of $A$ is out of the question. To make this concrete:
let $\partial_{1}$ denote the weak $L^{2}\left(\mathbb{R}\right)$-derivative.
Then our paradigmatic reference example on this elementary level would
be the transport operator $\partial_{0}+\partial_{1}$ rather than
the heat conduction operator $\partial_{0}-\partial_{1}^{2}$.%
} -- operator and the quantities $U,V$ are linked by a so-called material
law
\[
V=\mathcal{M}U.
\]
Solving such an equation would involve establishing the bounded invertibility
of $\overline{\partial_{0}\mathcal{M}+A}$. As a matter of ``philosophy''
we shall think of the -- here linear -- material law operator $\mathcal{M}$
as encoding the complexity of the physical material whereas $A$ is
kept simple and usually only contains spatial derivatives. If $\mathcal{M}$
commutes with $\partial_{0}$ we shall speak of an autonomous system,
otherwise we say the system is non-autonomous. 

Another -- more peripheral -- observation with regards to the classical
problems of mathematical physics is that they are predominantly of
first order not only with respect to the time derivative, which is
assumed in the above, but frequently even in\emph{ both} the temporal
and spatial derivatives. Indeed, acoustic waves, heat transport, visco-elastic
and electro-magnetic waves etc. are governed by first order systems
of partial differential operators, i.e. $A$ is a first order differential
operator in spatial derivatives, which only after some elimination
of unknowns turn into the more common second order equations, i.e.
the wave equation for the pressure field, the heat equation for the
temperature distribution, the visco-elastic wave equation for the displacement
field and the vectorial wave equation for the electric (or magnetic)
field. It is, however, only in the direct investigation of the first
order system that, as we shall see, the unifying feature of monotonicity becomes easily visible. Moreover, the first order formulation
reveals that the spatial derivative operator $A$ is of a Hamiltonian
type structure and consequently, by imposing suitable boundary conditions,
turn out -- in the standard cases -- to lead to skew-selfadjoint $A$
in a suitable Hilbert space $H$. So, from this perspective there
is also undoubtedly a flavor of the concept of \emph{symmetric hyperbolic
systems} as introduced by K. O. Friedrichs, \cite{Ref166}, and of
Petrovskii well-posedness, \cite{Petrovskii1986}, at the roots of
this approach.

For illustrational purposes let us consider from a purely heuristic
point of view the $\left(1+1\right)$-dimensional system 
\begin{equation}
\left(\partial_{0}M_{0}+M_{1}+A\right)\left(\begin{array}{c}
p\\
s
\end{array}\right)=\left(\begin{array}{c}
f\\
0
\end{array}\right),\label{eq:ex-system}
\end{equation}
where 
\[
M_{0}\coloneqq\left(\begin{array}{cc}
\eta & 0\\
0 & \alpha
\end{array}\right),\: M_{1}\coloneqq\left(\begin{array}{cc}
\left(1-\eta\right) & 0\\
0 & \left(1-\alpha\right)
\end{array}\right),\:\alpha,\eta\in\left\{ 0,1\right\} ,\: A\coloneqq\left(\begin{array}{cc}
0 & \partial_{1}\\
\partial_{1} & 0
\end{array}\right),
\]
 and $\partial_{1}$ is simply the weak $L^{2}\left(\mathbb{R}\right)$-derivative,
compare Footnote \ref{fn:paradigm}. Assuming $\eta=1$, $\alpha=1$,
in (\ref{eq:ex-system}) clearly results in a (symmetric) hyperbolic
system and eliminating the unknown $s$ yields the wave equation in
the form 
\[
\left(\partial_{0}^{2}-\partial_{1}^{2}\right)p=\partial_{0}f.
\]
For $\eta=1$,$\alpha=0$ we obtain a differential algebraic system,
which represents the parabolic case in the sense that after eliminating
$s$ we obtain the heat equation 
\[
\left(\partial_{0}-\partial_{1}^{2}\right)p=f.
\]
Finally, if both parameters vanish, we obtain a 1-dimensional elliptic
system and as expected after eliminating the unknown $s$ a 1-dimensional
elliptic equation for $p$ results: 
\[
\left(1-\partial_{1}^{2}\right)p=f.
\]
Allowing now $\alpha,\eta$ to be $L^{\infty}\left(\mathbb{R}\right)$-multiplication
operators with values in $\left\{ 0,1\right\} $, which would allow
the resulting equations to jump in space between elliptic, parabolic
and hyperbolic ``material properties'', could be a possible scenario
we envision for our framework. As will become clear, the basic idea
of this simple ``toy'' example can be carried over to general evolutionary
equations. Also in this connotation there are stronger and more general
results for specialized cases. A problem of this flavor of ``degeneracy''
has been for example discussed for a non-autonomous, degenerate integro-differential
equation of parabolic/elliptic type in \cite{zbMATH05124752,zbMATH05530294}. 

A prominent feature distinguishing general operator equations from
those describing dynamic processes is the specific role of time, which
is not just another space variable, but characterizes dynamic processes
via the property of causality%
\footnote{Note that this perspective specifically excludes the case of a periodic
time interval, where ``before'' and ``after'' makes little sense.%
}. Requiring causality for the solution operator $\left(\overline{\partial_{0}\mathcal{M}+A}\right)^{-1}$
results in very specific types of material law operators $\mathcal{M}$,
which are causal and compatible with causality of $\left(\overline{\partial_{0}\mathcal{M}+A}\right)^{-1}$.
This leads to deeper insights into the structural properties of mathematically
viable models of physical phenomena. 

The solution theory can be extended canonically to temporal distributions
with values in a Hilbert space. In this perspective initial value
problems, i.e. prescribing $V\left(0+\right)$ in (\ref{eq:evo_1}), amount to allowing
a source term $f$ of the form $\delta\otimes V_{0}$ defined by
\[
\left(\delta\otimes V_{0}\right)\left(\varphi\right)\coloneqq\left\langle V_{0}|\varphi\left(0\right)\right\rangle _{H}
\]
for $\varphi$ in the space $C_{c}\left(\mathbb{R},H\right)$ of continuous
$H$-valued functions with compact support. This source term encodes
the classical initial condition $V\left(0+\right)=V_{0}$. For the
constant coefficient case -- say -- $\mathcal{M}=1$, it is a standard
approach to establish the existence of a fundamental solution (or
more generally, e.g. in the non-autonomous case, a Green's functions)
and to represent general solutions as convolution with the fundamental
solution. This is of course nothing but a description of the continuous
one-parameter semi-group approach. Indeed, such a semi-group $U=\left(U\left(t\right)\right)_{t\in\lci0\infty}$
is, if extended by zero to the whole real time line, nothing but the
fundamental solution $G=\left(G\left(t\right)\right)_{t\in\mathbb{R}}$
with $G\left(t\right)\coloneqq\begin{cases}
U\left(t\right) & \mbox{ for }t\in\lci0\infty,\\
0 & \mbox{ for }t\in\oi{-\infty}0.
\end{cases}$ In the non-autonomous case, the role of $U$ is played by a so-called
evolution family. The regularity properties of such fundamental solutions
results in stronger regularity properties of the corresponding solutions. Since we allow $\mathcal{M}$
to be more general, constructing such fundamental solutions/Green's
functions is not always available or feasible. Indeed, we shall focus
for sake of simplicity on the case that the data $f$ do not contain
such Dirac type sources, which can be achieved simply by subtracting
the initial data or by including distributional objects such as $\delta\otimes V_{0}$
in the Hilbert space structure via extension to extrapolation spaces,
which for sake of simplicity we will not burden this presentation
with.

As a trade-off for our constraint, which in the simplest linear case
would reduce our discussion to considering $\partial_{0}+A$ as a
sum of commuting normal operators, which clearly cannot support any
claim of novelty, see e.g. \cite{zbMATH02500655}, we obtain by allowing
for a large class of material law operators $\mathcal{M}$ access
to a large variety of problems including such diverse topics as partial
differential-algebraic systems, integro-differential equations and
evolutionary equations of changing type in one unified setting. 

Based on the linear theory one has of course a first access to non-linear
problems by including Lipschitz continuous perturbations. A different
generalization towards a non-linear theory can be done by replacing
the (skew-selfadjoint) operator $A$ by a maximal monotone relation
or allowing for suitable maximal monotone material law relations (rather
than material law operators). In this way the class of evolutionary
problems also comprises evolutionary \emph{inclusion}s. 

Having established well-posedness, qualitative properties associated
with the solution theory come into focus. A first step in this direction
is done for the autonomous case by the discussion of the issue of
``exponential stability''. One can give criteria with regards to
the material law $\mathcal{M}$ ensuring exponential stability. 

Another aspect in connection with the discussion of partial differential
equations of mathematical physics is the problem of continuous dependence
of the solution on the coefficients. A main application of results
in this direction is the theory of homogenization, i.e., the study
of the behavior of solutions of partial differential equations having
large oscillatory coefficients. It is natural to discuss the weak
operator topology for the coefficients and it turns out that the problem
class under consideration is closed under limits in this topology
if further suitable structural assumptions are imposed. The closedness
of the problem class is a remarkable feature of the problem class,
which is spacious enough to also include -- hidden in the generality
of the material law operator -- integro-differential evolutionary
problems. In this regard it is worth recalling that there are examples
already for ordinary differential equations, for which the resulting
limit equations are of integro-differential type, showing that differential
equations are in this respect a too small problem class.

Although links to the core concepts which have entered the described
approach are too numerous to be recorded here to any appropriate extent,
we shall try modestly to put them in a bibliographical context. The
concept of the time-derivative considered as a continuously invertible
operator in a suitably weighted Hilbert space has its source in \cite{Picard1989}.
It has been employed in obtaining a solution theory for evolutionary
problems in the spirit described above only more recently, compare
e.g. \cite[Chapter 6]{Picard_McGhee}. General perspectives for well-posedness
to partial differential equations via strict positive definiteness
are of course at the heart of the theory of elliptic partial differential
equations. 

For the theory of maximal monotone operators/relations, we refer to
\cite{Brezis1971,papageogiou,hu2000handbook,Morosanu1988}. For non-autonomous equations,
we refer to \cite{Sohr1978,Tanabe1979} and -- with a focus on maximal
regularity -- to \cite{ADLO12}. Note that due to the generality of
our approach, one cannot expect maximal regularity of the solution
operator in general. In fact, maximal regularity for the solution
operator just means that the operator sum is already closed with its
natural domain. This is rarely the case neither in the paradigmatic
examples nor in our expanded general setting. 

For results regarding exponential stability for a class of hyperbolic
integro-differential equations, we refer to \cite{Pruss2009} and
to \cite{Datko1970,Gearhardt1978,EngNag} for the treatment of this
issue in the context of one-parameter semi-groups. A detailed introduction
to the theory of homogenization can be found in \cite{BenLiPap} and
in \cite{CioDon}. We also refer to \cite{TartarMemEff,TartarNonlHom},
where homogenization for ordinary differential equations has been
discussed extensively. 

The paper itself is structured as follows. We begin our presentation
with a description of the underlying prerequisites, even to the extent
that we review the celebrated well-posedness requirements due to Hadamard,
which we found inspirational for a deeper understanding of the case
of differential inclusions. A main point in this first section is
to introduce the classical concept of maximal strictly monotone relations
and to recall that such relations are inverse relations of Lipschitz
continuous mappings (Minty's Theorem \ref{thm:Minty}). Specializing
to the linear case we recall in particular the Lax-Milgram lemma (Corollary
\ref{cor:Lax-Mil}) and as a by-product derive a variant of the classical
solution theory for elliptic type equations. Moreover, we comment
on a general solution theory for (non-linear) elliptic type equations
in divergence form relying only on the validity of a Poincar\'e type
estimate (Theorem \ref{thm: STHE}). We conclude this section with
an example for an elliptic type equation with possible degeneracies
in the coefficients as an application of the ideas presented.

Based on the first section's general findings, Section \ref{sec:Linear-Evolutionary-Equations}
deals with the solution theory for linear evolutionary equations.
After collecting some guiding examples in Subsection \ref{sub:Leading-Examples},
we rigorously establish in Subsection \ref{sub:The-Time-Derivative}
the time-derivative as a strictly monotone, normal operator in a suitably
weighted Hilbert space. Based on this and with resulting structural
properties, such as a functional calculus, at hand, in Subsection
\ref{sub:evo_eq_auto} (Theorem \ref{thm:sol_theo_skew}) we formulate
a solution theory for autonomous, linear evolutionary equations. The
subsequent examples review some of those mentioned in Subsection \ref{sub:Leading-Examples}
in a rigorous functional analytic setting to illustrate the applicability
of the solution theory. As further applications we show that Theorem
\ref{thm:sol_theo_skew} also covers integro-differential equations
(Theorem \ref{thm:integro}) and equations containing fractional time-derivatives
(Theorem \ref{thm:frac}). We conclude this subsection with a conceptual
study of exponential stability (Definition \ref{Def:exp_stab} and
Theorem \ref{thm:criterion-stab}) in our theoretical context. 

In Subsection \ref{sub:The-closedness-of}, starting out with a short
motivating introductory part concerning homogenization issues, we
discuss the closedness of the problem class under the weak operator
topology for the coefficients. A first theorem in this direction is
then obtained as Theorem \ref{thm:Hom_1_auto_comp}. After presenting
some examples, we continue our investigation of homogenization problems
first for ordinary differential equations (Theorems \ref{thm:ode_degen}
and \ref{thm:ode_non_degen}). Then we formulate a general homogenization
result (Theorem \ref{thm:hom_gen_inf_null}), which is afterwards
exemplified by considering Maxwell's equations and in particular the
so-called eddy current problem of electro-magnetic theory. 

In Subsection \ref{sub:The-non-autonomous-case} we extend the solution
theory to include the non-autonomous case. A first step in this direction
is provided by Theorem \ref{thm:Solutiontheory}, for which the illustrative
Example \ref{ex:illustrating_non_autp} is given as an application.
A common generalization of the Theorems \ref{thm:sol_theo_skew} and
\ref{thm:Solutiontheory} is given in Theorem \ref{thm:non-auto2}.
This is followed by an adapted continuous dependence result Theorem
\ref{thm:general_cont_depend}, which in particular is applicable
to homogenization problems. A detailed example of a mixed type problem
concludes this section.

Section \ref{sec:Monotone-Evolutionary-Problems} gives an account
for a non-linear extension of the theory. Similarly to the previous
section, the results are considered in the autonomous case first (Subsection
\ref{sub:incl_auto}) and then generalized to the non-autonomous case
(Subsection \ref{sub:The-non-autonomous-case-incl}). Subsection \ref{sub:Problems-with-non-linear-bdy}
concludes this section and the paper with a discussion of an application
to evolutionary problems with non-linear boundary conditions. One
of the guiding conceptual ideas here is to avoid regularity assumptions
on the boundary of the underlying domain. This entails replacing the
classical boundary trace type data spaces, by a suitable generalized
analogue of $1$-harmonic functions. We exemplify our results with
an impedance type problem for the wave equation and with the elasic equations with frictional boundary conditions.

Note that inner products, indeed all sesqui-linear forms, are --
following the physicists habits -- assumed to be conjugate-linear
in the first component and linear in the second component.

\section{Well-posedness and Monotonicity\label{sec:Well-posedness-and-Monotonicity}}

To begin with, let us recall the well-known Hadamard requirements for
well-posedness. It is appropriate for our purposes, however, to formulate
them for the case of relations rather than -- as usually done -- for mappings.
Hadamard proposed to define what ``reasonably solvable'' should
entail. Solving a problem involves to establish a binary relation
$P\subseteq X\times Y$ between ``data'' in a topological space
$Y$ and corresponding ``solutions'' in a topological space $X$,
which is designed to cover a chosen pool of examples to our satisfaction.
Finding a solution then means, given $y\in Y$ find $x\in X$ such
that $\left(x,y\right)\in P$. If we wish to supply a solution for
all possible data, there are some natural requirements that the problem
class $P$ should have to ensure that this task is reasonably conceived.
To exclude cases of trivial failure to describe a solution theory
for $P$, we assume first that $P$ is already \emph{closed} in $X\times Y$.
Then well-posedness in the spirit of Hadamard requires the following
three properties.
\begin{enumerate}
\item (``Uniqueness'' of solution) the inverse relation $P^{-1}$ is right-unique,
thus, giving rise to a mapping%
\footnote{For subsets $M\subseteq X,\, N\subseteq Y$ the\emph{ post-set of
$M$ under $P$ }and the \emph{pre-set of $N$ under $P$ }is defined
as $P[M]\coloneqq\left\{ y\in Y\,|\,\bigvee_{x\in M}(x,y)\in P\right\} $
and $[N]P\coloneqq\left\{ x\in X\,|\,\bigvee_{y\in N}(x,y)\in P\right\} $,
respectively. The post-set $P\left[X\right]$ of the whole space $X$
under $P$ is then the domain of the mapping $P^{-1}$.%
}
\[
P^{-1}:P\left[X\right]\subseteq Y\to X
\]
performing the association of ``data'' to ``solutions''. 
\item (``Existence'' for every given data) we have that
\[
P\left[X\right]=Y,
\]
i.e. $P^{-1}$ is defined on the whole data space $Y$.
\item (``Continuous dependence'' of the solution on the data) The mapping
$P^{-1}$ is continuous.
\end{enumerate}
In case of $P$ being a mapping then $[Y]P=P^{-1}\left[Y\right]$
is the \emph{domain $D\left(P\right)$ of $P$.} For our purposes
here we shall assume that $X=Y$ and that $X$ is a complex Hilbert
space. 

A very particular but convenient instance of well-posedness, which
nevertheless appears to dominate in applications, is the maximal monotonicity
of $P-c\coloneqq\left\{ (x,y-cx)\in X\times X\,|\,(x,y)\in P\right\} $
for some $c\in\oi0\infty$. Recall that a relation $Q\subseteq X\times X$
is called \emph{monotone} if\emph{ }
\[
\Re\left\langle x_{0}-x_{1}|y_{0}-y_{1}\right\rangle _{X}\geq0
\]
for all $\left(x_{0},y_{0}\right),\left(x_{1},y_{1}\right)\in Q$.
Such a relation $Q$ is called \emph{maximal} if there exists no proper
monotone extension in $X\times X$. In other words, if $\left(x_{1},y_{1}\right)\in X\times X$
is such that $\Re\left\langle x_{0}-x_{1}|y_{0}-y_{1}\right\rangle _{X}\geq0$
for all $\left(x_{0},y_{0}\right)\in Q$, then $\left(x_{1},y_{1}\right)\in Q$. 
\begin{thm}[Minty, \cite{Minty1962}]
\label{thm:Minty} Let $\left(P-c\right)\subseteq X\times X$ be
a maximal monotone relation%
\footnote{Note here that maximal monotone relations are automatically closed,
see e.g. \cite[Proposition 2.5]{Brezis1971}.%
} for some $c\in\oi0\infty$%
\footnote{In this case $P$ would be called \emph{maximal strictly monotone}.%
}. Then the inverse relation $P^{-1}$ defines a Lipschitz continuous
mapping with domain $D(P^{-1})=X$ and $\frac{1}{c}$ as possible
Lipschitz constant.\end{thm}
\begin{proof}
We first note that the monotonicity of $P-c$ implies
\begin{equation}
\bigwedge_{\left(x_{0},y_{0}\right),\left(x_{1},y_{1}\right)\in P}\Re\left\langle x_{0}-x_{1}|y_{0}-y_{1}\right\rangle _{X}\geq c\left\langle x_{0}-x_{1}|x_{0}-x_{1}\right\rangle _{X}.\label{eq:strict:pos}
\end{equation}
Hence, if $y_{0}=y_{1}$ then $x_{0}$ must equal $x_{1}$, i.e. the
uniqueness requirement is satisfied, making $P^{-1}:P\left[X\right]\to X$
a well-defined mapping. Moreover, $P\left[X\right]$ is closed, since
from (\ref{eq:strict:pos}) we get 
\[
\bigwedge_{\left(x_{0},y_{0}\right),\left(x_{1},y_{1}\right)\in P}\left|y_{0}-y_{1}\right|_{X}\geq c\left|x_{0}-x_{1}\right|_{X}.
\]
The actually difficult part of the proof is to establish that $P\left[X\right]=X$.
This is the part we will omit and refer to \cite{Minty1962} instead.
To establish Lipschitz continuity of $P^{-1}:X\to X$ we observe that
\begin{align*}
\bigwedge_{y_{0},y_{1}\in X}\: c\left|P^{-1}\left(y_{0}\right)-P^{-1}\left(y_{1}\right)\right|_{X}^{2} & \leq\Re\left\langle P^{-1}\left(y_{0}\right)-P^{-1}\left(y_{1}\right)|y_{0}-y_{1}\right\rangle _{X}\\
 & \leq\left|P^{-1}\left(y_{0}\right)-P^{-1}\left(y_{1}\right)\right|_{X}\left|y_{0}-y_{1}\right|_{X},
\end{align*}
holds, from which the desired continuity estimate follows.
\end{proof}
For many problems, the strict monotonicity is easy to obtain. The
maximality, however, needs a deeper understanding of the operators
involved. In the linear case, writing now $A$ for $P$, there is
a convenient set-up to establish maximality by noting that
\[
\left(\left[\left\{ 0\right\} \right]A^{*}\right)^{\perp}=\overline{A\left[X\right]}
\]
according to the projection theorem. Here we denote by $A^{\ast}$
the \emph{adjoint of $A$, }given as the binary relation 
\[
A^{\ast}\coloneqq\left\{ (u,v)\in X\times X\, | \,\bigwedge_{(x,y)\in A}\langle y|u\rangle_{X}=\langle x|v\rangle_{X}\right\} .
\]
Thus, maximality for the strictly monotone linear mapping (i.e. strictly
accretive) $A$ is characterized%
\footnote{Recall that $A$ has closed range.%
} by
\begin{equation}
\left[\left\{ 0\right\} \right]A^{*}=\left\{ 0\right\} ,\label{eq:max}
\end{equation}
i.e. the uniqueness for the adjoint problem. Characterization (\ref{eq:max})
can be established in many ways, a particularly convenient one being
to require that $A^{*}$ is also strictly monotone. With this we arrive
at the following result.
\begin{thm}
\label{lin-theo}Let $A$ and $A^{*}$ be closed linear strictly monotone
relations in a Hilbert space $X$. Then for every $f\in X$ there
is a unique $u\in X$ such that
\[
\left(u,f\right)\in A.
\]
Indeed, the solution depends continuously on the data in the sense
that we have a (Lipschitz-) continuous linear operator $A^{-1}:X\to X$
with
\[
u=A^{-1}f.
\]

\end{thm}
Of course, the case that $A$ is a closed, densely defined linear
operator is a common case in applications.
\begin{cor}
\label{lin:op:theo}Let A be a closed, densely defined, linear operator
and $A,\: A^{*}$ strictly accretive in a Hilbert space $X$. Then
for every $f\in X$ there is a unique $u\in X$ such that
\[
Au=f.
\]
Indeed, solutions depend continuously on the data in the sense that
we have a (Lipschitz) continuous linear operator $A^{-1}:X\to X$
with
\[
u=A^{-1}f.
\]

\end{cor}
In the case that $A$ and $A^{*}$ are linear operators with $D\left(A\right)=D\left(A^{*}\right)$
the situation simplifies, since then strict accretivity of $A$ implies
strict accretivity of $A^{*}$ due to
\[
\Re\left\langle x|Ax\right\rangle _{X}=\Re\left\langle A^{*}x|x\right\rangle _{X}=\Re\left\langle x|A^{*}x\right\rangle _{X}
\]
for all $x\in D\left(A\right)=D\left(A^{*}\right)$.

~
\begin{cor}
\label{lin:op:D(A):theo}Let $A$ be a closed, densely defined, linear
strictly accretive operator in a Hilbert space $X$ with $D\left(A\right)=D\left(A^{*}\right)$.
Then for every $f\in X$ there is a unique $u\in X$ such that
\[
Au=f.
\]
Indeed, the solution depends continuously on the data in the sense
that we have a continuous linear operator $A^{-1}:X\to X$ with
\[
u=A^{-1}f.
\]

\end{cor}
The domain assumption of the last corollary is obviously satisfied
if $A:X\to X$ is a continuous linear operator. This observation leads
to the following simple consequence.
\begin{cor}
\label{cont-lin-op-theo}Let $A:X\to X$ be a strictly accretive,
continuous, linear operator in the Hilbert space $X$. Then for every
$f\in X$ there is a unique $u\in X$ such that
\[
Au=f.
\]
Indeed, the solution depends continuously on the data in the sense
that we have a continuous linear operator $A^{-1}:X\to X$ with
\[
u=A^{-1}f.
\]

\end{cor}
Note that since continuous linear operators and continuous sesqui-linear
forms are equivalent, the last corollary is nothing but the so-called
Lax-Milgram theorem. Indeed, if $A:X\to X$ is in the space $L\left(X\right)$
of a continuous linear operators then 
\[
\left(u,v\right)\mapsto\left\langle u|Av\right\rangle _{X}
\]
is in turn a continuous sesqui-linear form on $X$, i.e. an element
of the space $S\left(X\right)$ of continuous sesqui-linear forms
on $X$, and conversely if $\beta\left\langle \:\cdot\:|\:\cdot\:\right\rangle \in S\left(X\right)$
then $\overline{\beta\left\langle \:\cdot\:|v\right\rangle }\in X^{*}$
and utilizing the unitary%
\footnote{Recall that for this we have to define the complex structure of $X^{*}$
accordingly as $\left(\alpha f\right)\left(x\right)\coloneqq\overline{\alpha}\, f\left(x\right)$
for every $x\in X$ and every continuous linear functional $f$ on
$X$.%
} Riesz map $R_{X}:X^{*}\to X$ we get via the Riesz representation
theorem $\beta\left\langle u|v\right\rangle =\left\langle u|A_{\beta}v\right\rangle _{X},$
where $A_{\beta}v\coloneqq R_{X}\overline{\beta\left\langle \:\cdot\:|v\right\rangle },\: v\in X,$
defines indeed a continuous linear operator on $X$. Moreover, 
\begin{align*}
S\left(X\right) & \to L\left(X\right)\\
\beta & \mapsto A_{\beta}
\end{align*}
is not only a bijection but also an isometry. Indeed, 
\[\left|\beta\right|_{S\left(X\right)}\coloneqq\sup_{x,y\in B_{X}\left(0,1\right)}\left|\beta\left(x,y\right)\right|=\sup_{x,y\in B_{X}\left(0,1\right)}\left|\left\langle x|A_{\beta}y\right\rangle _{X}\right|=\left\Vert A_{\beta}\right\Vert _{L\left(X\right)}.\]

Strict accretivity for the corresponding operator $A_{\beta}$ results
in the so-called \emph{coercitivity}%
\footnote{The strict positivity in (\ref{eq:strict:pos}) can be weakened to
requiring merely $\bigwedge_{u\in X}\left|\beta\left\langle u|u\right\rangle \right|\geq c\left\langle u|u\right\rangle _{X}$,
which yields in an analogous way a corresponding well-posedness result.
This option is used in some applications. %
} of the sesqui-linear form $\beta$:
\begin{equation}
\Re\beta\left\langle u|u\right\rangle \geq c\left\langle u|u\right\rangle _{X}\label{eq:coercive}
\end{equation}
for some $c\in\oi0\infty$ and all $u\in X$. Thus, as an equivalent
formulation of the previous corollary we get the following.
\begin{cor}[Lax-Milgram theorem]
\label{cor:Lax-Mil}Let $\beta\left\langle \:\cdot\:|\:\cdot\:\right\rangle $
be a continuous, coercive sesqui-linear form on a Hilbert space $X$.
Then for every $f\in X^{*}$ there is a unique $u\in X$ such that
\[
\beta\left\langle u|v\right\rangle =f\left(v\right)
\]
for all $v\in X$. 
\end{cor}
Keeping in mind the latter approach has been utilized extensively
for elliptic type problems it may be interesting to note that its
generalization in the form of Corollary \ref{lin:op:theo} is perfectly
sufficient to solve elliptic, parabolic and hyperbolic systems in
a single approach. For further illustrating the Lax-Milgram theorem
in its abstract form, we discuss an example, which is related to the
sesqui-linear forms method.
\begin{example}
\label{ex:Lax-Milgram}Let $H_{0}$, $H_{1}$ be Hilbert spaces. Denote
by $H_{-1}$ the dual of $H_{1}$ and let $R_{H_{1}}\colon H_{-1}\to H_{1}$
be the corresponding Riesz-isomorphism. Consider a continuous linear
bijection $C:H_{1}\to H_{0}$ and a continuous linear operator $A:H_{0}\to H_{0}$
with
\[
\Re\left\langle x|Ax\right\rangle _{H_{0}}\geq\alpha_{0}\left\langle x|x\right\rangle _{H_{0}}\quad(x\in H_{0})
\]
for some $\alpha_{0}\in\mathbb{R}_{>0}.$ Denoting
\[
C^{\diamond}:H_{0}\to H_{-1},y\mapsto C^{\diamond}y\coloneqq\left\langle y|C\:\cdot\:\right\rangle _{H_0},
\]
we consider
\[
C^{\diamond}AC:H_{1}\to H_{-1}.
\]
Now, from 
\[
R_{H_{1}}C^{\diamond}AC=C^{*}AC,
\]
we read off that $C^{\diamond}AC$ is an isomorphism. This may also
be seen as an application of the Lax-Milgram theorem, since the equation
\[
C^{\diamond}ACw=f,
\]
for given $f\in H_{-1}$ amounts to be equivalent to the discussion
of the sesqui-linear form 
\begin{eqnarray*}
(v,w)\mapsto\beta\left\langle v|w\right\rangle  & \coloneqq & \left\langle ACv|Cw\right\rangle _{H_0}=\left(C^{\diamond}ACv\right)\left(w\right)
\end{eqnarray*}
similar to the way it was done in the above. %

\end{example}
In order to establish a solution theory for elliptic type equations,
it is possible to go a step further. For stating an adapted well-posedness
theorem we recall the following. Let $G:D(G)\subseteq H_{1}\to H_{2}$
be a densely defined closed linear operator with closed range $R(G)=G[H_{1}]$.
Then, the operator $B_{G}\colon D(G)\cap N(G)^{\bot}\subseteq N(G)^{\bot}\to R(G),x\mapsto Gx$,
where $N(G)=[\{0\}]G$ denotes the null-space of $G$, is continuously
invertible as it is one-to-one, onto and closed. Consequently, the
modulus $\left|B_{G}\right|$ of $B_{G}$ is continuously invertible
on $N(G)^{\bot}$. We denote by $H_{1}(\left|B_{G}\right|)$ the domain
of $\left|B_{G}\right|$ endowed with the norm $\left|\left|B_{G}\right|\cdot\right|_{H_{1}},$
which can be shown to be equivalent to the graph norm of $\left|B_{G}\right|$.
We denote by $H_{-1}(\left|B_{G}\right|)$ the dual of $H_{1}(\left|B_{G}\right|)$
with the pivot space $H_{0}(\left|B_{G}\right|)\coloneqq N(G)^{\bot}$%
\footnote{We use $H_{1}(A)$ also as a notation for the graph space of some
continuously invertible operator $A$ endowed with the norm $\left|A\cdot\right|$.
Similarly, we write $H_{-1}(A)$ for the respective dual space with
the pivot space $\overline{D(A)}$.%
}. It is possible to show that the range of $G^{*}$ is closed as well.
Thus, the above reasoning also applies to $G^{*}$ in the place of
$G$. \\
Moreover, the operators $B_{G}$ and $B_{G}^{\diamond}$, defined
as in Example \ref{ex:Lax-Milgram}, are unitary transformations from
$H_{1}(\left|B_{G}\right|)$ to $H_{0}(\left|B_{G^{*}}\right|)$ and
from $H_{0}(\left|B_{G^{*}}\right|)$ to $H_{-1}(\left|B_{G}\right|),$
respectively. Moreover, note that $B_{G}^{\diamond}$ is the continuous
extension of $B_{G^{*}}(=B_{G}^{*})$. The abstract result asserting
a solution theory for homogeneous elliptic boundary value problems
reads as follows.
\begin{thm}[{\cite[Theorem 3.1.1]{TrostorffWaurick2012_Elliptic}}]
\label{thm: STHE} Let $H_{1},H_{2}$ be Hilbert spaces and let $G:D(G)\subseteq H_{1}\to H_{2}$
be a densely defined closed linear operator, such that $R(G)\subseteq H_{2}$
is closed. Let $a\subseteq R(G)\oplus R(G)$ such that $a^{-1}:R(G)\to R(G)$
is Lipschitz-continuous. Then for all $f\in H_{-1}(\left|B_{G}\right|)$
there exists a unique $u\in H_{1}(\left|B_{G}\right|)$ such that
the following inclusion holds 
\[
(u,f)\in G^{\diamond}aG\coloneqq\left\{ (x,z)\in H_{1}(|G|+\ii)\times H_{-1}(|G|+\ii)\,|\,\bigvee_{y\in H_{0}(|G|+\ii)}(Gx,y)\in a\wedge z=G^{\diamond}y\right\} .
\]
Moreover, the solution $u$ depends Lipschitz-continuously on the
right-hand side with Lipschitz constant $|a^{-1}|_{\mathrm{Lip}}$
denoting the smallest Lipschitz-constant of $a^{-1}$. \\
In other words, the relation $(B_{G}^{\diamond}aB_{G})^{-1}\subseteq H_{-1}(\left|B_{G}\right|)\oplus H_{1}(\left|B_{G}\right|)$
defines a Lipschitz-continuous mapping with $\left|(B_{G}^{\diamond}aB_{G})^{-1}\right|_{\textnormal{Lip}}=|a^{-1}|_{\mathrm{Lip}}$. \end{thm}
\begin{proof}
It is easy to see that $(u,f)\in G^{\diamond}aG$ for $u\in H_1(|B_G|)$ and $f\in H_{-1}(|B_G|)$ if and only if $(u,f)\in B_{G}^{\diamond}aB_{G}$.
Hence, the assertion follows from $(B_{G}^{\diamond}aB_{G})^{-1}=B_{G}^{-1}a^{-1}(B_{G}^{\diamond})^{-1}$,
the unitarity of $B_{G}$ and $B_{G}^{\diamond}$, and the fact that
$a^{-1}$ is Lipschitz-continuous on $R(G)$. 
\end{proof}
To illustrate the latter result, we give an example.
\begin{defn}
\label{def:div_grad}Let $\Omega\subseteq\mathbb{R}^{n}$ open. We
define 
\begin{align*}
\widetilde{\dive}_{c}\colon\, C_{\infty,c}(\Omega)^{n}\subseteq\bigoplus_{k=1}^{n}L^{2}(\Omega) & \to L^{2}(\Omega)\\
\phi=(\phi_{1},\ldots,\phi_{n}) & \mapsto\sum_{k=1}^{n}\partial_{k}\phi_{k},
\end{align*}
 where $\partial_{k}$ denotes the derivative with respect to the
$k$'th variable ($k\in\{1,\ldots,n\}$) and $C_{\infty,c}(\Omega)$
is the space of arbitrarily differentiable functions with compact
support in $\Omega$. Furthermore, define 
\begin{align*}
\widetilde{\grad}_{c}\colon\, C_{\infty,c}(\Omega)\subseteq L^{2}(\Omega) & \to\bigoplus_{k=1}^{n}L^{2}(\Omega)\\
\phi & \mapsto(\partial_{1}\phi,\ldots,\partial_{n}\phi).
\end{align*}
Integration by parts gives $\widetilde{\dive}_{c}\subseteq-\left(\widetilde{\grad}_{c}\right)^{\ast}$
and consequently $\widetilde{\grad}_{c}\subseteq-\left(\widetilde{\dive}_{c}\right)^{\ast}.$
We set $\dive\coloneqq-\left(\widetilde{\grad}_{c}\right)^{*}$, $\grad\coloneqq-\left(\widetilde{\dive}_{c}\right)^{*}$,
$\dive_{c}\coloneqq-\grad^{*}$ and $\grad_{c}\coloneqq-\dive^{*}$.\\
In the particular case $n=1$ we set $\partial_{1,c}\coloneqq\grad_{c}=\dive_{c}$
and $\partial_{1}\coloneqq\grad=\dive.$ 
\end{defn}
With the latter operators, in order to apply the solution theory above,
one needs to impose certain geometric conditions on the open set $\Omega$.
Indeed, the above theorem applies to $\grad$ or $\grad_{c}$ in the
place of $G$ for the homogeneous Neumann and Dirichlet case, respectively
(in this case $G^{\diamond}$ is then the canonical extension of $-\dive_{c}$
and $-\dive$, respectively). The only thing that has to be guaranteed
is the closedness of the range of $\grad$ ($\grad_{c}$, resp.). This
in turn can be warranted, e.g., if $\Omega$ is bounded, connected
and satisfies the segment property for the Neumann case or if $\Omega$
is bounded in one direction for the Dirichlet case. In both cases,
one can prove the Poincar{\'e} inequality, which especially implies
the closedness of the corresponding ranges (see e.g. \cite[Satz 7.6, p. 120]{Wloka1982}
and \cite[Theorem 3.8, p. 24]{agmon2010lectures} for the Poincar{\'e}
inequality for the Dirichlet case and Rellich's theorem for the Neumann
case, respectively. Note that if the domain of the gradient endowed
with the graph norm is compactly embedded into the underlying space,
a Poincar{\'e} type estimate can be derived by a contradiction argument.).

In the remainder of this section, we discuss an elliptic-type problem
in one dimension with indefinite coefficients. A similar result in
two dimensions can be found in \cite{Dauge1999}.
\begin{example}
\label{ex:elliptic}%
Let $\Omega\coloneqq\left[-\frac{1}{2},\frac{1}{2}\right]$ and set
\[
\tilde{a}(x)\coloneqq\begin{cases}
\alpha & x\geq0,\\
\beta & x<0
\end{cases}\quad\left(x\in\left[-\frac{1}{2},\frac{1}{2}\right]\right)
\]
for some $\alpha,\beta\in\mathbb{R}\setminus\{0\}$. We denote the
corresponding multiplication-operator on $L^{2}\left(\left[-\frac{1}{2},\frac{1}{2}\right]\right)$
by $\tilde{a}(\mathrm{m})$ and consider the following equation in
divergence-form 
\begin{equation}
-\partial_{1}\tilde{a}(\mathrm{m})\partial_{1,c}u=f\label{eq:elliptic}
\end{equation}
for some $f\in H_{-1}(|\partial_{1,c}|).$ Clearly, $R(\partial_{1,c})$
is closed and we denote the canonical embedding from $R(\partial_{1,c})$
into $L^{2}\left(\left[-\frac{1}{2},\frac{1}{2}\right]\right)$ by
$\iota_{R(\partial_{1,c})}.$ Then $\iota_{R(\partial_{1,c})}^{\ast}:L^{2}\left(\left[-\frac{1}{2},\frac{1}{2}\right]\right)\to R(\partial_{1,c})$
is the orthogonal projection onto $R(\partial_{1,c})$ (see e.g. \cite[Lemma 3.2]{Picard2013_fractional})
and we can rewrite (\ref{eq:elliptic}) as 
\[
-\partial_{1}\iota_{R(\partial_{1,c})} \iota_{R(\partial_{1,c})}^{*} \tilde{a}(\mathrm{m})\iota_{R(\partial_{1,c})}\iota_{R(\partial_{1,c})}^{*} \partial_{1,c}u=f,
\]
where we have used that $\partial_1$ vanishes on $R(\partial_{1,c})^\bot$. Hence, we are in the setting of Theorem \ref{thm: STHE}, where $a=\iota_{R(\partial_{1,c})}^{\ast}\tilde{a}(\mathrm{m})\iota_{R(\partial_{1,c})}:R(\partial_{1,c})\to R(\partial_{1,c}).$
The only thing we have to show is that $a^{-1}$ defines a Lipschitz-continuous
mapping on $R(\partial_{1,c}).$ As $R(\partial_{1,c})$ is closed,
it follows that 
\[
R(\partial_{1,c})=N(\partial_{1})^{\bot}=\left\{ \mathbf{1}\right\} ^{\bot},
\]
where $\mathbf{1}$ denotes the constant function $\mathbf{1}(x)=1$
for $x\in\left[-\frac{1}{2},\frac{1}{2}\right].$ To show that $a$
is invertible, we have to solve the problem 
\[
a\varphi=\psi
\]
for given $\psi\in\left\{ \mathbf{1}\right\} ^{\bot}.$ The latter
can be written as 
\[
\psi=a\varphi=\tilde{a}(\mathrm{m})\varphi-\langle\mathbf{1}|\tilde{a}(\mathrm{m})\varphi\rangle\mathbf{1}.
\]
As $\tilde{a}(\mathrm{m})$ is continuously invertible (since $\alpha,\beta\ne0$)
we derive 
\[
\varphi=\tilde{a}(\mathrm{m})^{-1}\psi+\langle\mathbf{1}|\tilde{a}(\mathrm{m})\varphi\rangle\tilde{a}(\mathrm{m})^{-1}\mathbf{1}.
\]
Since $\varphi\in\{\mathbf{1}\}^{\bot}$ we obtain 
\[
0=\langle\mathbf{1}|\tilde{a}(\mathrm{m})^{-1}\psi\rangle+\langle\mathbf{1}|\tilde{a}(\mathrm{m})\varphi\rangle\langle\mathbf{1}|\tilde{a}(\mathrm{m})^{-1}\mathbf{1}\rangle,
\]
yielding 
\[
\langle\mathbf{1}|\tilde{a}(\mathrm{m})\varphi\rangle=-\frac{\langle\mathbf{1}|\tilde{a}(\mathrm{m})^{-1}\psi\rangle}{\langle\mathbf{1}|\tilde{a}(\mathrm{m})^{-1}\mathbf{1}\rangle},
\]
providing that $\langle\mathbf{1}|\tilde{a}(\mathrm{m})^{-1}\mathbf{1}\rangle\ne0.$
The latter holds if and only if $\alpha\ne-\beta.$ Thus, assuming
that $\alpha\ne-\beta$ we get 
\[
a^{-1}\psi=\tilde{a}(\mathrm{m})^{-1}\psi-\frac{\langle\mathbf{1}|\tilde{a}(\mathrm{m})^{-1}\psi\rangle}{\langle\mathbf{1}|\tilde{a}(\mathrm{m})^{-1}\mathbf{1}\rangle}\tilde{a}(\mathrm{m})^{-1}\mathbf{1},
\]
which clearly defines a Lipschitz-continuous mapping. Summarizing,
if $\alpha,\beta\ne0$ and $\alpha\ne-\beta,$ then for each $f\in H_{-1}(|\partial_{1,c}|)$
there exists a unique $u\in H_{1}(|\partial_{1,c}|)$ satisfying (\ref{eq:elliptic}). \end{example}
\begin{rem} (a)
If (\ref{eq:elliptic}) is replaced by the problem with homogeneous
Neumann-boundary conditions, then the constraint $\alpha\ne-\beta$
can be dropped, since in this case $R(\partial_{1})=N(\partial_{1,c})^{\bot}=L_{2}\left(\left[-\frac{1}{2},\frac{1}{2}\right]\right)$,
and thus, $a$ is invertible if $\tilde{a}(\mathrm{m})$ is invertible.

(b) Of course in view of Theorem \ref{thm: STHE}, the coefficient $a$ in Example \ref{ex:elliptic} may also be induced by a relation such that its inverse relation is a (nonlinear) Lipschitz continuous mapping in $R(\partial_{1,c})$.
\end{rem}

\section{Linear Evolutionary Equations and Strict Positivity\label{sec:Linear-Evolutionary-Equations}}

In this section we shall discuss equations of the form
\begin{equation}
\left(\partial_{0,\nu}\mathcal{M}+\mathcal{A}\right)U=F,\label{eq:gen_evo}
\end{equation}
where $\partial_{0,\nu}$ is the \emph{time-derivative operator} to
be introduced and specified below, $\mathcal{M}$ and $\mathcal{A}$
are linear operators, the former -- the\emph{ material law operator}
-- being bounded, and the latter being possibly unbounded. The task
is in finding the unknown $U$ for a given right hand side $F$. This
is done by showing that both (the closure of) $\left(\partial_{0,\nu}\mathcal{M}+\mathcal{A}\right)$
and $\left(\partial_{0,\nu}\mathcal{M}+\mathcal{A}\right)^{*}$ are
strictly accretive operators in a suitable Hilbert space and then
using Corollary \ref{lin:op:theo}. We will comment on the specific
assumptions on $\mathcal{M}$ and $\mathcal{A}$ in the subsequent
sections as well as on the rigorous (Hilbert space) framework the
equation (\ref{eq:gen_evo}) should be considered in. Before we discuss
the abstract theory, we give four elementary guiding examples which
shall lead us through the development of the abstract theory.

\subsection{Guiding examples\label{sub:Leading-Examples}}

\subsubsection*{Ordinary (integro-)differential equations}

We shall consider the following easy form of an ordinary differential
equation. For a given right hand side $f\in C_{c}(\mathbb{R}\times\mathbb{R})$,
i.e. $f$ is a continuous function on $\mathbb{R}\times\mathbb{R}$
with compact support, and a coefficient $a\in L^{\infty}(\mathbb{R})$
we consider the problem of finding $u$ in a suitable Hilbert space
such that for (a.e.) $(t,x)\in\mathbb{R}\times\mathbb{R}$ the equation
\[
u(\cdot,x)'(t)+a(x)u(t,x)=f(t,x)
\]
holds. We will also have the opportunity to consider an integro-differential
equation of the form
\[
u(\cdot,x)'(t)+a(x)u(t,x)+\int_{-\infty}^{t}k(t-s)u(s,x)ds=f(t,x).
\]
for a suitable kernel $k\colon\mathbb{R}\to\mathbb{R}$.

\subsubsection*{The heat equation}

The heat $\theta$ in a given body $\Omega\subseteq\mathbb{R}^{n}$
can be described by the conservation law
\[
\theta(\cdot,x)'(t)+\dive q(t,x)=f(t,x)\quad((t,x)\in\mathbb{R}\times\Omega\mbox{ a.e.}),
\]
where $f$ is a given heat source and $q$ is the heat flux given
by Fourier's law as follows
\[
q(t,x)=-k(x)\grad\theta(t,x)\quad((t,x)\in\mathbb{R}\times\Omega\mbox{ a.e.}),
\]
where $k$ is a certain coefficient matrix describing the specific
conductivities of the underlying material varying over $\Omega.$
Here $\dive$ and $\grad$ are the canonical extensions of the spatial
operators $\dive$ and $\grad$ defined in the previous section (Definition
\ref{def:div_grad}) to the space $L^{2}(\mathbb{R}\times\Omega,\mu\otimes\lambda)=L^{2}(\mathbb{R},\mu)\otimes L^{2}(\Omega,\lambda)$,
where $\mu$ is a Borel-measure on $\mathbb{R}$ and $\lambda$ denotes
the $n$-dimensional Lebesgue-measure, i.e. 
\begin{align*}
\dive q(t,x) & =\left(\dive q(t,\cdot)\right)(x)\\
\grad\theta(t,x) & =\left(\grad\theta(t,\cdot)\right)(x)\quad((t,x)\in\mathbb{R}\times\Omega\mbox{ a.e.}).
\end{align*}
In a block operator matrix form, recalling that $\partial_{0}$ denotes
the derivative with respect to time, we get 
\[
\left(\partial_{0}\left(\begin{array}{cc}
1 & 0\\
0 & 0
\end{array}\right)+\left(\begin{array}{cc}
0 & 0\\
0 & k^{-1}
\end{array}\right)+\left(\begin{array}{cc}
0 & \dive\\
\grad & 0
\end{array}\right)\right)\left(\begin{array}{c}
\theta\\
q
\end{array}\right)=\left(\begin{array}{c}
f\\
0
\end{array}\right),
\]
assuming that the coefficient $k$ is invertible. Imposing suitable
assumptions on data and coefficients and boundary conditions for the
operators $\dive$ and/or $\grad$ will be seen to warrant well-posedness
of the resulting system. We will comment on the precise details in
our discussion of abstract well-posedness results.

\subsubsection*{The elastic equations}

In the theory of elasticity, the open set $\Omega\subseteq\mathbb{R}^{n}$,
being the underlying domain, models a body in its non-deformed state
(of course, in applications $n=3$). The displacement field $u$ assigns
to each space-time coordinate $(t,x)\in\mathbb{R}\times\Omega$ direction
and size of the displacement at time $t$ of the material point at
position $x$. The displacement field $u$ satisfies the balance of
momentum equation (again writing $\partial_{0}$ for the time-derivative)
\[
\partial_{0}^{2}u-\Dive\sigma=f,
\]
with $f$ being an external forcing term, $\sigma$ being the (symmetric)
stress tensor and $\Dive$ being the row-wise (distributional) divergence
acting on suitable elements in the space $H_{\textnormal{sym}}(\Omega)$
of symmetric $n\times n$ matrices of $L^{2}(\Omega)$-functions as
an operator from $H_{\textnormal{sym}}(\Omega)$ to $L^{2}(\Omega)^{n}$
with maximal domain%
\footnote{The precise definition will be given later.%
}. Endowing $H_{\textnormal{sym}}(\Omega)$ with the Frobenius inner
product, we get that the negative adjoint of $\Dive$ is the symmetrized
gradient or strain tensor given by 
\[
\epsilon(u)\coloneqq\Grad u\coloneqq\frac{1}{2}\left(\partial_{i}u_{j}+\partial_{j}u_{i}\right)_{i,j}
\]
with Dirichlet boundary conditions as induced constraint on the domain.
Neumann boundary conditions can be modeled similarly. The stress tensor
satisfies the constitutive relation involving the elasticity tensor
$C$ in the way that
\[
\sigma=C\epsilon(u).
\]
Introducing the displacement velocity $v\coloneqq\partial_{0}u$ as
a new unknown, we write the elastic equations formally as the following
block operator matrix equation
\[
\left(\partial_{0}\left(\begin{array}{cc}
1 & 0\\
0 & C^{-1}
\end{array}\right)-\left(\begin{array}{cc}
0 & \Dive\\
\Grad & 0
\end{array}\right)\right)\left(\begin{array}{c}
v\\
\sigma
\end{array}\right)=\left(\begin{array}{c}
f\\
0
\end{array}\right),
\]
where we assume that $C$ is invertible.

\subsubsection*{The Maxwell's equations}

The equations for electro-magnetic theory describe evolution of the
electro-magnetic field $(E,H)$ in a $3$-dimensional open set $\Omega$.
As Gauss' law can be incorporated by a suitable choice of initial
data, we think of Maxwell's equations as Faraday's law of induction
(the Maxwell-Faraday equation), which reads as 
\[
\partial_{0}B+\curl_{c}E=0,
\]
where $\curl_{c}$ denotes the (distributional) $\curl$ operator
in $L^{2}(\Omega)^{3}$ with the electric boundary condition of vanishing
tangential components. The magnetic field $B$ satisfies the constitutive
equation
\[
B=\mu H,
\]
where $\mu$ is the magnetic permeability. Faraday's law is complemented
by Ampere's law
\[
\partial_{0}D+J_{c}-\curl H=J_{0}
\]
for $J_{0}$, $D$, $J_{c}$ being the external currents, the electric
displacement and the charge, respectively. The latter two quantities
satisfy the two equations
\begin{align*}
D & =\epsilon E,\text{ and }\\
J_{c} & =\sigma E.
\end{align*}
The former is a constitutive equation involving the dielectricity
$\epsilon$ and the latter is Ohm's law with conductivity $\sigma$.
Plugging the constitutive relations and Ohm's law into Faraday's law
of induction and Ampere's law and arranging them in a block operator
matrix equation, we arrive at 
\[
\left(\partial_{0}\left(\begin{array}{cc}
\epsilon & 0\\
0 & \mu
\end{array}\right)+\left(\begin{array}{cc}
\sigma & 0\\
0 & 0
\end{array}\right)+\left(\begin{array}{cc}
0 & -\curl\\
\curl_{c} & 0
\end{array}\right)\right)\left(\begin{array}{c}
E\\
H
\end{array}\right)=\left(\begin{array}{c}
J_{0}\\
0
\end{array}\right).
\]

Having these examples in mind, we develop the abstract theory a bit
further and discuss the time-derivative operator in the next section.
After having done so, we aim at giving a unified solution theory for
all of the latter examples. In fact we show that all of these equations
are of the general form (\ref{eq:gen_evo}).

\subsection{The time-derivative\label{sub:The-Time-Derivative}}

When considering evolutionary equations, we need a distinguished direction
of time. Anticipating this fact, we define a time-derivative $\partial_{0,\nu}$
as an operator in a weighted Hilbert space. 

Beforehand, we recall some well-known facts from the (time-)derivative
in the unweighted space $L^{2}(\mathbb{R}).$ We denote the Sobolev
space of $L^{2}(\mathbb{R})$-functions $f$ with distributional derivative
$f'$ representable as a $L^{2}(\mathbb{R})$-function by $H_{1}(\mathbb{R})$.
Then, the operator 
\[
\partial\colon H_{1}(\mathbb{R})\subseteq L^{2}(\mathbb{R})\to L^{2}(\mathbb{R}),f\mapsto f'
\]
is skew-selfadjoint. Indeed, using that the space $C_{\infty,c}(\mathbb{R})$
is a core for $\partial$, we immediately verify with integration
by parts that $\partial$ is skew-symmetric. With the help of some
elementary computations it is possible to show that the range of both
the operators $\partial+1$ and $\partial-1$ contains $C_{\infty,c}(\mathbb{R})$.
The closedness of $\partial$ thus implies the skew-selfadjointness
of $\partial$ (see also \cite[Example 3.14]{kato1995perturbation}).
Moreover, it is well-known that $\partial$ admits an explicit spectral
representation given by the Fourier transform $\mathcal{F}$, being
the unitary extension on $L^{2}(\mathbb{R})$ of the mapping given
by
\[
\mathcal{F}\phi(\xi)\coloneqq\frac{1}{\sqrt{2\pi}}\int_{\mathbb{R}}e^{-\ii\xi x}\phi(x)dx\quad(\xi\in\mathbb{R})
\]
for $\phi\in C_{\infty,c}(\mathbb{R})$. Denoting by $\mathrm{m}\colon D(\mathrm{m})\subseteq L^{2}(\mathbb{R})\to L^{2}(\mathbb{R})$
the multiplication-by-argument operator given by 
\[
\mathrm{m}f\coloneqq(\xi\mapsto\xi f(\xi))
\]
for $f$ belonging to the maximal domain $D(\mathrm{m})$ of $\mathrm{m}$,
we find the following unitary equivalence of differentiation and multiplication
(see \cite[Volume 1, p. 161-163]{Akhiezer_Glazman_1993}):
\[
\partial=\mathcal{F}^{*}\ii\mathrm{m}\mathcal{F},
\]
which, due to the selfadjointness of $\mathrm{m}$, confirms the skew-selfadjointness
of $\partial.$ 

As mentioned above, in evolutionary processes there is a particular
bias for the forward time direction. As $L^{2}(\mathbb{R})$ has no
such bias, we choose a suitable weight, which serves to express this
bias. For $\nu\in\mathbb{R}$ we let $L_{\nu}^{2}(\mathbb{R})\coloneqq\left\{ f\in L^{2,\mathrm{loc}}\left(\mathbb{R}\right)|\:\int_{\mathbb{R}}\left|f\left(t\right)\right|^{2}\exp\left(-2\nu t\right)\: dt<\infty\right\} $,
the Hilbert space of (equivalence classes of) functions with $\left(t\mapsto\exp\left(-\nu t\right)\: f\left(t\right)\right)\in L^{2}\left(\mathbb{R}\right)$
(with the obvious norm). In particular, $L_{0}^{2}(\mathbb{R})=L^{2}(\mathbb{R})$.
Moreover, it is easily seen that the mapping
\[
e^{-\nu\mathrm{m}}\colon L_{\nu}^{2}(\mathbb{R})\to L^{2}(\mathbb{R}),f\mapsto\left(t\mapsto e^{-\nu t}f(t)\right)
\]
defines a unitary mapping. To carry differentiation
over to the exponentially weighted $L^{2}$-spaces, we observe that
for $\phi\in C_{\infty,c}(\mathbb{R})$ we have
\begin{align*}
\left(e^{-\nu\mathrm{m}}\right)^{-1}\partial e^{-\nu\mathrm{m}}\phi & =\left(e^{-\nu\mathrm{m}}\right)^{-1}(-\nu e^{-\nu\mathrm{m}}\phi+e^{-\nu\mathrm{m}}\phi')\\
 & =-\nu\phi+\phi',
\end{align*}
or 
\[
\left(e^{-\nu\mathrm{m}}\right)^{-1}\left(\partial+\nu\right)e^{-\nu\mathrm{m}}\phi=\phi'.
\]
Defining $\partial_{0,\nu}\coloneqq\left(e^{-\nu\mathrm{m}}\right)^{-1}\left(\partial+\nu\right)e^{-\nu\mathrm{m}}=\left(e^{-\nu\mathrm{m}}\right)^{-1}\partial e^{-\nu\mathrm{m}}+\nu$,
we read off that $\partial_{0,\nu}$ is a realization of the (distributional)
derivative operator in the weighted space $L_{\nu}^{2}(\mathbb{R})$.
Moreover, we see that $\partial_{0,\nu}$ is a normal operator, i.e.
$\partial_{0,\nu}$ and $\partial_{0,\nu}^{\ast}$ commute. In particular,
$\Re\partial_{0,\nu}=\frac{1}{2}\overline{\left(\partial_{0,\nu}+\partial_{0,\nu}^{\ast}\right)}=\nu$,
due to the skew-selfadjointness of $\partial$. This also shows that
$\partial_{0,\nu}$ is continuously invertible if $\nu\neq0$. Indeed,
we find the following explicit formula for the inverse: For $t\in\mathbb{R},$
$\nu\in\mathbb{R}\setminus\{0\}$, $f\in L_{\nu}^{2}(\mathbb{R})$
we have
\[
\partial_{0,\nu}^{-1}f(t)=\begin{cases}
\int_{-\infty}^{t}f(\tau)d\tau, & \nu>0,\\
-\int_{t}^{\infty}f(\tau)d\tau, & \nu<0.
\end{cases}
\]
For positive $\nu,$ the latter formula also shows that the values
of $\partial_{0,\nu}^{-1}f$ at time $t$ only depend on the values
of $f$ up to time $t$. This is the nucleus of the notion of causality,
where the sign of $\nu$ switches the forward and backward time direction.
Later, we will comment on this in more detail. 

The spectral representation of $\partial$ induces a spectral representation
for $\partial_{0,\nu}$. Indeed, introducing the Fourier-Laplace transformation%
\footnote{For the classical Fourier-Laplace transformation, also known as two-sided
Laplace transformation, this unitary character is rarely invoked.
In fact, it is mostly considered as an integral expression acting
on suitably integrable functions, whereas the unitary Fourier-Laplace
transformation is continuously extended (thus, of course, including
also some non-integrable functions). %
} $\mathcal{L}_{\nu}\coloneqq\mathcal{F}e^{-\nu\mathrm{m}}:L_{\nu}^{2}(\mathbb{R})\to L^{2}(\mathbb{R})$
for $\nu\in\mathbb{R},$ which itself is a unitary operator as a composition
of unitary operators, we get that
\[
\partial_{0,\nu}=\mathcal{L}_{\nu}^{*}(\ii\mathrm{m}+\nu)\mathcal{L}_{\nu}.
\]
This shows that for $\nu\not=0$ the normal operator $\partial_{0,\nu}^{-1}$
is unitarily equivalent to the multiplication operator%
\footnote{In the sense of the induced functional calculus we have
\[
\left(\ii\mathrm{m}+\nu\right)^{-1}=\frac{1}{\ii\mathrm{m}+\nu}.
\]
} $\left(\ii\mathrm{m}+\nu\right)^{-1}$ with spectrum $\sigma(\partial_{0,\nu}^{-1})=\partial B(\frac{1}{2\nu},\frac{1}{2\nu})$,
where $\nu>0$ and $B(a,r)\coloneqq\{z\in\mathbb{C}||z-a|<r\},$ $a\in\mathbb{C},$
$r>0$. We will use this fact in the next section by establishing
a functional calculus for $\partial_{0,\nu}^{-1}$. Henceforth, if
not otherwise stated, the parameter $\nu$ will always be a positive
real number.

\subsection{The autonomous case \label{sub:evo_eq_auto}}

In order to formulate the Hilbert space framework of (\ref{eq:gen_evo})
properly, we need to consider the space of $H$-valued $L_{\nu}^{2}$-functions.
Consequently, we need to invoke the (canonical) extension of $\partial_{0,\nu}$
to $L_{\nu}^{2}(\mathbb{R},H)$ for some Hilbert space $H$. For convenience,
we re-use the notation for the respective extension since there is
no risk of confusion. Moreover, we shall do so for the Fourier-Laplace
transform $\mathcal{L}_{\nu}$ which is then understood as a unitary
operator from $L_{\nu}^{2}(\mathbb{R},H)$ onto $L^{2}(\mathbb{R},H)$. 

In this section we treat a particular example for the choice of the
operators $\mathcal{M}$ and $\mathcal{A}$, namely that of autonomous
operators. For this we need to introduce the operator of time translation:
For $h\in\mathbb{R}$, $\nu\in\mathbb{R}$ we define the \emph{time
translation operator} $\tau_{h}$ on $L_{\nu}^{2}(\mathbb{R})$ by
$\tau_{h}f\coloneqq f(\cdot+h)$. Again, we identify $\tau_{h}$ with
its extension to the $H$-valued case.
\begin{defn}
We call an operator $B\colon D(B)\subseteq L_{\nu}^{2}(\mathbb{R},H)\to L_{\nu}^{2}(\mathbb{R},H)$
\emph{autonomous} or \emph{time translation invariant}, if it commutes
with $\tau_{h}$ for all $h\in\mathbb{R},$ i.e.,
\[
\tau_{h}B\subseteq B\tau_{h}\quad(h\in\mathbb{R}).
\]

\end{defn}
For an evolution to take place, a physically reasonable property is
\emph{causality}. Denoting by $\1_{_{_{_{\mathbb{R}_{\leq a}}}}}(\mathrm{m}_{0})$
the multiplication operator on $L_{\nu}^{2}(\mathbb{R},H)$ given
by $\left(\1_{_{_{_{\mathbb{R}_{\leq a}}}}}(\mathrm{m}_{0})u\right)(t)\coloneqq\1_{_{_{_{\mathbb{R}_{\leq a}}}}}(t)u(t)$
for $u\in L_{\nu}^{2}(\mathbb{R},H)$ and $t\in\mathbb{R}$, the definition of causality reads as follows. \\

\begin{defn}
We call a closed mapping $M\colon D(M)\subseteq L_{\nu}^{2}(\mathbb{R},H)\to L_{\nu}^{2}(\mathbb{R},H)$
\emph{causal,} if for all $f,g\in D(M)$ and $a\in\mathbb{R}$ we
have
\begin{equation}
\1_{\mathbb{R}_{\leq a}}(\mathrm{m}_{0})f=\1_{\mathbb{R}_{\leq a}}(\mathrm{m}_{0})g\Rightarrow\1_{\mathbb{R}_{\leq a}}(\mathrm{m}_{0})M(f)=\1_{\mathbb{R}_{\leq a}}(\mathrm{m}_{0})M(g).\label{eq:causality}
\end{equation}
\end{defn}
\begin{rem}$\,$
\begin{enumerate}[(a)]
 \item Property (\ref{eq:causality}) reflects the idea that the
``future behavior does not influence the past'', which may be taken
as the meaning of \emph{causality}. 

\item If, in addition, $M$ in the latter definition is linear, then
$M$ is causal if and only if for all $u\in D(M)$ 
\[
\1_{_{_{_{\mathbb{R}_{\leq a}}}}}(\mathrm{m}_{0})u=0\Rightarrow\1_{_{_{_{\mathbb{R}_{\leq a}}}}}(\mathrm{m}_{0})Mu=0\qquad(a\in\mathbb{R}).
\]

\item For continuous mappings $M$ with full domain $L_{\nu}^{2}(\mathbb{R},H)$,
causality can also equivalently be expressed by the equation
\[
\1_{_{_{_{\mathbb{R}_{\leq a}}}}}(\mathrm{m}_{0})M=\1_{_{_{_{\mathbb{R}_{\leq a}}}}}(\mathrm{m}_{0})M\1_{_{{\mathbb{R}_{\leq a}}}}(\mathrm{m}_{0})
\]
holding for all $a\in\mathbb{R}$. If, in addition, $M$ is autonomous
then this condition is in turn equivalent to 
\[
\1_{_{{\mathbb{R}_{\leq a}}}}(\mathrm{m}_{0})M=\1_{_{{\mathbb{R}_{\leq a}}}}(\mathrm{m}_{0})M\1_{_{{\mathbb{R}_{\leq a}}}}(\mathrm{m}_{0})
\]
for some $a\in\mathbb{R}$, e.g. $a=0$.

\item The respective concept for causality for closable mappings is
a bit more involved, see \cite{Waurick2013_causality}.
\end{enumerate}
\end{rem}
In order to motivate the problem class discussed in this section a
bit further we state the following well-known representation theorem:
\begin{thm}[{see e.g. \cite[Theorem 2.3]{Weiss1991}, \cite[Theorem 9.1]{Thomas1997}
or \cite{MR0083623}}]
\label{thm:repr}Let $H$ be a Hilbert space, $\mathcal{M}\colon L_{\nu}^{2}(\mathbb{R},H)\to L_{\nu}^{2}(\mathbb{R},H)$
bounded, linear, causal and autonomous. Then there exists a unique
$\tilde{M}\colon\left\{ z\in\mathbb{C}|\Re z>\nu\right\} \to L(H)$
bounded and analytic, such that for all $u\in L_{\nu}^{2}(\mathbb{R},H)$
we have 
\[
\mathcal{M}u=\mathcal{L}_{\nu}^{*}\tilde{M}(\ii\mathrm{m}+\nu)\mathcal{L}_{\nu}u,
\]
where $\left(\tilde{M}(\ii\mathrm{m}+\nu)\phi\right)(\xi)\coloneqq\tilde{M}(\ii\xi+\nu)\phi(\xi)$
for $\phi\in L^{2}(\mathbb{R},H)$, $\xi\in\mathbb{R}$. 
\end{thm}
This theorem tells us that the class of bounded, linear, causal, autonomous
operators is described by bounded and analytic functions of $\partial_{0,\nu}$
or equivalently of $\partial_{0,\nu}^{-1}$. Thus, we are led to introduce
the Hardy space for some open $E\subseteq\mathbb{C}$ and Hilbert
space $H$:
\[
\mathcal{H}^{\infty}(E,L(H))\coloneqq\{M\colon E\to L(H)\,|\, M\text{ bounded, analytic}\}.
\]
Clearly, $\mathcal{H}^{\infty}(E,L(H))$ (or briefly $\mathcal{H}^{\infty}$
if $E$ and $H$ are clear from the context) is a Banach space with
norm
\[
\mathcal{H}^{\infty}\ni M\mapsto\|M\|_{\infty}\coloneqq\sup\{\|M(z)\||z\in E\}.
\]
In the particular case of $M\in\mathcal{H}^{\infty}(E,L(H))$ with
$E=B(r,r)$ for some $r>0$, we define for $\nu>\frac{1}{2r}$ 
\begin{align*}
M(\partial_{0,\nu}^{-1})\colon L_{\nu}^{2}(\mathbb{R},H) & \to L_{\nu}^{2}(\mathbb{R},H),\\
\phi & \mapsto\mathcal{L}_{\nu}^{\ast}M\left(\frac{1}{\ii\mathrm{m}+\nu}\right)\mathcal{L}_{\nu}\phi.
\end{align*}
Here $M\left(\frac{1}{\ii\mathrm{m}+\nu}\right)\in L(L^{2}(\mathbb{R},H))$
is given by 
\[
\left(M\left(\frac{1}{\ii\mathrm{m}+\nu}\right)w\right)(t)\coloneqq M\left(\frac{1}{\ii t+\nu}\right)w(t)
\]
for $w\in L^{2}(\mathbb{R},H)$ and $t\in\mathbb{R}.$ Note that $\sup_{z\in B\left(\frac{1}{2\nu},\frac{1}{2\nu}\right)}\|M(z)\|=\|M(\partial_{0,\nu}^{-1})\|_{L(H_{\nu,0}(\mathbb{R},H))}$
according to \cite[Theorem 9.1]{Thomas1997}.

Our first theorem asserting a solution theory for evolutionary equations
reads as follows.
\begin{thm}[{\cite[Solution Theory]{Picard2009}, \cite[Chapter 6]{Picard_McGhee}}]
\label{thm:sol_theo_skew}Let $\nu>0$, $r>\frac{1}{2\nu}$
and $M\in\mathcal{H}^{\infty}(B(r,r),L(H))$, $A\colon D(A)\subseteq H\to H$.
Assume that
\[
A\text{ is skew-selfadjoint, and}
\]
\begin{equation}
\bigvee_{c>0}\bigwedge_{z\in B(r,r)}z^{-1}M(z)-c\mbox{ is monotone.}\label{eq:pos_material_law}
\end{equation}
Then $\overline{\partial_{0,\nu}M(\partial_{0,\nu}^{-1})+A}$ is continuously
invertible in $L_{\nu}^{2}(\mathbb{R},H)$. The closure of the inverse
is causal. Moreover, the solution operator is independent of the choice
of $\nu$ in the sense that for $\rho>\nu$ and $f\in L_{\nu}^{2}(\mathbb{R},H)\cap L_{\rho}^{2}(\mathbb{R},H)$
we have that 
\[
\overline{\left(\partial_{0,\nu}M(\partial_{0,\nu}^{-1})+A\right)^{-1}}f=\overline{\left(\partial_{0,\rho}M(\partial_{0,\rho}^{-1})+A\right)^{-1}}f.
\]
\end{thm}
\begin{proof}[Sketch of the proof]
 First, one proves that the operator $\partial_{0,\nu}M(\partial_{0,\nu}^{-1})+A$
is closable and its closure is strictly monotone. The same holds for
its adjoint, which turns out to have the same domain. Thus, the well-posedness
of (\ref{eq:gen_evo}), where $A$ is skew-selfadjoint and $\mathcal{M}=M(\partial_{0,\nu}^{-1})$
follows by Corollary \ref{lin:op:D(A):theo}. The causality of the
solution operator can be shown by a Paley-Wiener type result (see
e.g. \cite[Theorem 19.2]{rudin1987real}).
\end{proof}
According to the latter theorem, the solution operator associated
with an evolutionary problem is independent of the particular choice
of $\nu$. Therefore, we usually will drop the index $\nu$ and write
instead $\partial_{0}$ and $M(\partial_{0}^{-1})$.

Since the positive definiteness condition in (\ref{eq:pos_material_law})
will occur several times, we define $\mathcal{H}^{\infty,c}$ to be
the set of $M\in\mathcal{H}^{\infty}$ satisfying condition (\ref{eq:pos_material_law})
with the constant $c\in\mathbb{R}_{>0}$.

Note that with $A=0$ in Theorem \ref{thm:sol_theo_skew}, ordinary
differential equations are covered. We shall further elaborate this
fact in Subsection \ref{sub:The-closedness-of}. Here,\textcolor{black}{{}
let us illustrate the versatility of this well-posedness result, by
applying the result to several (partial) }differ\textcolor{black}{ential
equations arising in mathematical physics.}
\begin{example}[The heat equation]
\label{ex:heat_eq} Recall the definition of the operators $\grad_{c},\grad,\dive_{c}$
and $\dive$ from Definition \ref{def:div_grad}. The domain of $\grad_{c}$
coincides with the classical Sobolev-space $H_{0}^{1}(\Omega)$, the
space of $L^{2}(\Omega)$-functions with distributional gradients
lying in $L^{2}(\Omega)^{n}$ and having vanishing trace, while the
domain of $\grad$ is $H^{1}(\Omega)$, the Sobolev-space of weakly
differentiable functions in $L^{2}(\Omega)$. Analogously, $v\in D(\dive_{c})$
is a $L^{2}(\Omega)$-vector field, whose distributional divergence
is in $L^{2}(\Omega)$ and satisfies a generalized Neumann-condition
$v\cdot N=0$ on $\partial\Omega,$ where $N$ denotes the outward
unit normal vector field on $\partial\Omega$.%
\footnote{Note that the definition of $\dive_{c}$ still makes sense, even if
the boundary of $\Omega$ is non-smooth and hence, the normal vector
field $N$ does not exist. Thus, for rough domains $\Omega$ the condition
$v\in D(\dive_{c})$ is a suitable substitute for the condition $v\cdot N=0$
on $\partial\Omega$. %
} Recall the conservation of energy equation, given by 
\begin{equation}
\partial_{0}\theta+\dive q=f,\label{eq:conservation_energy}
\end{equation}
where $\theta\in L_{\nu}^{2}(\mathbb{R},L^{2}(\Omega))$ denotes the
(unknown) heat, $q\in L_{\nu}^{2}(\mathbb{R},L^{2}(\Omega)^{3})$
stands for the heat flux and $f\in L_{\nu}^{2}(\mathbb{R},L^{2}(\Omega))$
models an external heat source. This equation is completed by a constitutive
relation, called Fourier's law 
\begin{equation}
q=-k\grad\theta,\label{eq:Fourier}
\end{equation}
where $k\in L^{\infty}(\Omega)^{n\times n}$ denotes the heat conductivity
and is assumed to be selfadjoint-matrix-valued and strictly positive,
i.e. there is some $c>0$ such that $k(x)\geq c$ for almost every
$x\in\Omega.$ Plugging Fourier's law into the conservation of energy
equation, we end up with the familiar form of the heat equation 
\[
\partial_{0}\theta-\dive k\grad\theta=f.
\]
However, it is also possible to write the equations (\ref{eq:conservation_energy})
and (\ref{eq:Fourier}) as the system
\[
\left(\partial_{0}\left(\begin{array}{cc}
1 & 0\\
0 & 0
\end{array}\right)+\left(\begin{array}{cc}
0 & 0\\
0 & k^{-1}
\end{array}\right)+\left(\begin{array}{cc}
0 & \dive\\
\grad & 0
\end{array}\right)\right)\left(\begin{array}{c}
\theta\\
q
\end{array}\right)=\left(\begin{array}{c}
f\\
0
\end{array}\right).
\]
Requiring suitable boundary conditions, say Dirichlet boundary conditions
for the temperature density $\theta$ (i.e. $\theta\in D(\grad_{c})$)
the system becomes an evolutionary equation of the form discussed
in Theorem \ref{thm:sol_theo_skew} with 
\[
M(z)=\left(\begin{array}{cc}
1 & 0\\
0 & 0
\end{array}\right)+z\left(\begin{array}{cc}
0 & 0\\
0 & k^{-1}
\end{array}\right)
\]
and the skew-selfadjoint operator 
\[
A=\left(\begin{array}{cc}
0 & \dive\\
\grad_{c} & 0
\end{array}\right).
\]
By our assumptions on the coefficient $k$, the solvability condition
(\ref{eq:pos_material_law}) can easily be verified for our material
law $M$. Indeed, with $k$ being bounded and strictly positive, the
inverse operator $k^{-1}$ is bounded and strictly positive as well.
Now, since for $z\in B(r,r)$ for some $r>0$ the real part of $z^{-1}$
is bounded below by $\frac{1}{2r}$, we deduce that 
\[
\Re\left(z^{-1}M(z)\right)=\Re\left(z^{-1}\left(\begin{array}{cc}
1 & 0\\
0 & 0
\end{array}\right)+\left(\begin{array}{cc}
0 & 0\\
0 & k^{-1}
\end{array}\right)\right)\geq\left(\begin{array}{cc}
\frac{1}{2r} & 0\\
0 & \frac{1}{c}
\end{array}\right).
\]

\end{example}
$\,$
\begin{example}[Maxwell's equations]
\label{ex:Maxwell} To begin with, we formulate the functional analytic
setting for the operator $\curl$ with and without the electric boundary
condition. For this let $\Omega\subseteq\mathbb{R}^{3}$ be a non-empty
open set and define
\begin{align*}
\tilde{\curl_{c}}\colon C_{\infty,c}(\Omega)^{3}\subseteq L^{2}(\Omega)^{3} & \to L^{2}(\Omega)^{3}\\
\phi & \mapsto\left(\begin{array}{c}
\partial_{2}\phi_{3}-\partial_{3}\phi_{2}\\
\partial_{3}\phi_{1}-\partial_{1}\phi_{3}\\
\partial_{1}\phi_{2}-\partial_{2}\phi_{1}
\end{array}\right).
\end{align*}
Analogously to the previous example, we let $\curl\coloneqq\left(\tilde{\curl_{c}}\right)^{*}$
and $\curl_{c}\coloneqq\curl^{*}$. Now, let $\varepsilon,\mu,\sigma$
be bounded linear operators in $L^{2}(\Omega)^{3}$. We assume that
both $\varepsilon$ and $\mu$ are selfadjoint and strictly positive.
As a consequence, the operator function
\[
z\mapsto\left(\begin{array}{cc}
\varepsilon & 0\\
0 & \mu
\end{array}\right)+z\left(\begin{array}{cc}
\sigma & 0\\
0 & 0
\end{array}\right)
\]
belongs to $\mathcal{H}^{\infty,c}(B(r,r),L(L^{2}(\Omega)^{6}))$
for some $c\in\mathbb{R}_{>0}$, if $r$ is chosen small enough. Thus,
due to the skew-selfadjointness of the operator $\left(\begin{array}{cc}
0 & -\curl\\
\curl_{c} & 0
\end{array}\right)$ in $L^{2}(\Omega)^{6}$, Theorem \ref{thm:sol_theo_skew} applies
to the operator sum
\[
\partial_{0}\left(\begin{array}{cc}
\varepsilon & 0\\
0 & \mu
\end{array}\right)+\left(\begin{array}{cc}
\sigma & 0\\
0 & 0
\end{array}\right)+\left(\begin{array}{cc}
0 & -\curl\\
\curl_{c} & 0
\end{array}\right),
\]
which yields continuous invertibility of the closure of the latter
operator in $L_{\nu}^{2}(\mathbb{R},L^{2}(\Omega)^{3})$ for sufficiently
large $\nu.$ It is noteworthy that the well-posedness theorem of
course also applies to the case, where $\varepsilon=0$ and the real
part of $\sigma$ is strictly positive definite, i.e., to the operator
of the form 
\[
\partial_{0}\left(\begin{array}{cc}
0 & 0\\
0 & \mu
\end{array}\right)+\left(\begin{array}{cc}
\sigma & 0\\
0 & 0
\end{array}\right)+\left(\begin{array}{cc}
0 & -\curl\\
\curl_{c} & 0
\end{array}\right).
\]
In the literature this arises when dealing with the so-called ``eddy
current problem''.
\end{example}
~
\begin{example}[The equations of elasticity and visco-elasticity]
\label{ex:visco_elastic} We begin by introducing the differential
operators involved. Let $\Omega\subseteq\mathbb{R}^{3}$ open. We
consider the Hilbert space $H_{\mathrm{sym}}(\Omega)$ given as the
space of symmetric $L^{2}(\Omega)^{3\times3}$-matrices equipped with
the inner product 
\[
\langle\Phi|\Psi\rangle_{H_{\mathrm{sym}}(\Omega)}\coloneqq\intop_{\Omega}\trace\left(\Phi(x)^{\ast}\Psi(x)\right)dx,
\]
where $\Phi(x)^{\ast}$ denotes the adjoint matrix of $\Phi(x).$
Using this Hilbert space we define the operator $\Grad_{c}$ as the
closure of 
\begin{align*}
\tilde{\Grad_{c}}:C_{\infty,c}(\Omega)^{3}\subseteq L^{2}(\Omega)^{3} & \to H_{\mathrm{sym}}(\Omega)\\
(\phi_{i})_{i\in\{1,2,3\}} & \mapsto\left(\frac{\partial_{i}\phi_{j}+\partial_{j}\phi_{i}}{2}\right)_{i,j\in\{1,2,3\}}
\end{align*}
and likewise we define $\Dive_{c}$ as the closure of 
\begin{align*}
\tilde{\Dive_{c}}:C_{\infty,c}(\Omega)^{3\times3}\cap H_{\mathrm{sym}}(\Omega)\subseteq H_{\mathrm{sym}}(\Omega) & \to L^{2}(\Omega)^{3}\\
\left(\psi_{ij}\right)_{i,j\in\{1,2,3\}} & \mapsto\left(\sum_{j=1}^{3}\partial_{j}\psi_{ij}\right)_{i\in\{1,2,3\}}.
\end{align*}
Integration by parts yields that $\Grad_{c}\subseteq-\left(\Dive_{c}\right)^{\ast}$
as well as $\Dive_{c}\subseteq-\left(\Grad_{c}\right)^{\ast}$ and
we define 
\begin{align*}
\Grad & \coloneqq-\left(\Dive_{c}\right)^{\ast},\\
\Dive & \coloneqq-\left(\Grad_{c}\right)^{\ast}.
\end{align*}
Similar to the case of $\grad$ and $\dive$, elements $u$ in the
domain of $\Grad_{c}$ satisfy an abstract Dirichlet boundary condition
of the form $u=0$ on $\partial\Omega$, while elements $\sigma$
in the domain of $\Dive_{c}$ satisfy an abstract Neumann boundary
condition of the form $\sigma N=0$ on $\partial\Omega,$ where $N$
denotes the unit outward normal vector field on $\partial\Omega.$\\
The equation of linear elasticity is given by 
\begin{equation}
\partial_{0}^{2}u-\Dive\sigma=f,\label{eq:elast}
\end{equation}
where $u\in L_{\nu}^{2}(\mathbb{R},L^{2}(\Omega)^{3})$ denotes the
displacement field of the elastic body $\Omega$, $\sigma\in L_{\nu}^{2}(\mathbb{R},H(\Omega))$
is the stress tensor and $f\in L_{\nu}^{2}(\mathbb{R},L^{2}(\Omega)^{3})$
is an external force. Equation (\ref{eq:elast}) is completed by Hooke's
law 
\begin{equation}
\sigma=C\Grad u,\label{eq:Hooke}
\end{equation}
where $C\in L(H_{\mathrm{sym}}(\Omega))$ is the so-called elasticity
tensor, which is assumed to be selfadjoint and strictly positive definite
(which in particular gives the bounded invertibility of $C$). Defining
$v\coloneqq\partial_{0}u$ as the displacement velocity, (\ref{eq:elast})
and (\ref{eq:Hooke}) can be written as a system of the form 
\[
\left(\partial_{0}\left(\begin{array}{cc}
1 & 0\\
0 & C^{-1}
\end{array}\right)+\left(\begin{array}{cc}
0 & -\Dive\\
-\Grad & 0
\end{array}\right)\right)\left(\begin{array}{c}
v\\
\sigma
\end{array}\right)=\left(\begin{array}{c}
f\\
0
\end{array}\right).
\]
Imposing boundary conditions on $v$ or $\sigma$, say for simplicity
Neumann boundary conditions for $\sigma$, we end up with an evolutionary
equation with 
\[
M(z)=\left(\begin{array}{cc}
1 & 0\\
0 & C^{-1}
\end{array}\right)
\]
and 
\[
A=\left(\begin{array}{cc}
0 & -\Dive_{c}\\
-\Grad & 0
\end{array}\right).
\]
Since $C^{-1}$ is also selfadjoint and strictly positive definite,
we obtain that $M$ satisfies the solvability condition (\ref{eq:pos_material_law}).\\
In order to incorporate viscous materials, i.e. materials showing
some memory effects, one modifies Hooke's law (\ref{eq:Hooke}) for
example by 
\begin{equation}
\sigma=C\Grad u+D\Grad\partial_{0}u,\label{eq:Kelvin-Voigt}
\end{equation}
where $D\in L(H_{\mathrm{sym}}(\Omega)).$ This modification is known
as the Kelvin-Voigt model for visco-elastic materials (see e.g. \cite[p. 163]{Duvaut1976},
\cite[Section 1.3.3]{Bertram2005}). Using $v$ instead of $u$, the
latter equation reads as 
\[
\sigma=\left(\partial_{0}^{-1}C+D\right)\Grad v=D\left(\partial_{0}^{-1}D^{-1}C+1\right)\Grad v,
\]
where we assume that $D$ is selfadjoint and strictly positive definite,
while $C\in L(H_{\mathrm{sym}}(\Omega))$ is arbitrary (the assumption
on $D$ can be relaxed, by requiring suitable positivity constraints
on $C$, see e.g. \cite[Theorem 4.1]{Picard2013_fractional}). Since
$\|\partial_{0,\nu}^{-1}\|=\nu^{-1}$ we may choose $\nu_{0}>0$ large
enough in order to get that $\|\partial_{0,\nu}^{-1}D^{-1}C\|<1$
for all $\nu\text{\ensuremath{\geq}}\nu_{0}.$ Then, by the Neumann
series, we end up with 
\[
\left(1+\partial_{0}^{-1}D^{-1}C\right)^{-1}D^{-1}\sigma=\Grad v.
\]
Hence, our new material law operator becomes 
\begin{align*}
M(z) & =\left(\begin{array}{cc}
1 & 0\\
0 & z(1+zD^{-1}C)^{-1}D^{-1}
\end{array}\right)\\
 & =\left(\begin{array}{cc}
1 & 0\\
0 & 0
\end{array}\right)+z\left(\begin{array}{cc}
0 & 0\\
0 & D^{-1}
\end{array}\right)+\sum_{k=1}^{\infty}(-1)^{k}z^{k+1}\left(\begin{array}{cc}
0 & 0\\
0 & \left(D^{-1}C\right)^{k}D^{-1}
\end{array}\right),
\end{align*}
which satisfies the condition (\ref{eq:pos_material_law}) for $z\in B(r,r),$
where $r>0$ is chosen sufficiently small. 

Another way to model materials with memory is to add a convolution term on the right hand side of
Hooke's law (\ref{eq:Hooke}), see e.g. \cite[Section III.7]{Duvaut1976},
\cite{Dafermos1970_abtract_Volterra}. The new constitutive relation
then reads as 
\[
\sigma=C\Grad u-k\ast\Grad u=\partial_{0}^{-1}\left(C-k\ast\right)\Grad v,
\]
where $k:\mathbb{R}_{\geq0}\to L(H_{\mathrm{sym}}(\Omega))$ is a
strongly measurable function satisfying $\intop_{0}^{\infty}\|k(t)\|e^{-\mu t}\mbox{ d}t<\infty$
for some $\mu\geq0$ (note that in \cite{Dafermos1970_abtract_Volterra}
$k$ was assumed to be absolutely continuous). Again, choosing $\nu>0$
large enough, we end up with a material law operator (see \cite[Subsection 4.1]{Trostorff2012_integro})
\begin{equation}
M(z)=\left(\begin{array}{cc}
1 & 0\\
0 & C^{-\frac{1}{2}}(1-\sqrt{2\pi}C^{-\frac{1}{2}}\:\hat{k}(-\ii z^{-1})C^{-\frac{1}{2}})^{-1}C^{-\frac{1}{2}}
\end{array}\right),\label{eq:material_law_dafermos}
\end{equation}
where $\hat{k}$ denotes the Fourier transform of $k$ and where we
have used that 
\[
\mathcal{L}_{\nu}\left(k\ast\right)\mathcal{L}_{\nu}^{\ast}=\sqrt{2\pi}\:\hat{k}(\mathrm{m}-\ii\nu)
\]
for $\nu\geq\mu$ (see e.g. \cite[Lemma 3.4]{Trostorff2012_integro}).
According to the solution theory presented in Theorem \ref{thm:sol_theo_skew},
we have to find suitable assumptions on the kernel $k$ in order to
obtain the positivity condition (\ref{eq:pos_material_law}). This
is done in the following theorem. \end{example}
\begin{thm}[Integro-differential equations, \cite{Trostorff2012_integro}]
\label{thm:integro} Let $H$ be a separable Hilbert space and $k:\mathbb{R}_{\geq0}\to L(H)$
a strongly measurable function satisfying $\intop_{0}^{\infty}\|k(t)\|e^{-\mu t}<\infty$
for some $\mu\in\mathbb{R}$. If

\begin{enumerate}[(a)]

\item $k(t)$ is selfadjoint for almost every $t\in\mathbb{R}_{\geq0}$,

\item there exists $d\geq0,\nu_{0}\geq\mu$ such that 
\[
t\Im\hat{k}(t-\ii\nu_{0})\leq d
\]
for all $t\in\mathbb{R},$

\end{enumerate}

then there exists $r>0$ such that the material law $M(z)\coloneqq 1+\sqrt{2\pi}\:\hat{k}(-\ii z^{-1})$
with $z\in B(r,r)$ satisfies (\ref{eq:pos_material_law}). If, in
addition, $k$ satisfies

\begin{enumerate}[(c)]

\item $k(t)k(s)=k(s)k(t)$ for almost every $s,t\in\mathbb{R}_{\geq0},$

\end{enumerate}

then there exists $r'>0$ such that the material law $\tilde{M}(z)\coloneqq\left(1-\sqrt{2\pi}\:\hat{k}(-\ii z^{-1})\right)^{-1}$
with $z\in B(r',r')$ satisfies (\ref{eq:pos_material_law}).\end{thm}
\begin{rem}
For real-valued kernels $k$, a typical assumption is that $k$ should
be non-negative and non-increasing (see e.g. \cite{Pruss2009}). This,
however implies the assumptions (a)-(c) of Theorem \ref{thm:integro}
for $k$. Moreover, kernels $k$ of bounded variation satisfy the
assumptions of the latter theorem (see \cite[Remark 3.6]{Trostorff2012_integro}).
\end{rem}
Using Theorem \ref{thm:integro} we get that $M$ given by (\ref{eq:material_law_dafermos})
satisfies the solvability condition (\ref{eq:pos_material_law}),
if $k$ satisfies the assumptions (a)-(c) of Theorem \ref{thm:integro}
and $k(t)$ and $C$ commute for almost every $t\in\mathbb{R}_{\geq0}$
(see \cite[Subsection 4.1]{Trostorff2012_integro} for a detailed
study).

Convolutions as discussed in the previous theorem need to be incorporated
due to the fact that, for instance, the elastic behavior of a solid
body depends on the stresses the body experienced in the past. One
also speaks of so-called memory effects. From a similar type of nature
is the modeling of material behavior with the help of fractional (time-)derivatives.
In fact, in recent years, material laws for visco-elastic solids were
described by using fractional derivatives (see e.g. \cite{Bertram2005,Nolte2003})
as an ansatz to better approximate a polynomial in $\partial_{0}^{-1}$
by potentially fewer terms of real powers of $\partial_{0}^{-1}$.
In \cite{Waurick2013} a model for visco-elasticity with fractional
derivatives has been analyzed mainly in the context of homogenization
issues, which will be discussed below. For the moment, we stick to
the presentation of the model and sketch the idea of the well-posedness
result for this type of equation presented in \cite[Theorem 2.1 and 2.2]{Waurick2013}.
\begin{example}[Visco-elasticity with fractional derivatives]
 In this model the Kelvin-Voigt model (\ref{eq:Kelvin-Voigt}) is
replaced by a fractional analogue of the form 
\[
\sigma=C\Grad u+D\Grad\partial_{0}^{\alpha}u,
\]
for some $\alpha\in]0,1]$. We emphasize here that $\partial_{0}^{\alpha}$
has a natural meaning as a function of a normal operator. We refer
to \cite[Subsection 2.1]{Picard2013_fractional} for a comparison
with classical notions of fractional derivatives (see e.g.~\cite{Podlubny1999}
for an introduction to fractional derivatives). As in the case $\alpha=1$
we get that $\partial_{0}^{\alpha}$ is boundedly invertible for each
$\alpha\in]0,1]$ and $\nu>0$ and by \cite[Lemma 2.1]{Picard2013_fractional}
we can estimate the norm of its inverse by 
\[
\|\partial_{0,\nu}^{-\alpha}\|\leq\nu^{-\alpha}.
\]
Again we assume for simplicity that $D$ is selfadjoint and strictly
positive definite, and thus we may rewrite the constitutive relation
above as 
\[
\sigma=(C+D\partial_{0}^{\alpha})\Grad u=D\partial_{0}^{\alpha}\left(\partial_{0}^{-\alpha}D^{-1}C+1\right)\Grad u
\]
yielding for large enough $\nu>0$ 
\[
\partial_{0}^{1-\alpha}\left(\partial_{0}^{-\alpha}D^{-1}C+1\right)^{-1}D^{-1}\sigma=\Grad v.
\]
Hence, our material law operator is given by 
\begin{equation}
\left(\begin{array}{cc}
1 & 0\\
0 & z^{\alpha}\left(1+z^{\alpha}D^{-1}C\right)^{-1}D^{-1}
\end{array}\right)=\left(\begin{array}{cc}
1 & 0\\
0 & z^{\alpha}D^{-1}
\end{array}\right)+\sum_{k=1}^{\infty}(-1)^{k}z^{(k+1)\alpha}\left(\begin{array}{cc}
0 & 0\\
0 & \left(D^{-1}C\right)^{k}D^{-1}
\end{array}\right).\label{eq:frac_mart_law}
\end{equation}
Note that the visco-elastic model under consideration is slightly
different from the one treated in \cite{Waurick2013} as there is
no further restriction on the parameter $\text{\ensuremath{\alpha}}$.
In \cite{Waurick2013}, we assumed $\alpha\geq\frac{1}{2}$ and showed
the positive definiteness of the sum in (\ref{eq:frac_mart_law})
with the help of a perturbation argument. This argument does not apply
to the situation discussed here. However, assuming selfadjointness
and non-negativity of the operators $C$ and $D$ well-posedness can
be warranted even for $\alpha<\frac{1}{2}$.
\end{example}
More generally, if one considers material law operators containing
fractional derivatives of the form%
\footnote{The infinite series in (\ref{eq:frac_mart_law}) can be treated with
the help of a perturbation argument.%
} 
\begin{equation}
M(\partial_{0}^{-1})=M_{0}+\sum_{\alpha\in\Pi}\partial_{0}^{-\alpha}M_{\alpha}+\partial_{0}^{-1}M_{1},\label{eq:material_law_fractional}
\end{equation}
where $\Pi\subseteq\left]0,1\right[$ is finite and $M_{\alpha}\in L(H)$
for some Hilbert space $H$ and each $\alpha\in\{0,1\}\cup\Pi,$ one
imposes the following conditions on the operators $M_{\alpha}$ in
order to get an estimate of the form (\ref{eq:pos_material_law})
for the material law (\ref{eq:material_law_fractional}):
\begin{thm}[{\cite[Theorem 3.5]{Picard2013_fractional}}]
\label{thm:frac}Let $(\alpha_{0},\ldots,\alpha_{k})$ be a monotonically
increasing enumeration of $\Pi.$ Assume that the operators $M_{0}$
and $M_{\alpha_{j}}$ are selfadjoint for each $j\in\{0,\ldots,k\}.$
Moreover, let $P,Q,F\in L(H)$ be three orthogonal projectors satisfying
\[
P+Q+F=1
\]
and assume, that $M_{0}$ and $M_{\alpha_{j}}$ commute with $P,Q$
and $F$ for every $j\in\{1,\ldots,k\}.$ If $PM_{\alpha_{j}}P\geq0,$
$QM_{\alpha_{j}}Q\geq0,$ $M_{0}\geq0$ and $M_{0},\Re M_{1}$ and
$M_{\alpha_{0}}$ are strictly positive definite on the ranges of
$P,Q$ and $F$ respectively, then the material law (\ref{eq:material_law_fractional})
satisfies the solvability condition (\ref{eq:pos_material_law}). 
\end{thm}
With this theorem we end our tour through different kinds of evolutionary
equations, which are all covered by the solution theory stated in
Theorem \ref{thm:sol_theo_skew}. 

Besides the well-posedness of evolutionary equations, it is also possible
to derive a criterion for (exponential) stability in the abstract
setting of Theorem \ref{thm:sol_theo_skew}. Since the systems under
consideration do not have any regularizing property, we are not able
to define exponential stability as it is done classically, since our
solutions $u$ do not have to be continuous. So, point-wise evaluation
of $u$ does not have any meaning. Indeed, the problem class discussed
in Theorem \ref{thm:sol_theo_skew} covers also purely algebraic systems,
where definitely no regularity of the solutions is to be expected
unless the given data is regular. Thus, we are led to define a weaker
notion of exponential stability as follows.
\begin{defn}
\label{Def:exp_stab}Let $A\colon D(A)\subseteq H\to H$ be skew-selfadjoint%
\footnote{For sake of presentation, we assume $A$ to be skew-selfadjoint. However,
in \cite{Trostorff2014_stability,Trostorff2013_stability} $A$ was
assumed to be a linear maximal monotone operator. We will give a solution
theory for this type of problem later on. One then might replace the
condition of skew-selfadjointness in this definition and the subsequent
theorem by the condition of being linear and maximal monotone.%
} and $M\in\mathcal{H}^{\infty}(B(r,r),L(H))$ satisfying (\ref{eq:pos_material_law})
for some $r>0$. Let $\nu>\frac{1}{2r}$. We call the operator $\left(\overline{\partial_{0}M(\partial_{0}^{-1})+A}\right)^{-1}$
\emph{exponentially stable} with stability rate $\nu_{0}>0$ if for
each $0\leq\nu'<\nu_{0}$ and $f\in L_{-\nu'}^{2}\cap L_{\nu}^{2}(\mathbb{R},H)$
we have
\[
u\coloneqq\left(\overline{\partial_{0}M(\partial_{0}^{-1})+A}\right)^{-1}f\in L_{-\nu'}^{2}(\mathbb{R},H),
\]
which in particular implies that $\int_{\mathbb{R}}e^{2\nu't}\left|u(t)\right|^{2}dt<\infty$. 
\end{defn}
As it turns out, this notion of exponential stability yields the exponential
decay of the solutions, provided the solution $u$ is regular enough.
For instance, this can be achieved by assuming more regularity on
the given right-hand side (see \cite[Remark 3.2 (a)]{Trostorff2013_stability}).
The result for exponential stability reads as follows.
\begin{thm}[{\cite[Theorem 3.2]{Trostorff2014_stability}}]
\label{thm:criterion-stab}Let $A\colon D(A)\subseteq H\to H$ be
a skew-selfadjoint operator and $M$ be a mapping satisfying the following
assumptions for some $\nu_{0}>0$:

(a) $M\colon\mathbb{C}\setminus\overline{B\left(-\frac{1}{2\nu_{0}},\frac{1}{2\nu_{0}}\right)}\to L(H)$
is analytic;

(b) for every $0<\nu'<\nu_{0}$ there exists $c>0$ such that for
all $z\in\mathbb{C}\setminus\overline{B\left(-\frac{1}{2\nu'},\frac{1}{2\nu'}\right)}$
we have
\[
\Re z^{-1}M(z)\geq c.
\]
Then for each $\nu>0$ the solution operator $\left(\overline{\partial_{0}M(\partial_{0}^{-1})+A}\right)^{-1}$
is exponentially stable with stability rate $\nu_{0}$.\end{thm}
\begin{example}[{A parabolic-hyperbolic system, \cite[Example 4.2]{Trostorff2013_stability}}]
 \label{ex:para-hypo-stab}Let $\Omega\subseteq\mathbb{R}^{n}$ be
an open subset and $\Omega_{0},\Omega_{1}\subseteq\Omega$ measurable,
disjoint, non-empty, $c>0$. Then the solution operator for the equation
\[
\left(\partial_{0}\left(\begin{array}{cc}
\1_{_{\Omega_{0}}}+\1_{_{\Omega_{1}}} & 0\\
0 & \1_{_{\Omega_{0}}}
\end{array}\right)+\left(\begin{array}{cc}
c & 0\\
0 & c
\end{array}\right)+\left(\begin{array}{cc}
0 & \dive_{c}\\
\grad & 0
\end{array}\right)\right)\left(\begin{array}{c}
v\\
q
\end{array}\right)=\left(\begin{array}{c}
f\\
0
\end{array}\right)
\]
for suitable $f$ is exponentially stable with stability rate $c$.\end{example}
\begin{rem}
As in \cite[Initial Value Problems]{Trostorff2013_stability},
the stability of the corresponding initial value problems can be discussed
similarly. %
\end{rem}
\begin{example}[{Example \ref{ex:para-hypo-stab} continued (see also \cite[Theorem 4.4]{Trostorff2013_stability})}]
 Let $h<0$ and assume, in addition, that $c>1$. Then the solution
operator for the equation 
\[
\left(\partial_{0}\left(\begin{array}{cc}
\1_{_{\Omega_{0}}}+\1_{_{\Omega_{1}}} & 0\\
0 & \1_{_{\Omega_{0}}}
\end{array}\right)+\tau_{h}+\left(\begin{array}{cc}
c & 0\\
0 & c
\end{array}\right)+\left(\begin{array}{cc}
0 & \dive_{c}\\
\grad & 0
\end{array}\right)\right)\left(\begin{array}{c}
v\\
q
\end{array}\right)=\left(\begin{array}{c}
f\\
0
\end{array}\right)
\]
is exponentially stable with stability rate $\nu_{0}>0$ such that
\[
\nu_{0}+e^{-\nu_{0}h}=c.
\]
\end{example}
\begin{rem}
We note that the exponential stability of integro-differential equations
can be treated in the same way, see \cite[Section 4.3]{Trostorff2013_stability}.
\end{rem}

\subsection{The closedness of the problem class and homogenization\label{sub:The-closedness-of}}

In this section we discuss the closedness of the problem class under
consideration with respect to perturbations in the material law $M$.
We will treat perturbations in the weak operator topology, which will
also have strong connections to issues stemming from homogenization
theory. For illustrational purposes we discuss the one dimensional
case of an elliptic type equation first.
\begin{example}[{see e.g.~\cite[Example 1.1.3]{BenLiPap}}]
\label{ex:simple_dim1} Let $A\colon\mathbb{R}\to\mathbb{R}$ be
a bounded, uniformly strictly positive, measurable, $1$-periodic
function. We denote the multiplication operator on $L^{2}(\oi01)$
associated with $A$ by $A(\mathrm{m}_{1})$. Denoting the one-dimensional
gradient on $\oi01$ with homogeneous Dirichlet boundary conditions
by $\partial_{1,c}$ (see also Definition \ref{def:div_grad}) and $\partial_{1}$
for its skew-adjoint, we consider the problem of finding $u_{\epsilon}\in D(\partial_{1,c})$
such that for given $f\in L^{2}(\oi01)$ and $\epsilon>0$ we have
\[
-\partial_{1}A\left(\frac{1}{\epsilon}\mathrm{m}_{1}\right)\partial_{1,c}u_{\epsilon}=f.
\]
Of course, the solvability of the latter problem is clear due to Corollary
\ref{cor:Lax-Mil}. Now, we address the question whether $(u_{\epsilon})_{\epsilon>0\text{ }}$
is convergent in any particular sense and if so, whether the limit
satisfies a differential equation of ``similar type''. Before, however,
doing so, we need the following result.\end{example}
\begin{prop}[{see e.g. \cite[Theorem 2.6]{CioDon}}]
\label{prop:per_impl_conv} Let $A\colon\mathbb{R}^{N}\to\text{\ensuremath{\mathbb{C}}}$
be bounded, measurable and \emph{$\oi01{}^{N}$-periodic},\emph{ }i.e.,
for all $k\in\mathbb{Z}^{N}$ and a.e. $x\in\mathbb{R}^{N}$ we have
$A(x+k)=A(x)$. Then 
\[
A\left(\frac{\cdot}{\epsilon}\right)\to\int_{\oi01{}^{N}}A(x)dx\quad(\epsilon\to0)
\]
 in the weak-$*$-topology $\sigma(L^{\infty}(\mathbb{R}^{N}),L^{1}(\mathbb{R}^{N}))$
of $L^{\infty}(\mathbb{R}^{N})$.\end{prop}
\begin{example}[Example \ref{ex:simple_dim1} continued]
 For $\epsilon>0$, we define $\xi_{\epsilon}\coloneqq A\left(\frac{1}{\epsilon}\mathrm{m}_{1}\right)\partial_{1,c}u_{\epsilon}$.
It is easy to see that $(\xi_{\epsilon})_{\epsilon}$ is bounded in
$L^{2}(\oi01)$ and also in $H_{1}(\oi01)=D(\partial_{1})$. The Arzela-Ascoli
theorem implies that $(\xi_{\epsilon})_{\epsilon}$ has a convergent
subsequence (again labelled with $\epsilon$), which converges in
$L^{2}(\oi01)$. We denote $\xi\coloneqq\lim_{\epsilon\to0}\xi_{\epsilon}.$
In consequence, by Proposition \ref{prop:per_impl_conv}, we deduce
that 
\[
\partial_{1,c}u_{\epsilon}=\frac{1}{A\left(\frac{1}{\epsilon}\mathrm{m}_{1}\right)}\xi_{\epsilon}\rightharpoonup\left(\int_{0}^{1}\frac{1}{A(x)}dx\right)\xi
\]
weakly in $L^{2}(\oi01)$ as $\epsilon\to0$. Hence, $(\partial_{1,c}u_{\epsilon})_{\epsilon>0}$
weakly converges in $L^{2}(\oi01)$, which, again by compact embedding,
implies that $(u_{\epsilon})_{\epsilon}$ converges in $L^{2}(\oi01)$.
Denoting the respective limit by $u$, we infer
\[
f=-\partial_{1}\xi=-\partial_{1}\left(\int_{0}^{1}\frac{1}{A(x)}dx\right)^{-1}\left(\int_{0}^{1}\frac{1}{A(x)}dx\right)\xi=-\partial_{1}\left(\int_{0}^{1}\frac{1}{A(x)}dx\right)^{-1}\partial_{1,c}u.
\]
Now, unique solvability of the latter equation together with a subsequence
argument imply convergence of $(u_{\epsilon})_{\epsilon}$ without
choosing subsequences.\end{example}
\begin{rem}
Note that examples in dimension $2$ or higher are far more complicated.
In particular, the computation of the limit (if there is one) is more
involved. To see this, we refer to \cite[Sections 5.4 and 6.2]{CioDon},
where the case of so-called laminated materials and general periodic
materials is discussed. In the former the limit may be expressed as
certain integral means, whereas in the latter so-called local problems
have to be solved to determine the effective equation. Having these
issues in mind, we will only give structural (i.e. compactness) results
on homogenization problems of (evolutionary) partial differential
equations. In consequence, the compactness properties of the differential
operators as well as the ones of the coefficients play a crucial role
in homogenization theory.
\end{rem}
Regarding Proposition \ref{prop:per_impl_conv}, the right topology
for the operators under consideration is the weak operator topology.
Indeed, with the examples given in the previous section in mind and
modeling local oscillations as in Example \ref{ex:simple_dim1}, we
shall consider the weak-$*$-topology of an appropriate $L^{\infty}$-space.
Now, if we identify any $L^{\infty}$-function with the corresponding
multiplication operator on $L^{2}$, we see that convergence in the
weak-$*$-topology of the functions is equivalent to convergence of
the associated multiplication operator in the weak operator topology
of $L(L^{2})$. This general perspective also enables us to treat
problems with singular perturbations and problems of mixed type.

Before stating a first theorem concerning the issues mentioned, we
need to introduce a topology tailored for the case of autonomous and
causal material laws. 
\begin{defn}[{\cite[Definition 3.1]{Waurick}}]
\label{def:topo} For Hilbert spaces $H_{1}$, $H_{2}$ and an open
set $E\subseteq\mathbb{C},$ we define $\tau_{\textnormal{w}}$ to
be the initial topology on $\mathcal{H}^{\infty}(E,L(H_{1},H_{2}))$
induced by the mappings 
\[
\mathcal{H}^{\infty}\ni M\mapsto\left(z\mapsto\langle\phi,M(z)\psi\rangle\right)\in\mathcal{H}(E)
\]
for $\phi\in H_{2},\psi\in H_{1}$, where $\mathcal{H}(E)$ is the
set of holomorphic functions endowed with the compact open topology,
i.e., the topology of uniform convergence on compact sets. We write
$\mathcal{H}_{\textnormal{w}}^{\infty}\coloneqq\left(\mathcal{H}^{\infty},\tau_{\textnormal{w}}\right)$
for the topological space and re-use the notation $\mathcal{H}_{\textnormal{w}}^{\infty}$
for the underlying set. 
\end{defn}
We note the following remarkable fact.
\begin{thm}[{\cite[Theorem 3.4]{Waurick}, \cite[Theorem 4.3]{Waurick2013}}]
 Let $H_{1},H_{2}$ be Hilbert spaces, $E\subseteq\mathbb{C}$ open.
Then
\[
B_{\mathcal{H}^{\infty}}\coloneqq\left\{ M\in\mathcal{H}^{\infty}(E,L(H_{1},H_{2})) \,|\, \|M\|_{\infty}\leq1\right\} \subseteq\mathcal{H}_{\textnormal{w}}^{\infty}
\]
is compact. If, in addition, $H_{1}$ and $H_{2}$ are separable,
then $B_{\mathcal{H}^{\infty}}$ is metrizable.\end{thm}
\begin{proof}[Sketch of the proof]
 For $s\in[0,\infty[$ introduce the set $B_{\mathcal{H}(E)}(s)\coloneqq\{f\in\mathcal{H}(E)|\forall z\in E:|f(z)|\leq s\}$.
The proof is based on the following equality 
\[
B_{\mathcal{H}^{\infty}}=\left(\prod_{\phi\in H_{1},\psi\in H_{2}}B_{\mathcal{H}(E)}\left(\|\phi\|\|\psi\|\right)\right)\cap\left\{ M\colon E\to L(H_{1},H_{2})\,|\, M(z)\text{ sesquilinear }(z\in E)\right\} ,
\]
which itself follows from a Dunford type theorem ensuring the holomorphy
(with values in the space $L(H_{1},H_{2})$) of the elements on the right-hand
side and the Riesz-Fr\'{e}chet representation theorem for sesqui-linear
forms. Now, invoking Montel's theorem, we deduce that $B_{\mathcal{H}(E)}(s)$
is compact for every $s\in[0,\infty[$. Thus, Tikhonov's theorem applies
to deduce the compactness of $B_{\mathcal{H}^{\infty}}$. The proof
for metrizability is standard.
\end{proof}
Recall for $r,c\in\mathbb{R}_{>0},$ and a Hilbert space $H$, we
set 
\[
\mathcal{H}^{\infty,c}(B(r,r),L(H))=\left\{ M\in\mathcal{H}^{\infty}(B(r,r),L(H))\,|\,\bigwedge_{z\in B(r,r)}\Re z^{-1}M(z)\geq c\right\} .
\]
In accordance to Definition \ref{def:topo}, we will also write $\mathcal{H}_{\textnormal{w}}^{\infty,c}$
for the set $\mathcal{H}^{\infty,c}$ endowed with $\tau_{\textnormal{w}}$.
The compactness properties from $\mathcal{H}_{\textnormal{w}}^{\infty}$
are carried over to $\mathcal{H}_{\textnormal{w}}^{\infty,c}$. The
latter follows from the following proposition:
\begin{prop}[{\cite[Proposition 1.3]{Waurick2012a}}]
 Let $c\in\mathbb{R}_{>0}$. Then the set $\mathcal{H}_{\textnormal{w}}^{\infty,c}\subseteq\mathcal{H}_{\textnormal{w}}^{\infty}$
is closed.
\end{prop}
We are now ready to discuss a first theorem on the continuous dependence
on the coefficients for autonomous and causal material laws, which
particularly covers a class of homogenization problems in the sense
mentioned above. For a linear operator $A$ in some Hilbert space
$H$, we denote $D(A)$ endowed with the graph norm of $A$ by $D_{A}$.
If a Hilbert space $H_{1}$ is compactly embedded in $H$, we write
$H_{1}\hookrightarrow\hookrightarrow H$. A subset $M\subseteq\mathcal{H}^{\infty}$
is called \emph{bounded}, if there is $\lambda>0$ such that $M\subseteq\lambda B_{\mathcal{H}^{\infty}}$.
The result reads as follows.
\begin{thm}[{\cite[Theorem 3.5]{Waurick2012a}, \cite[Theorem 4.1]{Waurick2013}}]
\label{thm:Hom_1_auto_comp} Let $\nu,c\in\mathbb{R}_{>0}$, $r>\frac{1}{2\nu},$
$(M_{n})_{n}$ be a bounded sequence in $\mathbb{\mathcal{H}}^{\infty,c}(B(r,r),L(H))$, $A\colon D(A)\subseteq H\to H$ skew-selfadjoint. Assume
that $D_{A}\hookrightarrow\hookrightarrow H$. Then there exists a
subsequence of $(M_{n})_{n}$ such that $(M_{n_{k}})_{k}$ converges
in $\mathcal{H}_{\textnormal{w}}^{\infty}$ to some $M\in\mathcal{H}^{\infty,c}$
and
\[
\left(\overline{\partial_{0}M_{n_{k}}(\partial_{0}^{-1})+A}\right)^{-1}\to\left(\overline{\partial_{0}M(\partial_{0}^{-1})+A}\right)^{-1}
\]
in the weak operator topology.
\end{thm}
We first apply this theorem to an elliptic type equation.
\begin{example}
\label{ex:hom_dirich}Let $\Omega\subseteq\mathbb{R}^{n}$ be open
and bounded. Let $\grad_{c}$ and $\dive$ be the operators introduced
in Definition \ref{def:div_grad}. Let $\left(a_{k}\right)_{k\in\mathbb{N}}$
be a sequence of uniformly strictly positive bounded linear operators
in $L^{2}(\Omega)^{n}$. For $f\in L^{2}(\Omega)$ consider for $k\in\mathbb{N}$
the problem of finding $u_{k}\in L^{2}(\Omega)$ such that the equation
\[
u_{k}-\dive a_{k}\grad_{c}u_{k}=f
\]
holds. Observe that if $\iota\colon R(\grad_{c})\to L^{2}(\Omega)^{n}$
denotes the canonical embedding, this equation is the same as
\begin{equation}
u_{k}-\dive\iota\iota^{*}a_{k}\iota\iota^{*}\grad_{c}u_{k}=f.\label{eq:mod_ell}
\end{equation}
Indeed, by Poincar{\'e}'s inequality $R(\grad_{c})\subseteq L^{2}(\Omega)^{n}$
is closed, the projection theorem ensures that $\iota\iota^{*}$ is
the orthogonal projection on $R(\grad_{c})$. Moreover, $N(\dive)=R(\grad_{c})^{\bot}$
yields that $\dive=\dive(\iota\iota^{*}+(1-\iota\iota^{*}))=\dive\iota\iota^{*}.$
Now, we realize that due to the positive definiteness of $a_{k}$
so is $\iota^{*}a_{k}\iota$. Consequently, the latter operator is
continuously invertible. Introducing $v_{k}\coloneqq\iota^{*}a_{k}\iota\grad_{c}u_{k}$
for $k\in\mathbb{N},$ we rewrite the equation (\ref{eq:mod_ell})
as follows
\[
\left(\left(\begin{array}{cc}
1 & 0\\
0 & \left(\iota^{*}a_{k}\iota\right)^{-1}
\end{array}\right)-\left(\begin{array}{cc}
0 & \dive\iota\\
\iota^{*}\grad_{c} & 0
\end{array}\right)\right)\left(\begin{array}{c}
u_{k}\\
v_{k}
\end{array}\right)=\left(\begin{array}{c}
f\\
0
\end{array}\right).
\]
Now, let $\nu>0$ and lift the above problem to the space $L_{\nu}^{2}(\mathbb{R},L^{2}(\Omega)\oplus R(\grad_{c}))$
by interpreting $\left(\begin{array}{c}
f\\
0
\end{array}\right)$ as $\left(t\mapsto\1_{_{\mathbb{R}_{>0}}}(t)\left(\begin{array}{c}
f\\
0
\end{array}\right)\right)\in L_{\nu}^{2}(\mathbb{R},L^{2}(\Omega)\oplus R(\grad_{c}))$. Then this equation fits into the solution theory stated in Theorem
\ref{thm:sol_theo_skew} with 
\[
M_{k}(\partial_{0}^{-1})\coloneqq\partial_{0}^{-1}\left(\begin{array}{cc}
1 & 0\\
0 & \left(\iota^{*}a_{k}\iota\right)^{-1}
\end{array}\right),\quad A\coloneqq\left(\begin{array}{cc}
0 & \dive\iota\\
\iota^{*}\grad_{c} & 0
\end{array}\right).
\]
Note that the skew-selfadjointness of $A$ is easily obtained from
$\dive^{*}=-\grad_{c}$. In order to conclude the applicability of
Theorem \ref{thm:Hom_1_auto_comp}, we need the following observation.\end{example}
\begin{prop}[{\cite[Lemma 4.1]{Waurick2012a}}]
\label{prop:null_away_comp} Let $H_{1},$ $H_{2}$ be Hilbert spaces,
$C\colon D(C)\subseteq H_{1}\to H_{2}$ densely defined, closed, linear.
Assume that $D_{C}\hookrightarrow\hookrightarrow H_{1}$. Then $D_{C^{*}}\cap N(C^{*})^{\bot_{H_{2}}}\hookrightarrow\hookrightarrow H_{2}$.\end{prop}
\begin{example}[Example \ref{ex:hom_dirich} continued]
 With the help of the theorem of Rellich-Kondrachov and Proposition
\ref{prop:null_away_comp}, we deduce that $A$ has compact resolvent.
Thus, Theorem \ref{thm:Hom_1_auto_comp} is applicable. We find a
subsequence such that $\mathfrak{a}\coloneqq\tau_{\textnormal{w}}-\lim_{l\to\infty}\left(\iota^{*}a_{k_{l}}\iota\right)^{-1}$
exists, where we denoted by $\tau_{\textnormal{w}}$ the weak operator
topology. Therefore, $(u_{k_{l}})_{l}$ weakly converges to some $u$,
which itself is the solution of
\[
u-\dive\iota\mathfrak{a}^{-1}\iota^{*}\grad_{c}u=f.
\]
In fact it is possible to show that $\iota\mathfrak{a}^{-1}\iota^{*}$
coincides with the usual homogenized matrix (if the possibly additional
assumptions on the sequence $(a_{k})_{k}$ permit the computation
of a limit in the sense of $H$- or $G$-convergence, see e.g. \cite[Chapter 13]{CioDon}
and the references therein).
\end{example}
As a next example let us consider the heat equation.
\begin{example}
\label{ex:heat_hom}Recall the heat equation introduced in Example
\ref{ex:heat_eq}:
\[
\left(\partial_{0}\left(\begin{array}{cc}
1 & 0\\
0 & 0
\end{array}\right)+\left(\begin{array}{cc}
0 & 0\\
0 & k^{-1}
\end{array}\right)+\left(\begin{array}{cc}
0 & \dive\\
\grad_{c} & 0
\end{array}\right)\right)\left(\begin{array}{c}
\theta\\
q
\end{array}\right)=\left(\begin{array}{c}
f\\
0
\end{array}\right).
\]
To warrant the compactness condition in Theorem \ref{thm:Hom_1_auto_comp},
we again assume that the underlying domain $\Omega$ is bounded. Similarly
to Example \ref{ex:hom_dirich}, we assume that we are given $(k_{l})_{l}$,
a bounded sequence of uniformly strictly monotone linear operators
in $L(L^{2}(\Omega)^{n})$. Consider the sequence of equations
\[
\left(\partial_{0}\left(\begin{array}{cc}
1 & 0\\
0 & 0
\end{array}\right)+\left(\begin{array}{cc}
0 & 0\\
0 & k_{l}^{-1}
\end{array}\right)+\left(\begin{array}{cc}
0 & \dive\\
\grad_{c} & 0
\end{array}\right)\right)\left(\begin{array}{c}
\theta_{l}\\
q_{l}
\end{array}\right)=\left(\begin{array}{c}
f\\
0
\end{array}\right).
\]
Now, focussing \emph{only} on the behavior of the temperature $(\theta_{l})_{l}$,
we can proceed as in the previous example. 
\end{example}
Assuming more regularity of $\Omega$, e.g., the segment property
and finitely many connected components, we can apply Theorem \ref{thm:Hom_1_auto_comp}
also to the corresponding homogeneous Neumann problems of Examples
\ref{ex:hom_dirich} and \ref{ex:heat_hom}. Moreover, the aforementioned
theorem can also be applied to the homogenization of (visco-)elastic
problems (see also Example \ref{ex:visco_elastic}). For this we need
criteria ensuring the compactness condition $D_{\Grad_{c}}\hookrightarrow L^{2}(\Omega)^{n}$
(or $D_{\Grad}\hookrightarrow L^{2}(\Omega)^{n}$). The latter is
warranted for a bounded $\Omega$ for the homogeneous Dirichlet case
or an $\Omega$ satisfying suitable geometric requirements (see e.g.
\cite{Weck1994}). An example of a different type of nature is that
of Maxwell's equations:
\begin{example}
\label{ex:max_hom}Recall Maxwell's equation as introduced in Example
\ref{ex:Maxwell}:
\[
\left(\partial_{0}\left(\begin{array}{cc}
\varepsilon & 0\\
0 & \mu
\end{array}\right)+\left(\begin{array}{cc}
\sigma & 0\\
0 & 0
\end{array}\right)+\left(\begin{array}{cc}
0 & -\curl\\
\curl_{c} & 0
\end{array}\right)\right)\left(\begin{array}{c}
E\\
H
\end{array}\right)=\left(\begin{array}{c}
J\\
0
\end{array}\right).
\]
In this case, we also want to consider sequences $(\varepsilon_{n})_{n},(\mu_{n})_{n},(\sigma_{n})_{n}$
and corresponding solutions $(E_{n},H_{n})_{n}$. In any case the
nullspaces of both $\curl_{c}$ and $\curl$ are infinite-dimensional.
Thus, the projection mechanism introduced above for the heat and the
elliptic equation cannot apply in the same manner. Moreover, considering
the Maxwell's equations on the nullspace of $\left(\begin{array}{cc}
0 & -\curl\\
\curl_{c} & 0
\end{array}\right)$, we realize that the equation amounts to be an \emph{ordinary} differential
equation in an \emph{infinite-dimensional} state space. For the latter
we have not stated any homogenization or continuous dependence result
yet. Thus, before dealing with Maxwell's equations in full generality,
we focus on ordinary (integro-)differential equations next.\end{example}
\begin{thm}[{\cite[Theorem 4.4]{Waurick2012}}]
\label{thm:ode_degen}Let $\nu,\epsilon\in\mathbb{R}_{>0}$, $r>\frac{1}{2\nu}$,
$(M_{n})_{n}$ in $\mathbb{\mathcal{H}}^{\infty,c}(B(r,r),L(H))\cap\mathbb{\mathcal{H}}^{\infty}(B(0,\epsilon),L(H))$
bounded, $H$ separable Hilbert space. Assume that\footnote{Note that $M_n\in \mathbb{\mathcal{H}}^{\infty,c}(B(r,r),L(H))\cap\mathbb{\mathcal{H}}^{\infty}(B(0,\epsilon),L(H))$ implies that $M_n(0)$ is selfadjoint.} 
\begin{align*}
M_{n}(0) & \geq c\text{ on }R(M_{n}(0))=R(M_{1}(0))
\end{align*}
for all $n\in\mathbb{N}$. Then there exists a subsequence $(n_{k})_{k}$
of $(n)_{n}$ and some $M\in\mathcal{H}^{\infty}$ such that
\[
\left(\partial_{0}M_{n_{k}}(\partial_{0}^{-1})\right)^{-1}\to\left(\partial_{0}M(\partial_{0}^{-1})\right)^{-1}
\]
in the weak operator topology.\end{thm}
\begin{rem}
Note that in the latter theorem, in general, the sequence $\left(M_{n_{k}}(\partial_{0}^{-1})\right)_{k}$
does \emph{not }converge to $M(\partial_{0}^{-1})$. The reason for
that is that the computation of the inverse is not continuous in the
weak operator topology. So, even if one chose a further subsequence
$(n_{k_{l}})_{l}$ of $(n_{k})_{k}$ such that $\left(M_{n_{k_{l}}}(\partial_{0}^{-1})\right)_{l}$
converges in the weak operator topology, then, in general, 
\[
M_{n_{k_{l}}}(\partial_{0}^{-1})\nrightarrow M(\partial_{0}^{-1})
\]
 in $\tau_{\textnormal{w}}$. Indeed, the latter can be seen by considering
the periodic extensions of the mappings $a^{1},a^{2}$ to all of $\mathbb{R}$
with 
\[
a^{1}(x)\coloneqq\begin{cases}
\frac{1}{2}, & 0\leq x<\frac{1}{2},\\
1, & \frac{1}{2}\leq x<1,
\end{cases}\quad a^{2}(x)\coloneqq\frac{3}{4},\quad(x\in[0,1]).
\]
We let $a_{n}\coloneqq a^{1}(n\cdot)$ for odd $n\in\mathbb{N}$ and
$a_{n}\coloneqq a^{2}(n\cdot)$ if $n\in\mathbb{N}$ is even. Then,
by Proposition \ref{prop:per_impl_conv}, we conclude that $a_{n}\to\frac{3}{4}$,
$a_{2n+1}^{-1}\to\frac{3}{2},$ and $a_{2n}^{-1}\to\frac{4}{3}$ as $n\to\infty$
in $\sigma(L^{\infty},L^{1})$. 
\end{rem}
In a way complementary to the latter theorem is the following. The latter theorem assumes analyticity of the $M_{n}$'s at $0$.
But the zero'th order term in the power series expansion of the $M_{n}$'s
may be non-invertible. In the next theorem, the analyticity at $0$
is not assumed any more. The (uniform) positive definiteness condition,
however, is more restrictive.
\begin{thm}[{\cite[Theorem 5.2]{Waurick}}]
\label{thm:ode_non_degen}Let $\nu,\epsilon\in\mathbb{R}_{>0}$,
$r>\frac{1}{2\nu},$ $(M_{n})_{n}$ in $\mathbb{\mathcal{H}}^{\infty,c}(B(r,r),L(H))$
bounded, $H$ separable Hilbert space. Assume that 
\[
\Re M_{n}(z)\geq c\quad(z\in B(r,r))
\]
for all $n\in\mathbb{N}$. Then there exists a subsequence $(M_{n_{k}})_{k}$
of $(M_{n})_{n}$ and some $M\in\mathcal{H}^{\infty}$ such that
\[
\left(\partial_{0}M_{n_{k}}(\partial_{0}^{-1})\right)^{-1}\to\left(\partial_{0}M(\partial_{0}^{-1})\right)^{-1}
\]
in the weak operator topology.
\end{thm}
Now, we turn to more concrete examples. With the methods developed,
we can characterize the convergence of a particular ordinary equation.
In a slightly more restrictive context these types of equations have
been discussed by Tartar in 1989 (see \cite{TartarMemEff,TartarNonlHom})
using the notion of Young-measures, see also the discussion in \cite[Remark 3.8]{Waurick2013a}. 
\begin{prop}[]
Let $(a_{n})_{n}$ in $L(H)$ be bounded, $H$ a separable Hilbert
space, $\nu>2\sup_{n\in\mathbb{N}}\|a_{n}\|+1$. Then
\[
\left(\left(\partial_{0}+a_{n}\right)^{-1}\right)_{n}
\]
converges in the weak operator topology if and only if for all $\ell\in\mathbb{N}$
\[
\left(a_{n}^{\ell}\right)_{n}
\]
converges in the weak operator topology to some $b_{\ell}\in L(H)$.
In the latter case $\left(\left(\partial_{0}+a_{n}\right)^{-1}\right)_{n}$
converges to 
\[
\left(\partial_{0}+\partial_{0}\sum_{j=1}^{\infty}\left(-\sum_{\ell=1}^{\infty}\left(-\partial_{0}^{-1}\right)^{\ell}b_{\ell}\right)^{j}\right)^{-1}
\]
in the weak operator topology.\end{prop}
\begin{proof}
The 'if'-part is a straightforward application of a Neumann series
expansion of $\left(\partial_{0}+a_{n}\right)^{-1}$, see e.g. \cite[Theorem 2.1]{Waurick2011a}.
The 'only-if'-part follows from the representation
\[
\left(\partial_{0}+a_{n}\right)^{-1}=\sum_{j=0}^{\infty}\left(-\partial_{0}^{-1}\right)^{j}a_{n}^{j}\partial_{0}^{-1}\eqqcolon M_{n}(\partial_{0}^{-1})\quad(n\in\mathbb{N}),
\]
the application of the Fourier-Laplace transform and Cauchy's integral
formulas for the derivatives of holomorphic functions. For the latter
argument note that $(M_{n})_{n}$ is a bounded sequence in $\mathcal{H}^{\infty}(B(r,r),L(H))$
for $r>\frac{1}{2\nu}$ and, thus, contains a $\mathcal{H}_{\textnormal{w }}^{\infty}$-convergent
subsequence, whose limit $M$ satisfies $M(\partial_{0}^{-1})=\mathbb{\tau}_{\textnormal{w}}-\lim_{n\to\infty}\left(\partial_{0}+a_{n}\right)^{-1}$. 
\end{proof}
One might wonder under which circumstances the conditions in the latter
theorem happen to be satisfied. We discuss the following example initially
studied by Tartar.
\begin{example}[Ordinary differential equations]
 Let $a\in L^{\infty}(\mathbb{R})$. If $a$ is $1$-periodic then
$a(n\cdot)$ converges to $\int_{0}^{1}a$ in the $\sigma(L^{\infty},L^{1})$-topology.
Regard $a$ as a multiplication operator $a(\mathrm{m}_1)$ on $L^{2}(\mathbb{R})$.
Now, we have the explicit formula
\[
\left(\partial_{0}+a(n\mathrm{m}_1)\right)^{-1}\stackrel{\tau_{\textnormal{w}}}{\to}\left(\partial_{0}+\partial_{0}\sum_{j=1}^{\infty}\left(-\sum_{\ell=1}^{\infty}\left(-\partial_{0}^{-1}\right)^{\ell}\int_{0}^{1}a^{\ell}\right)^{j}\right)^{-1}.
\]
We should remark here that the classical approach to this problem
uses the theory of Young-measures to express the limit equation. This
is not needed in our approach.
\end{example}
With the latter example in mind, we now turn to the discussion of
a general theorem also working for Maxwell's equation. As mentioned
above, these equations can be reduced to the cases of Theorem \ref{thm:Hom_1_auto_comp}
and \ref{thm:ode_degen}. Consequently, the limit equations become
more involved. For sake of this presentation, we do not state the
explicit formulae for the limit expressions and instead refer to \cite[Corollary 4.7]{Waurick2012}.
\begin{thm}[{\cite[Corollary 4.7]{Waurick2012}}]
\label{thm:hom_gen_inf_null}Let $\nu,\epsilon\in\mathbb{R}_{>0}$,
$r>\frac{1}{2\nu}$, $(M_{n})_{n}$ in $\mathbb{\mathcal{H}}^{\infty,c}(B(r,r),L(H))\cap\mathbb{\mathcal{H}}^{\infty}(B(0,\epsilon),L(H))$
bounded, $A\colon D(A)\subseteq H\to H$ skew-selfadjoint, $H$ separable.
Assume that $D_{A}\cap N(A)^{\bot}\hookrightarrow\hookrightarrow H$
and, in addition,
\begin{align*}
M_{n}(0) & \geq c\text{ on }R(M_{n}(0))=R(M_{1}(0))\\
\iota_{N(A)^\bot}^\ast M_{n}^{'}(0)\iota_{N(A)}\left(\iota_{N(A)}^\ast M_{n}^{'}(0)\iota_{N(A)}\right)^{-1} & =\iota_{N(A)^\bot}^\ast M_{n}^{'}(0)^{*}\iota_{N(A)}\left(\iota_{N(A)}^\ast M_{n}^{'}(0)^{*}\iota_{N(A)}\right)^{-1},
\end{align*}
for all $n\in\mathbb{N}$, where $\iota_{N(A)^\bot} \colon  N(M_{1}(0))\cap N(A)^{\bot} \to H$,
$\iota_{N(A)} \colon  N(M_{1}(0))\cap N(A) \to H $ denote the canonical embeddings. Then there exists a subsequence $(M_{n_{k}})_{k}$ of
$(M_{n})_{n}$ such that 
\[
\left(\overline{\partial_{0}M_{n_{k}}(\partial_{0}^{-1})+A}\right)^{-1}\to\left(\overline{\partial_{0}M(\partial_{0}^{-1})+A}\right)^{-1}
\]
converges in the weak operator topology.\end{thm}
\begin{rem}
It should be noted that, similarly to the case of ordinary differential
equations, in general, we do \emph{not }have $M_{n_{k}}(\partial_{0}^{-1})\to M(\partial_{0}^{-1})$
in the weak operator topology. 
\end{rem}
Before we discuss possible generalizations of the above results to
the non-autonomous case, we illustrate the applicability of Theorem
\ref{thm:hom_gen_inf_null} to Maxwell's equations:
\begin{example}[Example \ref{ex:max_hom} continued]
 Consider 
\[
\left(\partial_{0}\left(\begin{array}{cc}
\varepsilon_{n} & 0\\
0 & \mu_{n}
\end{array}\right)+\left(\begin{array}{cc}
\sigma_{n} & 0\\
0 & 0
\end{array}\right)+\left(\begin{array}{cc}
0 & -\curl\\
\curl_{c} & 0
\end{array}\right)\right)\left(\begin{array}{c}
E_{n}\\
H_{n}
\end{array}\right)=\left(\begin{array}{c}
J\\
0
\end{array}\right)
\]
for bounded sequences of bounded linear operators $(\varepsilon_{n})_{n},(\mu_{n})_{n},(\sigma_{n})_{n}$.
Assuming suitable geometric requirements on the underlying domain
$\Omega$, see e.g. \cite{Witsch1993}, we realize that the compactness
condition is satisfied. Thus, we only need to guarantee the compatibility
conditions: Essentially, there are two complementary cases. On the
one hand, one assumes uniform strict positive definiteness of the
(selfadjoint) operators $\left(\begin{array}{cc}
\varepsilon_{n} & 0\\
0 & \mu_{n}
\end{array}\right)$. On the other hand, we may also consider the eddy current problem,
which results in $\varepsilon_{n}=0$. Then, in order to apply Theorem
\ref{thm:hom_gen_inf_null}, we have to assume selfadjointness of
$\sigma_{n}$ and the existence of some $c>0$ such that $\sigma_{n}\geq c$
for all $n\in\mathbb{N}.$ In this respect our homogenization theorem
only works under additional assumptions on the material laws apart
from (uniform) well-posedness conditions. We also remark that the
limit equation is of integro-differential type, see \cite[Corollary 4.7]{Waurick2012}
or \cite{Well}.
\end{example}

\subsection{The non-autonomous case\label{sub:The-non-autonomous-case}}

The non-autonomous case is characterized by the fact that the operators
$\mathcal{M}$ and $\mathcal{A}$ in (\ref{eq:gen_evo}) does not have to commute with
the translation operators $\tau_{h}$. A rather general abstract result
concerning well-posedness reads as follows:
\begin{thm}[{\cite[Theorem 2.4]{Waurick2013_nonauto}}]
\label{thm:non-auto2}Let $\nu>0$ and $\mathcal{M},\mathcal{N}\in L(L_{\nu}^{2}(\mathbb{R},H))$.
Assume that there exists $M\in L(L_{\nu}^{2}(\mathbb{R},H))$ such
that 
\[
\mathcal{M}\partial_{0,\nu}\subseteq\partial_{0,\nu}\mathcal{M}+M.
\]
Let $\mathcal{A}\colon D(\mathcal{A})\subseteq L_{\nu}^{2}(\mathbb{R},H)\to L_{\nu}^{2}(\mathbb{R},H)$
be densely defined, closed, linear and such that $\partial_{0,\nu}^{-1}\mathcal{A}\subseteq\mathcal{A}\partial_{0,\nu}^{-1}$.
Assume there exists $c>0$ such that the positivity conditions 
\[
\Re\langle\left(\partial_{0,\nu}\mathcal{M}+\mathcal{N}+\mathcal{A}\right)\phi|\1_{_{\mathbb{R}_{\leq a}}}(\mathrm{m}_{0})\phi\rangle\geq c\langle\phi|\1_{_{\mathbb{R}_{\leq a}}}(\mathrm{m}_{0})\phi\rangle
\]
and 
\[
\Re\langle\left(\left(\partial_{0,\nu}\mathcal{M}+\mathcal{N}\right)^{*}+\mathcal{A}^{*}\right)\psi|\psi\rangle\geq c\langle\psi|\psi\rangle
\]
hold for all $a\in\mathbb{R}$, $\phi\in D(\partial_{0,\nu})\cap D(\mathcal{A})$,
$\psi\in D(\partial_{0,\nu})\cap D(\mathcal{A}^{*})$. Then $\mathcal{B}\coloneqq\overline{\partial_{0,\nu}\mathcal{M}+\mathcal{N}+\mathcal{A}}$
is continuously invertible, $\left\Vert \mathcal{B}^{-1}\right\Vert \leq\frac{1}{c}$,
and the operator $\mathcal{B}^{-1}$ is causal in $L_{\nu}^{2}(\mathbb{R},H)$.
\end{thm}
In order to capture the main idea of this general abstract result,
we consider the following special non-autonomous problem of the form
\begin{equation}
\left(\partial_{0,\nu}M_{0}(\mathrm{m}_{0})+M_{1}(\mathrm{m}_{0})+A\right)u=f,\label{eq:non_auto_skew}
\end{equation}
where $\partial_{0,\nu}$ denotes the time-derivative as introduced in
Subsection \ref{sub:The-Time-Derivative}, and $A$ denotes a skew-selfadjoint
operator on some Hilbert space $H$ (and its canonical extension to
the space $L_{\nu}^{2}(\mathbb{R},H)$). Moreover, $M_{0},M_{1}:\mathbb{R}\to L(H)$
are assumed to be strongly measurable and bounded (in symbols $M_{0},M_{1}\in L_{s}^{\infty}(\mathbb{R},L(H))$)
and therefore, they give rise to multiplication operators on $L_{\nu}^{2}(\mathbb{R},H)$
by setting 
\[
\left(M_{i}(\mathrm{m}_{0})u\right)(t)\coloneqq M_{i}(t)u(t)\quad(\mbox{a.e. }t\in\mathbb{R})
\]
for $u\in L_{\nu}^{2}(\mathbb{R},H)$, where $\nu\geq0$ and $i\in\{0,1\}.$
Of course, the so defined multiplication operators are bounded with
\[
\|M_{i}(\mathrm{m}_{0})\|_{L(L_{\nu}^{2}(\mathbb{R},H))}\leq|M_{i}|_{\infty}=\esssup_{t\in\mathbb{R}}\|M_{i}(t)\|_{L(H)}
\]
for $i\in\{0,1\}$ and $\nu\geq0$. In order to formulate the theorem
in a less cluttered way, we introduce the following hypotheses.

\begin{hyp} \label{hyp:nonauto_linear}We say that $T\in L_{s}^{\infty}(\mathbb{R},L(H))$
satisfies the property

\begin{enumerate}[(a)]

\item \label{selfadjoint} if $T(t)$ is selfadjoint $(t\in\mathbb{R})$,

\item \label{non-negative} if $T(t)$ is non-negative $(t\in\mathbb{R})$,

\item \label{Lipschitz} if the mapping $T$ is Lipschitz-continuous,
where we denote the smallest Lipschitz-constant of $T$ by $|T|_{\mathrm{Lip}}$,
and

\item \label{differentiable} if there exists a set $N\subseteq\mathbb{R}$
of measure zero such that for each $x\in H$ the function 
\[
\mathbb{R}\setminus N\ni t\mapsto T(t)x
\]
is differentiable%
\footnote{If $H$ is separable, then the strong differentiability of $T$ on
$\mathbb{R}\setminus N$ for some set $N$ of measure zero already
follows from the Lipschitz-continuity of $T$ by Rademachers theorem.%
}. 

\end{enumerate}

\end{hyp}

If $T\in L_{s}^{\infty}(\mathbb{R},H)$ satisfies the hypotheses above,
then for each $t\in\mathbb{R}\setminus N$ the operator 
\begin{align*}
\dot{T}(t):H & \to H\\
x & \mapsto\left(T(\cdot)x\right)'(t)
\end{align*}
becomes a selfadjoint linear operator satisfying $\|\dot{T}(t)\|_{L(H)}\leq|T|_{\mathrm{Lip}}$
for every $t\in\mathbb{R}\setminus N$ and consequently $\dot{T}\in L_{s}^{\infty}(\mathbb{R},L(H))$.
We are now able to state the well-posedness result for non-autonomous
problems of the form (\ref{eq:non_auto_skew}).
\begin{thm}[{\cite[Theorem 2.13]{Picard2013_nonauto}}]
\label{thm:Solutiontheory} Let $A\colon D(A)\subseteq H\to H$ be
skew-selfadjoint and $M_{0},M_{1}\in L_{s}^{\infty}(\mathbb{R},L(H)).$
Furthermore, assume that $M_{0}$ satisfies the hypotheses (\ref{selfadjoint})-(\ref{differentiable})
and that there exists a set $N_{1}\subseteq\mathbb{R}$ of measure
zero with $N\subseteq N_{1}$ such that
\begin{equation}
\bigvee_{c_{0}>0,\nu_{0}>0}\;\bigwedge_{t\in\mathbb{R}\setminus N_{1},\nu\geq\nu_{0}}:\nu M_{0}(t)+\frac{1}{2}\dot{M}_{0}(t)+\Re M_{1}(t)\geq c_{0}.\label{eq:pos_def}
\end{equation}
Then the operator $\overline{\partial_{0,\nu}M_{0}\left(\mathrm{m}_{0}\right)+M_{1}\left(\mathrm{m}_{0}\right)+A}$
is continuously invertible in $L_{\nu}^{2}(\mathbb{R},H)$ for each
$\nu\geq\nu_{0}$. A norm bound for the inverse is $1/c_{0}$. Moreover,
we get that 
\begin{equation}
\left(\partial_{0,\nu}M_{0}\left(\mathrm{m}_{0}\right)+M_{1}\left(\mathrm{m}_{0}\right)+A\right)^{*}=\left(\overline{M_{0}\left(\mathrm{m}_{0}\right)\partial_{0,\nu}^{*}+M_{1}\left(\mathrm{m}_{0}\right)^{*}-A}\right).\label{eq:adjoint}
\end{equation}
\end{thm}
\begin{proof}
The result can easily be established, when observing that $M_{0}(\mathrm{m}_{0})\partial_{0,\nu}\subseteq\partial_{0,\nu}M_{0}(\mathrm{m}_{0})+\dot{M}_{0}(\mathrm{m}_{0})$
and using Theorem \ref{thm:non-auto2}.
\end{proof}
Independently of Theorem \ref{thm:non-auto2}, note that condition
(\ref{eq:pos_def}) is an appropriate non-autonomous analogue of the
positive definiteness constraint (\ref{eq:pos_material_law}) in the
autonomous case. With the help of (\ref{eq:pos_def}) one can prove
that the operator $\partial_{0,\nu}M_{0}(\mathrm{m}_{0})+M_{1}(\mathrm{m}_{0})+A$
is strictly monotone and after establishing the equality (\ref{eq:adjoint}),
the same argumentation works for the adjoint. Hence, the well-posedness
result may also be regarded as a consequence of Corollary \ref{lin:op:theo}.
\begin{example}
\label{ex:illustrating_non_autp}As an illustrating example for the
applicability of Theorem \ref{thm:Solutiontheory} we consider a non-autonomous
evolutionary problem, which changes its type in time and space. Let
$\nu>0$. Consider the $\left(1+1\right)$-dimensional wave equation:
\begin{eqnarray*}
\partial_{0,\nu}^{2}u-\partial_{1}^{2}u & = & f\mbox{ on }\mathbb{R}\times\mathbb{R}.
\end{eqnarray*}
As usual we rewrite this equation as a first order system of the form
\begin{equation}
\left(\partial_{0,\nu}\left(\begin{array}{cc}
1 & 0\\
0 & 1
\end{array}\right)+\left(\begin{array}{cc}
0 & -\partial_{1}\\
-\partial_{1} & 0
\end{array}\right)\right)\left(\begin{array}{c}
u\\
v
\end{array}\right)=\left(\begin{array}{c}
\partial_{0,\nu}^{-1}f\\
0
\end{array}\right).\label{eq:wave-eq}
\end{equation}
In this case we can compute the solution by Duhamel's formula in terms
of the unitary group generated by the skew-selfadjoint operator 
\[
\left(\begin{array}{cc}
0 & -\partial_{1}\\
-\partial_{1} & 0
\end{array}\right).
\]
Let us now, based on this, consider a slightly more complicated situation,
which is, however, still autonomous:
\begin{align}
 & \left(\partial_{0,\nu}\left(\begin{array}{cc}
\X_{_{\mathbb{R}\setminus\,]-\varepsilon,0[}}(\mathrm{m}_{1}) & 0\\
0 & \X_{_{\mathbb{R}\setminus\,]-\varepsilon,\varepsilon[}}(\mathrm{m}_{1})
\end{array}\right)+\left(\begin{array}{cc}
\X_{_{]-\varepsilon,0[}}(\mathrm{m}_{1}) & 0\\
0 & \X_{_{]-\varepsilon,\varepsilon[}}(\mathrm{m}_{1})
\end{array}\right)+\left(\begin{array}{cc}
0 & -\partial_{1}\\
-\partial_{1} & 0
\end{array}\right)\right)\left(\begin{array}{c}
u\\
v
\end{array}\right)\nonumber \\
 & =\left(\begin{array}{c}
\partial_{0,\nu}^{-1}f\\
0
\end{array}\right),\label{eq:auto_ex}
\end{align}

\end{example}
where $\X_{_{I}}(\mathrm{m}_{1})$ denotes the spatial multiplication
operator with the cut-off function $\X_{_{I}}$, given by $\left(\X_{_{I}}(\mathrm{m}_{1})f\right)(t,x)=\X_{_{I}}(x)f(t,x)$
for almost every $(t,x)\in\mathbb{R}\times\mathbb{R}$, every $f\in L_{\nu}^{2}(\mathbb{R},L^{2}(\mathbb{R}))$
and $I\subseteq\mathbb{R}$. Hence, (\ref{eq:auto_ex}) is an equation
of the form (\ref{eq:non_auto_skew}) with 
\[
M_{0}\left(\mathrm{m}_{0}\right)\coloneqq\left(\begin{array}{cc}
\X_{_{\mathbb{R}\setminus\,]-\varepsilon,0[}}(\mathrm{m}_{1}) & 0\\
0 & \X_{_{\mathbb{R}\setminus\,]-\varepsilon,\varepsilon[}}(\mathrm{m}_{1})
\end{array}\right)
\]

and

\[
M_{1}\left(\mathrm{m}_{0}\right)\coloneqq\left(\begin{array}{cc}
\X_{_{]-\varepsilon,0[}}(\mathrm{m}_{1}) & 0\\
0 & \X_{_{]-\varepsilon,\varepsilon[}}(\mathrm{m}_{1})
\end{array}\right)
\]
and both are obviously not time-dependent. Note that our solution
condition (\ref{eq:pos_def}) is satisfied and hence, problem (\ref{eq:auto_ex})
is well-posed in the sense of Theorem \ref{thm:Solutiontheory}.%
\footnote{Indeed, the well-posedness already follows from Theorem \ref{thm:sol_theo_skew},
since $M$ is autonomous and satisfies (\ref{eq:pos_material_law}).%
} By the dependence of the operators $M_{0}(\mathrm{m}_{0})$ and $M_{1}(\mathrm{m}_{0})$
on the spatial parameter, we see that (\ref{eq:auto_ex}) changes
its type from hyperbolic to elliptic to parabolic and back to hyperbolic
and so standard semi-group techniques are not at hand to solve the
equation. Indeed, in the subregion $]-\varepsilon,0[$ the problem
reads as 

\[
\left(\begin{array}{c}
u\\
v
\end{array}\right)+\left(\begin{array}{cc}
0 & -\partial_{1}\\
-\partial_{1} & 0
\end{array}\right)\left(\begin{array}{c}
u\\
v
\end{array}\right)=\left(\begin{array}{c}
\partial_{0,\nu}^{-1}f\\
0
\end{array}\right),
\]

which may be rewritten as an elliptic equation for $u$ of the form
\[
u-\partial_{1}^{2}u=\partial_{0,\nu}^{-1}f.
\]

For the region $]0,\varepsilon[$ we get 
\[
\left(\partial_{0,\nu}\left(\begin{array}{cc}
1 & 0\\
0 & 0
\end{array}\right)+\left(\begin{array}{cc}
0 & 0\\
0 & 1
\end{array}\right)+\left(\begin{array}{cc}
0 & -\partial_{1}\\
-\partial_{1} & 0
\end{array}\right)\right)\left(\begin{array}{c}
u\\
v
\end{array}\right)=\left(\begin{array}{c}
\partial_{0,\nu}^{-1}f\\
0
\end{array}\right),
\]

which yields a parabolic equation for $u$ of the form 
\[
\partial_{0,\nu}u-\partial_{1}^{2}u=\partial_{0,\nu}^{-1}f.
\]

In the remaining sub-domain $\mathbb{R}\setminus\,]-\varepsilon,\varepsilon[$
the problem is of the original form (\ref{eq:wave-eq}), which corresponds
to a hyperbolic problem for $u$. \\
To turn this into a genuinely time-dependent problem we now make a
modification to problem (\ref{eq:auto_ex}). We define the function
\[
\varphi(t)\coloneqq\begin{cases}
0 & \mbox{ if }t\leq0,\\
t & \mbox{ if }0<t\leq1,\\
1 & \mbox{ if }1<t
\end{cases}\quad(t\in\mathbb{R})
\]
and consider the material-law operator 
\[
M_{0}\left(\mathrm{m}_{0}\right)=\varphi(\mathrm{m}_{0})\left(\begin{array}{cc}
\X_{_{\mathbb{R}\setminus\,]-\varepsilon,0[}}(\mathrm{m}_{1}) & 0\\
0 & \X_{_{\mathbb{R}\setminus\,]-\varepsilon,\varepsilon[}}(\mathrm{m}_{1})
\end{array}\right),
\]

which now also degenerates in time. Moreover we modify $M_{1}(\mathrm{m}_{0})$
by adding a time-dependence of the form 
\[
M_{1}(\mathrm{m}_{0})=\left(\begin{array}{cc}
\X_{_{]-\infty,0[}}(\mathrm{m}_{0})+\X_{_{[0,\infty[}}(\mathrm{m}_{0})\X_{_{]-\varepsilon,0[}}(\mathrm{m}_{1}) & 0\\
0 & \X_{_{]-\infty,0[}}(\mathrm{m}_{0})+\X_{_{[0,\infty[}}(\mathrm{m}_{0})\X_{_{]-\varepsilon,\varepsilon[}}(\mathrm{m}_{1})
\end{array}\right).
\]
We show that this time-dependent material law still satisfies our
solvability condition. Note that 
\[
\varphi'(t)=\begin{cases}
1 & \mbox{ if }t\in]0,1[,\\
0 & \mbox{ otherwise}
\end{cases}
\]

and thus, for $t\leq0$ we have 
\[
\nu M_{0}(t)+\frac{1}{2}\dot{M}_{0}(t)+\Re M_{1}(t)=\left(\begin{array}{cc}
1 & 0\\
0 & 1
\end{array}\right)\geq1.
\]

For $0<t\leq1$ we estimate 
\begin{align*}
 & \nu M_{0}(t)+\frac{1}{2}\dot{M}_{0}(t)+\Re M_{1}(t)\\
 & =\left(\frac{1}{2}+\nu t\right)\left(\begin{array}{cc}
\X_{_{\mathbb{R}\setminus\,]-\varepsilon,0[}}(\mathrm{m}_{1}) & 0\\
0 & \X_{_{\mathbb{R}\setminus\,]-\varepsilon,\varepsilon[}}(\mathrm{m}_{1})
\end{array}\right)+\left(\begin{array}{cc}
\X_{_{]-\varepsilon,0[}}(\mathrm{m}_{1}) & 0\\
0 & \X_{_{]-\varepsilon,\varepsilon[}}(\mathrm{m}_{1})
\end{array}\right)\geq\frac{1}{2}
\end{align*}

and, finally, for $t>1$ we obtain that 
\begin{align*}
 & \nu M_{0}(t)+\frac{1}{2}\dot{M}_{0}(t)+\Re M_{1}(t)\\
 & =\nu\left(\begin{array}{cc}
\X_{_{\mathbb{R}\setminus\,]-\varepsilon,0[}}(\mathrm{m}_{1}) & 0\\
0 & \X_{_{\mathbb{R}\setminus\,]-\varepsilon,\varepsilon[}}(\mathrm{m}_{1})
\end{array}\right)+\left(\begin{array}{cc}
\X_{_{]-\varepsilon,0[}}(\mathrm{m}_{1}) & 0\\
0 & \X_{_{]-\varepsilon,\varepsilon[}}(\mathrm{m}_{1})
\end{array}\right)\geq\min\{\nu,1\}.
\end{align*}

There is also an adapted result on the closedness of the problem class
for the non-autonomous situation. The case $\mathcal{A}=0$ is thoroughly
discussed in \cite{Waurick2013a}. We give the corresponding result
for the situation where $\mathcal{A}$ is non-zero and satisfies a
certain compactness condition.
\begin{thm}[{\cite[Theorem 3.1]{Waurick2013_nonauto_homo}}]
\label{thm:general_cont_depend}Let $\nu>0$. Let $\left(\mathcal{M}_{n}\right)_{n}$
be a bounded sequence in $L(L_{\nu}^{2}(\mathbb{R},H))$ such that
$\left(\left[\mathcal{M}_{n},\partial_{0,\nu}\right]\right)_{n}$
is bounded in $L\!\left(L_{\nu}^{2}(\mathbb{R},H)\right)$. Moreover,
let $\mathcal{A}\colon D(\mathcal{A})\subseteq L_{\nu}^{2}(\mathbb{R},H)\to L_{\nu}^{2}(\mathbb{R},H)$
be linear and maximal monotone commuting with $\partial_{0,\nu}$
and assume that $\mathcal{M}_{n}$ is causal for each $n\in\mathbb{N}$.
Moreover, assume the positive definiteness conditions 
\begin{equation}
\Re\left\langle \partial_{0,\nu}\mathcal{M}_{n}u|\1_{_{\mathbb{R}_{\leq a}}}(\mathrm{m}_{0})u\right\rangle \geq c\left\langle u|\1_{_{\mathbb{R}_{\leq a}}}(\mathrm{m}_{0})u\right\rangle ,\quad\left\langle \mathcal{A}u|\1_{_{\mathbb{R}_{\leq0}}}(\mathrm{m}_{0})u\right\rangle \geq0\label{eq:truncated_pos_def}
\end{equation}
for all $u\in D(\partial_{0,\nu})\cap D(\mathcal{A})$, $a\in\mathbb{R}$,
$n\in\mathbb{N}$ and some $c>0$.

Assume that there exists a Hilbert space $K$ such that $K\hookrightarrow\hookrightarrow H$
and $D_{\mathcal{A}}\hookrightarrow L_{\nu}^{2}(\mathbb{R},K)$ and
that $\left(\mathcal{M}_{n}\right)_{n}$ converges in the weak operator
topology to some $\mathcal{M}$.

Then $\partial_{0,\nu}\mathcal{M}+\mathcal{A}$ is continuously invertible
in $L_{\nu}^{2}(\mathbb{R},H)$ and $\left(\overline{\partial_{0,\nu}\mathcal{M}_{n}+\mathcal{A}}\right)^{-1}\to\left(\overline{\partial_{0,\nu}\mathcal{M}+\mathcal{A}}\right)^{-1}$
in the weak operator topology of $L_{\nu}^{2}(\mathbb{R},H)$ as $n\to\infty$.
\end{thm}
As in \cite{Waurick2013_nonauto_homo}, we illustrate the latter theorem
by the following example, being an adapted version of Example \ref{ex:illustrating_non_autp}.
\begin{example}[{\cite[Section 1]{Waurick2013_nonauto_homo}}]
 Recalling the definition of $\partial_{1},\,\partial_{1,c}$ on $L^2([0,1])$ from Definition \ref{def:div_grad}, we
treat the following system written in block operator matrix form:
\begin{multline}
\left(\partial_{0,\nu}\begin{pmatrix}\1_{_{[0,\frac{1}{4}]\cup[\frac{1}{2},\frac{3}{4}]}}(\mathrm{m}_{1}) & 0\\
0 & \1_{_{[0,\frac{1}{4}]\cup[\frac{3}{4},1]}}(\mathrm{m}_{1})
\end{pmatrix}+\begin{pmatrix}\1_{_{[\frac{1}{4},\frac{1}{2}]\cup[\frac{3}{4},1]}}(\mathrm{m}_{1}) & 0\\
0 & \1_{_{[\frac{1}{4},\frac{1}{2}]\cup[\frac{1}{2},\frac{3}{4}]}}(\mathrm{m}_{1})
\end{pmatrix}\right.\\
\left.+\begin{pmatrix}0 & \partial_{1}\\
\partial_{1,c} & 0
\end{pmatrix}\right)\begin{pmatrix}u\\
v
\end{pmatrix}=\begin{pmatrix}f\\
g
\end{pmatrix}\label{eq:mixed_type}
\end{multline}
 where $f,g$ are thought of being given. We find that $\mathcal{M}$
is given by 
\[
\mathcal{M}=\begin{pmatrix}\1_{_{[0,\frac{1}{4}]\cup[\frac{1}{2},\frac{3}{4}]}}(\mathrm{m}_{1}) & 0\\
0 & \1_{_{[0,\frac{1}{4}]\cup[\frac{3}{4},1]}}(\mathrm{m}_{1})
\end{pmatrix}+\partial_{0}^{-1}\begin{pmatrix}\1_{_{[\frac{1}{4},\frac{1}{2}]\cup[\frac{3}{4},1]}}(\mathrm{m}_{1}) & 0\\
0 & \1_{_{[\frac{1}{4},\frac{1}{2}]\cup[\frac{1}{2},\frac{3}{4}]}}(\mathrm{m}_{1})
\end{pmatrix}.
\]
 We realize that 
\[
\mathcal{A}=\begin{pmatrix}0 & \partial_{1}\\
\partial_{1,c} & 0
\end{pmatrix}
\]
is skew-selfadjoint and, thus, maximal monotone. Note that the system
describes a mixed type equation. The system varies between hyperbolic,
elliptic and parabolic type equations either with homogeneous Dirichlet
or Neumann data. Well-posedness of the system (\ref{eq:mixed_type})
can be established in $L_{\nu}^{2}(\mathbb{R},L^{2}(0,1))$.

Now, instead of (\ref{eq:mixed_type}), we consider the sequence of
problems 
\begin{multline}
\left(\partial_{0,\nu}\begin{pmatrix}\1_{_{[0,\frac{1}{4}]\cup[\frac{1}{2},\frac{3}{4}]}}(n\cdot\mathrm{m}_{1}\!\!\!\mod1) & 0\\
0 & \1_{_{[0,\frac{1}{4}]\cup[\frac{3}{4},1]}}(n\cdot\mathrm{m}_{1}\!\!\!\mod1)
\end{pmatrix}\right.\\
+\left.\begin{pmatrix}\1_{_{[\frac{1}{4},\frac{1}{2}]\cup[\frac{3}{4},1]}}(n\cdot\mathrm{m}_{1}\!\!\!\mod1) & 0\\
0 & \1_{_{[\frac{1}{4},\frac{1}{2}]\cup[\frac{1}{2},\frac{3}{4}]}}(n\cdot\mathrm{m}_{1}\!\!\!\mod1)
\end{pmatrix}\right.\\
\left.+\begin{pmatrix}0 & \partial_{1}\\
\partial_{1,c} & 0
\end{pmatrix}\right)\begin{pmatrix}u_{n}\\
v_{n}
\end{pmatrix}=\begin{pmatrix}f\\
g
\end{pmatrix}\label{eq:M_n}
\end{multline}
 for $n\in\mathbb{N}$, where $x\!\!\!\mod1\coloneqq x-\lfloor x\rfloor$,
$x\in\mathbb{R}$. With the same arguments from above well-posedness
of the latter equation is warranted in the space $L_{\nu}^{2}(\mathbb{R},L^{2}(\oi01))$.
Now, 
\begin{multline*}
\begin{pmatrix}\1_{_{[0,\frac{1}{4}]\cup[\frac{1}{2},\frac{3}{4}]}}(n\cdot\mathrm{m}_{1}\!\!\!\mod1) & 0\\
0 & \1_{_{[0,\frac{1}{4}]\cup[\frac{3}{4},1]}}(n\cdot\mathrm{m}_{1}\!\!\!\mod1)
\end{pmatrix}\\
+\partial_{0,\nu}^{-1}\left.\begin{pmatrix}\1_{_{[\frac{1}{4},\frac{1}{2}]\cup[\frac{3}{4},1]}}(n\cdot\mathrm{m}_{1}\!\!\!\mod1) & 0\\
0 & \1_{_{[\frac{1}{4},\frac{1}{2}]\cup[\frac{1}{2},\frac{3}{4}]}}(n\cdot\mathrm{m}_{1}\!\!\!\mod1)
\end{pmatrix}\right.\\
\to\begin{pmatrix}\frac{1}{2} & 0\\
0 & \frac{1}{2}
\end{pmatrix}+\partial_{0,\nu}^{-1}\begin{pmatrix}\frac{1}{2} & 0\\
0 & \frac{1}{2}
\end{pmatrix}
\end{multline*}
 in the weak operator topology due to periodicity. Theorem \ref{thm:general_cont_depend}
asserts that the sequence $\begin{pmatrix}u_{n}\\
v_{n}
\end{pmatrix}_{n}$ weakly converges to the solution $\begin{pmatrix}u\\
v
\end{pmatrix}$ of the problem 
\[
\left(\partial_{0,\nu}\begin{pmatrix}\frac{1}{2} & 0\\
0 & \frac{1}{2}
\end{pmatrix}+\begin{pmatrix}\frac{1}{2} & 0\\
0 & \frac{1}{2}
\end{pmatrix}+\begin{pmatrix}0 & \partial_{1}\\
\partial_{1,c} & 0
\end{pmatrix}\right)\begin{pmatrix}u\\
v
\end{pmatrix}=\begin{pmatrix}f\\
g
\end{pmatrix}.
\]
 It is interesting to note that the latter system does not coincide
with any of the equations discussed above.

Theorem \ref{thm:general_cont_depend} deals with coefficients $\mathbb{\mathcal{M}}$
that live in space-time. Going a step further instead of treating
(\ref{eq:M_n}), we let $(\kappa_{n})_{n}$ in $W_{1}^{1}(\mathbb{R})$
be a $W_{1}^{1}(\mathbb{R})$-convergent sequence of weakly differentiable
$L^{1}(\mathbb{R})$-functions with limit $\kappa$ and support on
the positive reals. Then it is easy to see that the associated convolution
operators $(\kappa_{n}*)_{n}$ converge in $L(L^{2}(\mathbb{R}_{\geq0}))$
to $\kappa*$. Moreover, using Young's inequality, we deduce that
\[
\|\kappa_{n}\ast\|_{L(L_{\nu}^{2}(\mathbb{R}))},\|\kappa_{n}'\ast\|_{L(L_{\nu}^{2}(\mathbb{R}))}\to0\quad(\nu\to\infty)
\]
 uniformly in $n$. Thus, the strict positive definiteness of 
\[
\partial_{0,\nu}(1+\kappa*)\left(\begin{pmatrix}\1_{_{[0,\frac{1}{4}]\cup[\frac{1}{2},\frac{3}{4}]}}(\mathrm{m}_{1}) & 0\\
0 & \1_{_{[0,\frac{1}{4}]\cup[\frac{3}{4},1]}}(\mathrm{m}_{1})
\end{pmatrix}+\partial_{0,\nu}^{-1}\begin{pmatrix}\1_{_{[\frac{1}{4},\frac{1}{2}]\cup[\frac{3}{4},1]}}(\mathrm{m}_{1}) & 0\\
0 & \1_{_{[\frac{1}{4},\frac{1}{2}]\cup[\frac{1}{2},\frac{3}{4}]}}(\mathrm{m}_{1})
\end{pmatrix}\right)
\]
in the truncated form as in (\ref{eq:truncated_pos_def}) in Theorem
\ref{thm:general_cont_depend} above follows from the respective inequality
for 
\[
\partial_{0,\nu}\begin{pmatrix}\1_{_{[0,\frac{1}{4}]\cup[\frac{1}{2},\frac{3}{4}]}}(\mathrm{m}_{1}) & 0\\
0 & \1_{_{[0,\frac{1}{4}]\cup[\frac{3}{4},1]}}(\mathrm{m}_{1})
\end{pmatrix}+\begin{pmatrix}\1_{_{[\frac{1}{4},\frac{1}{2}]\cup[\frac{3}{4},1]}}(\mathrm{m}_{1}) & 0\\
0 & \1_{_{[\frac{1}{4},\frac{1}{2}]\cup[\frac{1}{2},\frac{3}{4}]}}(\mathrm{m}_{1})
\end{pmatrix}.
\]
Now, the product of a sequence converging in the weak operator topology
and a sequence converging in the norm topology converges in the weak
operator topology. Hence, the solutions of 
\begin{multline*}
\left(\partial_{0,\nu}\left(1+\kappa_{n}*\right)\left(\begin{pmatrix}\1_{_{[0,\frac{1}{4}]\cup[\frac{1}{2},\frac{3}{4}]}}(n\cdot\mathrm{m}_{1}\!\!\!\mod1) & 0\\
0 & \1_{_{[0,\frac{1}{4}]\cup[\frac{3}{4},1]}}(n\cdot\mathrm{m}_{1}\!\!\!\mod1)
\end{pmatrix}\right.\right.\\
+\left.\left.\partial_{0,\nu}^{-1}\begin{pmatrix}\1_{_{[\frac{1}{4},\frac{1}{2}]\cup[\frac{3}{4},1]}}(n\cdot\mathrm{m}_{1}\!\!\!\mod1) & 0\\
0 & \1_{_{[\frac{1}{4},\frac{1}{2}]\cup[\frac{1}{2},\frac{3}{4}]}}(n\cdot\mathrm{m}_{1}\!\!\!\mod1)
\end{pmatrix}\right.\right)\\
\left.+\begin{pmatrix}0 & \partial_{1}\\
\partial_{1,c} & 0
\end{pmatrix}\right)\begin{pmatrix}u_{n}\\
v_{n}
\end{pmatrix}=\begin{pmatrix}f\\
g
\end{pmatrix}
\end{multline*}
 converge weakly to the solution of 
\[
\left(\partial_{0,\nu}\left(1+\kappa*\right)\left(\begin{pmatrix}\frac{1}{2} & 0\\
0 & \frac{1}{2}
\end{pmatrix}+\partial_{0,\nu}^{-1}\begin{pmatrix}\frac{1}{2} & 0\\
0 & \frac{1}{2}
\end{pmatrix}\right)+\begin{pmatrix}0 & \partial_{1}\\
\partial_{1,c} & 0
\end{pmatrix}\right)\begin{pmatrix}u\\
v
\end{pmatrix}=\begin{pmatrix}f\\
g
\end{pmatrix}.
\]
 The latter considerations dealt with time-translation invariant coefficients.
We shall also treat another example, where time-translation invariance
is not warranted. For this take a sequence of Lipschitz continuous
functions $(N_{n}\colon\mathbb{R}\to\mathbb{R})_{n}$ with uniformly
bounded Lipschitz semi-norm and such that $(N_{n})_{n}$ converges
point-wise almost everywhere to some function $N\colon\mathbb{R}\to\mathbb{R}$.
Moreover, assume that there exists $c>0$ such that $\frac{1}{c}\geq N_{n}\geq c$
for all $n\in\mathbb{N}$. Then, by Lebesgue's dominated convergence
theorem $N_{n}(\mathrm{m}_{0})\to N(\mathrm{m}_{0})$ in the strong
operator topology, where we anticipated that $N_{n}(\mathrm{m}_{0})$
acts as a multiplication operator with respect to the temporal variable.
The strict monotonicity in the above truncated sense of 
\begin{multline*}
\partial_{0,\nu}\left(N_{n}(\mathrm{m}_{0})\begin{pmatrix}\1_{_{[0,\frac{1}{4}]\cup[\frac{1}{2},\frac{3}{4}]}}(n\cdot\mathrm{m}_{1}\!\!\!\mod1) & 0\\
0 & \1_{_{[0,\frac{1}{4}]\cup[\frac{3}{4},1]}}(n\cdot\mathrm{m}_{1}\!\!\!\mod1)
\end{pmatrix}\right.\\
+\left.\partial_{0,\nu}^{-1}\begin{pmatrix}\1_{_{[\frac{1}{4},\frac{1}{2}]\cup[\frac{3}{4},1]}}(n\cdot\mathrm{m}_{1}\!\!\!\mod1) & 0\\
0 & \1_{_{[\frac{1}{4},\frac{1}{2}]\cup[\frac{1}{2},\frac{3}{4}]}}(n\cdot\mathrm{m}_{1}\!\!\!\mod1)
\end{pmatrix}\right)
\end{multline*}
 is easily seen using integration by parts, see e.g. \cite[Lemma 2.6]{Picard2013_nonauto}.
Our main convergence theorem now yields that the solutions of 
\begin{multline*}
\left(\partial_{0,\nu}\left(N_{n}(\mathrm{m}_{0})\begin{pmatrix}\1_{_{[0,\frac{1}{4}]\cup[\frac{1}{2},\frac{3}{4}]}}(n\cdot\mathrm{m}_{1}\!\!\!\mod1) & 0\\
0 & \1_{_{[0,\frac{1}{4}]\cup[\frac{3}{4},1]}}(n\cdot\mathrm{m}_{1}\!\!\!\mod1)
\end{pmatrix}\right.\right.\\
+\left.\left.\partial_{0,\nu}^{-1}\begin{pmatrix}\1_{_{[\frac{1}{4},\frac{1}{2}]\cup[\frac{3}{4},1]}}(n\cdot\mathrm{m}_{1}\!\!\!\mod1) & 0\\
0 & \1_{_{[\frac{1}{4},\frac{1}{2}]\cup[\frac{1}{2},\frac{3}{4}]}}(n\cdot\mathrm{m}_{1}\!\!\!\mod1)
\end{pmatrix}\right.\right)\\
\left.+\begin{pmatrix}0 & \partial_{1}\\
\partial_{1,c} & 0
\end{pmatrix}\right)\begin{pmatrix}u_{n}\\
v_{n}
\end{pmatrix}=\begin{pmatrix}f\\
g
\end{pmatrix}
\end{multline*}
 converge weakly to the solution of 
\[
\left(\partial_{0,\nu}\left(N(\mathrm{m}_{0})\begin{pmatrix}\frac{1}{2} & 0\\
0 & \frac{1}{2}
\end{pmatrix}+\partial_{0,\nu}^{-1}\begin{pmatrix}\frac{1}{2} & 0\\
0 & \frac{1}{2}
\end{pmatrix}\right)+\begin{pmatrix}0 & \partial_{1}\\
\partial_{1,c} & 0
\end{pmatrix}\right)\begin{pmatrix}u\\
v
\end{pmatrix}=\begin{pmatrix}f\\
g
\end{pmatrix}.
\]
 
\end{example}

\section{Nonlinear Monotone Evolutionary Problems\label{sec:Monotone-Evolutionary-Problems}}

This last section is devoted to the generalization of the well-posedness
results of the previous sections to a particular case of non-linear
problems. Instead of considering differential equations we turn our
attention to the study of differential inclusions. As in the previous
section, we begin to consider the autonomous case and present the
well-posedness result.

\subsection{The autonomous case\label{sub:incl_auto}}

Let $\nu>0$. The problem class under consideration is given as follows
\begin{equation}
(u,f)\in\partial_{0,\nu}M\left(\partial_{0,\nu}^{-1}\right)+A,\label{eq:evol_incl}
\end{equation}
where $M(\partial_{0,\nu}^{-1})$ is again a linear material law,
arising from an analytic and bounded function $M:B_{\mathbb{C}}\left(r,r\right)\to L(H)$
for some $r>\frac{1}{2\nu}$, $f\in L_{\nu}^{2}(\mathbb{R},H)$ is
a given right-hand side and $u\in L_{\nu}^{2}(\mathbb{R},H)$ is to
be determined. In contrast to the above problems, $A\subseteq L_{\nu}^{2}(\mathbb{R},H)\oplus L_{\nu}^{2}(\mathbb{R},H)$
is now a maximal monotone relation, which in particular need not to be linear.
By this lack of linearity we cannot argue as in the previous section,
where the maximal monotonicity of the operators were shown by proving
the strict monotonicity of their adjoints (in other words, we cannot
apply Corollary \ref{lin:op:theo}). Thus, the maximal monotonicity
has to be shown by employing other techniques and the key tools are
perturbation results for maximal monotone operators. 

In the autonomous case, our hypotheses read as follows: 

\begin{hyp} We say that $A$ satisfies the hypotheses (H1) and (H2)
respectively, if 

\begin{enumerate}[(H1)]

\item $A$ is maximal monotone and \emph{translation-invariant},
i.e. for every $h\in\mathbb{R}$ and $(u,v)\in A$ we have $\left(u(\cdot+h),v(\cdot+h)\right)\in A.$

\item for all $(u,v),(x,y)\in A$ the estimate $\intop_{-\infty}^{0}\Re\langle u(t)-x(t)|v(t)-y(t)\rangle e^{-2\nu t}\mbox{d}t\geq0$
holds.

\end{enumerate}

\end{hyp}

Assuming the standard assumption (\ref{eq:pos_material_law}) for
the function $M$, the operator $\partial_{0,\nu}M\left(\partial_{0,\nu}^{-1}\right)-c$
is maximal monotone on $L_{\nu}^{2}(\mathbb{R},H).$ Thus, the well-posedness
of (\ref{eq:evol_incl}) just relies on the maximal monotonicity of
the sum of $\partial_{0,\nu}M\left(\partial_{0,\nu}^{-1}\right)-c$
and $A$. Since $A$ is assumed to be maximal monotone, we can apply
well-known perturbation results in the theory of maximal monotone
operators to prove that $\overline{\partial_{0,\nu}M\left(\partial_{0,\nu}^{-1}\right)+A-c}$
is indeed maximal monotone, which in particular yields that 
\[
\left(\overline{\partial_{0,\nu}M\left(\partial_{0,\nu}^{-1}\right)+A}\right)^{-1}
\]
 is a Lipschitz-continuous mapping on $L_{\nu}^{2}(\mathbb{R},H)$
(see Theorem \ref{thm:Minty}). Moreover, using hypothesis (H2) we
can prove the causality of the corresponding solution operator $\left(\overline{\partial_{0,\nu}M\left(\partial_{0,\nu}^{-1}\right)+A}\right)^{-1}.$
The well-posedness result reads as follows:
\begin{thm}[Well-posedness of autonomous evolutionary inclusions, \cite{Trostorff2012_nonlin_bd}]
\label{thm:well_posedness} Let $H$ be a Hilbert space, $M:B_{\mathbb{C}}\left(\frac{1}{2\nu_{0}},\frac{1}{2\nu_{0}}\right)\to L(H)$
a linear material law for some $\nu_{0}>0$ satisfying (\ref{eq:pos_material_law}).
Let $\nu>\nu_{0}$ and $A\subseteq L_{\nu}^{2}(\mathbb{R},H)\oplus L_{\nu}^{2}(\mathbb{R},H)$
a relation satisfying (H1). Then for each $f\in L_{\nu}^{2}(\mathbb{R},H)$
there exists a unique $u\in L_{\nu}^{2}(\mathbb{R},H)$ such that
\begin{equation}
(u,f)\in\overline{\partial_{0,\nu}M\left(\partial_{0,\nu}^{-1}\right)+A}.\label{eq:evol_inc_closure}
\end{equation}
Moreover, $\left(\overline{\partial_{0,\nu}M\left(\partial_{0,\nu}^{-1}\right)+A}\right)^{-1}$
is Lipschitz-continuous with a Lipschitz constant less than or equal
to $\frac{1}{c}.$ If in addition $A$ satisfies (H2), then the solution
operator $\left(\overline{\partial_{0,\nu}M\left(\partial_{0,\nu}^{-1}\right)+A}\right)^{-1}$
is causal.
\end{thm}
A typical example for a maximal monotone relation satisfying (H1) and (H2) is an extension
of a maximal monotone relation $A\subseteq H\oplus H$ satisfying
$(0,0)\in A.$ Indeed, if $A\subseteq H\oplus H$ is maximal monotone
and $(0,0)\in A$, we find that 
\begin{equation}
A_{\nu}\coloneqq\left\{ \left(u,v\right)\in L_{\nu}^{2}(\mathbb{R},H)\oplus L_{\nu}^{2}(\mathbb{R},H)\,|\,\left(u(t),v(t)\right)\in A\mbox{ for a.e. }t\in\mathbb{R}\right\} \label{eq:extension}
\end{equation}
is maximal monotone (see e.g. \cite[p. 31]{Morosanu1988}). Moreover,
$A_{\nu}$ obviously satisfies (H1) and (H2). 
\begin{rem}
It is possible to drop the assumption $(0,0)\in A,$ if one considers
the differential inclusion on the half-line $\mathbb{R}_{\geq0}$
instead of $\mathbb{R}$. In this case, an analogous definition of
the time derivative on the space $L_{\nu}^{2}(\mathbb{R}_{\geq0},H)$
can be given and the well-posedness of initial value problems of the
form
\begin{align*}
(u,f) & \in\left(\partial_{0,\nu}M_{0}+M_{1}+A_{\nu}\right)\\
M_{0}u(0+) & =u_{0},
\end{align*}
where $A_{\nu}$ is given as the extension of a maximal monotone relation
$A\subseteq H\oplus H$ and $M_{0},M_{1}\in L(H)$ satisfy a suitable
monotonicity constraint, can be shown similarly (see \cite{Trostorff2012_NA}).
\end{rem}
The general coupling mechanism as illustrated e.g. in \cite{Picard2009}
also works for the non-linear situation. This is illustrated in the
following example. 
\begin{example}[{\cite[Section 5.1]{Trostorff2012_NA}}]
We consider the equations of thermo-plasticity in a domain $\Omega\subseteq\mathbb{R}^{3}$,
given by
\begin{align}
M\partial_{0,\nu}^{2}u-\Dive\sigma & =f,\label{eq:displacement}\\
\rho\partial_{0,\nu}\theta-\dive\kappa\grad\theta+\tau_{0}\trace\,\Grad\partial_{0,\nu}u & =g.\label{eq:heat}
\end{align}
The functions $u\in L_{\nu}^{2}(\mathbb{R},L^{2}(\Omega)^{3})$ and
$\theta\in L_{\nu}^{2}(\mathbb{R},L^{2}(\Omega))$ are the unknowns,
standing for the displacement field of the medium and its temperature,
respectively. $f\in L_{\nu}^{2}(\mathbb{R},L^{2}(\Omega)^{3})$ and
$g\in L_{\nu}^{2}(\mathbb{R},L^{2}(\Omega))$ are given source terms.
The stress tensor $\sigma\in L_{\nu}^{2}(\mathbb{R},H_{\mathrm{sym}}(\Omega))$
is related to the strain tensor and the temperature by the following
constitutive relation, generalizing Hooke's law, 
\begin{equation}
\sigma=C(\Grad u-\varepsilon_{p})-c\trace^{\ast}\theta,\label{eq:Hooke-1}
\end{equation}

\end{example}
where $c>0$ and $C:H_{\mathrm{sym}}(\Omega)\to H_{\mathrm{sym}}(\Omega)$
is a linear, selfadjoint and strictly positive definite operator (the
elasticity tensor). The operator $\trace:H_{\mathrm{sym}}(\Omega)\to L^{2}(\Omega)$
is the usual trace for matrices and its adjoint can be computed by
$\trace^{\ast}f=\left(\begin{array}{ccc}
f & 0 & 0\\
0 & f & 0\\
0 & 0 & f
\end{array}\right)$. The function $\rho\in L^{\infty}(\Omega)$ describes the mass density
and is assumed to be real-valued and uniformly strictly positive,
$M,\kappa\in L^{\infty}(\Omega)^{3\times3}$ are assumed to be uniformly
strictly positive definite and selfadjoint and $\tau_{0}>0$ is a
real numerical parameter. The additional term $\varepsilon_{p}$ models
the inelastic strain and is related to $\sigma$ by
\begin{equation}
(\sigma,\partial_{0,\nu}\varepsilon_{p})\in\mathbb{I}\label{eq:incl}
\end{equation}

where $\mathbb{I}\subseteq H_{\mathrm{sym}}(\Omega)\oplus H_{\mathrm{sym}}(\Omega)$
is a maximal monotone relation satisfying $\trace\left[\mathbb{I}[H_{\mathrm{sym}}(\Omega)]\right]=\{0\}$,
i.e. each element in the post-set of $\mathbb{I}$ is trace-free.
If $\varepsilon_{p}=0,$ then (\ref{eq:displacement})-(\ref{eq:Hooke-1})
are exactly the equations of thermo-elasticity (see \cite[p. 420 ff.]{Picard_McGhee}).
The quasi-static case was studied in \cite{Chelminski2006} for a
particular relation $\mathbb{I}$, depending on the temperature $\theta$
under the additional assumption that the material possesses the linear
kinematic hardening property. We complete the system (\ref{eq:displacement})-(\ref{eq:incl})
by suitable boundary conditions for $u$ and $\theta$, for instance
$u,\theta=0$ on $\partial\Omega$. We set $v:=\partial_{0,\nu}u$
and $q:=\tau_{0}^{-1}c\kappa\grad_{c}\theta$. %

Following \cite[Subsection 5.1]{Trostorff2012_NA}, the system (\ref{eq:displacement})-(\ref{eq:incl})
can be written as 
\[
\left(\left(\begin{array}{c}
\theta\\
q\\
v\\
\sigma
\end{array}\right),\left(\begin{array}{c}
c\tau_{0}^{-1}f\\
0\\
F\\
0
\end{array}\right)\right)\in \partial_{0,\nu}M(\partial_{0,\nu}^{-1})+\left(\begin{array}{cccc}
0 & -\dive & 0 & 0\\
-\grad_{c} & 0 & 0 & 0\\
0 & 0 & 0 & -\Dive\\
0 & 0 & -\Grad_{c} & \mathbb{I}
\end{array}\right),
\]

where
\begin{equation*}
M(\partial_{0,\nu}^{-1})= \left(\begin{array}{cccc}
c\tau_{0}^{-1}w+\trace\, cC^{-1}c\trace^{\ast} & 0 & 0 & \trace\, cC^{-1}\\
0 & 0 & 0 & 0\\
0 & 0 & M & 0\\
C^{-1}c\trace^{\ast} & 0 & 0 & C^{-1}
\end{array}\right)
 +\partial_{0,\nu}^{-1}\left(\begin{array}{cccc}
0 & 0 & 0 & 0\\
0 & \kappa^{-1}c^{-1}\tau_{0} & 0 & 0\\
0 & 0 & 0 & 0\\
0 & 0 & 0 & 0
\end{array}\right).
\end{equation*}
Thus, we have that $M(\partial_{0,\nu}^{-1})=M_{0}+\partial_{0,\nu}^{-1}M_{1}$
with 
\begin{align*}
M_{0} & =\left(\begin{array}{cccc}
c\tau_{0}^{-1}w+\trace\, cC^{-1}c\trace^{\ast} & 0 & 0 & \trace\, cC^{-1}\\
0 & 0 & 0 & 0\\
0 & 0 & M & 0\\
C^{-1}c\trace^{\ast} & 0 & 0 & C^{-1}
\end{array}\right),\\
M_{1} & =\left(\begin{array}{cccc}
0 & 0 & 0 & 0\\
0 & \kappa^{-1}c^{-1}\tau_{0} & 0 & 0\\
0 & 0 & 0 & 0\\
0 & 0 & 0 & 0
\end{array}\right).
\end{align*}
It can easily be verified, that the material law $M(\partial_{0,\nu}^{-1})$
satisfies (\ref{eq:pos_material_law}). Thus, we only have to check
that 
\[
A\coloneqq\left(\begin{array}{cccc}
0 & -\dive & 0 & 0\\
-\grad_{c} & 0 & 0 & 0\\
0 & 0 & 0 & -\Dive\\
0 & 0 & -\Grad_{c} & 0
\end{array}\right)+\left(\begin{array}{cccc}
0 & 0 & 0 & 0\\
0 & 0 & 0 & 0\\
0 & 0 & 0 & 0\\
0 & 0 & 0 & \mathbb{I}
\end{array}\right)
\]
is maximal monotone (note that the other assumptions on $A$ are trivially
satisfied, since $A$ is given as in (\ref{eq:extension})). Since
$A$ is the sum of two maximal monotone operators its maximal monotonicity
can be obtained by assuming suitable boundedness constraints on $\mathbb{I}$
and applying classical perturbation results for maximal monotone operators.%
\footnote{The easiest assumption would be the boundedness of $\mathbb{I},$
i.e. for every bounded set $M$ the post-set $\mathbb{I}[M]$ is bounded.
For more advanced perturbation results we refer to \cite[p. 331 ff.]{papageogiou}.%
}

\subsection{The non-autonomous case\label{sub:The-non-autonomous-case-incl}}

We are also able to treat non-autonomous differential inclusions.
Consider the following problem
\begin{equation}
(u,f)\in\left(\partial_{0,\nu}M_{0}(\mathrm{m}_{0})+M_{1}(\mathrm{m}_{0})+A_{\nu}\right),\label{eq:non_auto_incl}
\end{equation}
where $M_{0},M_{1}\in L_{s}^{\infty}(\mathbb{R},L(H))$ and $A_{\nu}$
is the canonical extension of a maximal monotone relation $A\subseteq H\oplus H$
with $(0,0)\in A$ as defined in (\ref{eq:extension}). As in Subsection
\ref{sub:The-non-autonomous-case} we assume that $M_{0}$ satisfies
Hypotheses \ref{hyp:nonauto_linear} (\ref{selfadjoint})-(\ref{differentiable}). 

Our well-posedness result reads as follows:
\begin{thm}[Solution theory for non-autonomous evolutionary inclusions, \cite{Trostorff2013_nonautoincl}]
\label{thm:sol-theory} Let $M_{0},M_{1}\in L_{s}^{\infty}(\mathbb{R};L(H))$,
where $M_{0}$ satisfies Hypotheses \ref{hyp:nonauto_linear} (\ref{selfadjoint})-(\ref{differentiable}).
Moreover, we assume that $N(M_{0}(t))=N(M_{0}(0))$ for every $t\in\mathbb{R}$
and%
\footnote{We denote by $\iota_{R(M_{0}(0))}$ and $\iota_{N(M_{0}(0))}$ the
canonical embeddings into $H$ of $R(M_{0}(0))$ and $N(M_{0}(0))$,
respectively.%
} 
\begin{equation}
\bigvee_{c>0}\;\bigwedge_{t\in\mathbb{R}}:\iota_{R(M_{0}(0))}^{\ast}M_{0}(t)\iota_{R(M_{0}(0))}\geq c\mbox{ and }\iota_{N(M_{0}(0))}^{\ast}\Re M_{1}(t)\iota_{N(M_{0}(0))}\geq c.\label{eq:pos_def_non_auto_incl}
\end{equation}
Let $A\subseteq H\oplus H$ be a maximal monotone relation with $(0,0)\in A.$
Then there exists $\nu_{0}>0$ such that for every $\nu\geq\nu_{0}$
\[
\left(\overline{\partial_{0,\nu}M_{0}(\mathrm{m}_{0})+M_{1}(\mathrm{m}_{0})+A_{\nu}}\right)^{-1}:L_{\nu}^{2}(\mathbb{R},H)\to L_{\nu}^{2}(\mathbb{R},H)
\]
is a Lipschitz-continuous, causal mapping. Moreover, the mapping is
independent of $\nu$ in the sense that, for $\nu,\nu'\geq\nu_{0}$
and $f\in L_{\nu'}^{2}(\mathbb{R},H)\cap L_{\nu}^{2}(\mathbb{R},H)$
we have that 
\[
\left(\overline{\partial_{0,\nu'}M_{0}(\mathrm{m}_{0})+M_{1}(\mathrm{m}_{0})+A_{\nu'}}\right)^{-1}(f)=\left(\overline{\partial_{0,\nu}M_{0}(\mathrm{m}_{0})+M_{1}(\mathrm{m}_{0})+A_{\nu}}\right)^{-1}(f).
\]

\end{thm}
Note that in Subsection \ref{sub:The-non-autonomous-case} we do not
require that $N(M_{0}(t))$ is $t$-independent. However, in order
to apply perturbation results, which are the key tools for proving
the well-posedness of (\ref{eq:non_auto_incl}), we need to impose
this additional constraint (compare \cite[Theorem 2.19]{Picard2013_nonauto}).

\subsection{Problems with non-linear boundary conditions\label{sub:Problems-with-non-linear-bdy}}

As we have seen in Subsection \ref{sub:incl_auto} the maximal monotonicity
of the relation $A\subseteq L_{\nu}^{2}(\mathbb{R},H)\oplus L_{\nu}^{2}(\mathbb{R},H)$
plays a crucial role for the well-posedness of the corresponding evolutionary
problem (\ref{eq:evol_incl}). Motivated by several examples from
mathematical physics, we might restrict our attention to (possibly
non-linear) operators $A:D(A)\subseteq L_{\nu}^{2}(\mathbb{R},H)\to L_{\nu}^{2}(\mathbb{R},H)$
of a certain block structure. As a motivating example, we consider
the wave equation with impedance-type boundary conditions, which was
originally treated in \cite{Picard2012_Impedance}. 
\begin{example}
\label{ex:impedance}Let $\Omega\subseteq\mathbb{R}^{n}$ be open
and consider the following boundary value problem 
\begin{align}
\partial_{0,\nu}^{2}u-\dive\grad u & =f\mbox{ on }\Omega,\label{eq:wave}\\
\left(\partial_{0,\nu}^{2}a(\mathrm{m})u+\grad u\right)\cdot N & =0\mbox{ on }\partial\Omega,\label{eq:impedance_bd}
\end{align}
where $N$ denotes the outward normal vector field on $\partial\Omega$
and $a\in L^{\infty}(\Omega)^{n}$ such that $\dive a\in L^{\infty}(\Omega)$%
\footnote{Here we mean the divergence in the distributional sense.%
}. Formulating (\ref{eq:wave}) as a first order system we obtain 
\[
\partial_{0,\nu}\left(\begin{array}{c}
v\\
q
\end{array}\right)+\left(\begin{array}{cc}
0 & \dive\\
\grad & 0
\end{array}\right)\left(\begin{array}{c}
v\\
q
\end{array}\right)=\left(\begin{array}{c}
f\\
0
\end{array}\right),
\]
where $v\coloneqq\partial_{0,\nu}u$ and $q\coloneqq-\grad u.$ The
boundary condition (\ref{eq:impedance_bd}) then reads as 
\[
\left(\partial_{0,\nu}a(\mathrm{m})v-q\right)\cdot N=0\mbox{ on }\partial\Omega.
\]
The latter condition can be reformulated as 
\[
a(\mathrm{m})v-\partial_{0,\nu}^{-1}q\in D(\dive_{c}),
\]
where $\dive_{c}$ is defined as in Definition \ref{def:div_grad}.
Thus, we end up with a problem of the form 
\[
\left(\partial_{0,\nu}+A\right)\left(\begin{array}{c}
v\\
q
\end{array}\right)=\left(\begin{array}{c}
f\\
0
\end{array}\right),
\]
where $A\subseteq\left(\begin{array}{cc}
0 & \dive\\
\grad & 0
\end{array}\right)$ with $D(A)\coloneqq\left\{ (v,q)\in D(\grad)\times D(\dive)\,|\, a(\mathrm{m})v-\partial_{0,\nu}^{-1}q\in D(\dive_{c})\right\} $.
In order to apply the solution theory, we have to ensure that the
operator $A$, defined in that way, is maximal monotone as an operator
in $L_{\nu}^{2}(\mathbb{R},L^2(\Omega)\oplus L^{2}(\Omega)^{n})$.\end{example}
\begin{rem}
In \cite{Picard2012_Impedance} a more abstract version of Example
\ref{ex:impedance} was studied, where the vector field $a$ was replaced
by a suitable material law operator $a(\partial_{0,\nu}^{-1})$ as
it is defined in Subsection \ref{sub:evo_eq_auto}.
\end{rem}
Following this guiding example, we are led to consider restrictions
$A$ of block operator matrices 
\[
\left(\begin{array}{cc}
0 & D\\
G & 0
\end{array}\right),
\]
where $G:D(G)\subseteq H_{0}\to H_{1}$ and $D:D(D)\subseteq H_{1}\to H_{0}$
are densely defined closed linear operators satisfying $D^{\ast}\subseteq-G$
and consequently $G^{\ast}\subseteq-D$. We set $D_{c}\coloneqq-G^{\ast}$
and $G_{c}\coloneqq-D^{\ast}$ and obtain densely defined closed linear
restrictions of $D$ and $G$, respectively. Regarding the example
above, $G=\grad$ and $D=\dive$, whereas $G_{c}=\grad_{c}$ and $D_{c}=\dive_{c}$.
Having this guiding example in mind, we interpret $G_{c}$ and $D_{c}$
as the operators with vanishing boundary conditions and $G$ and $D$
as the operators with maximal domains. This leads to the following
definition of so-called abstract boundary data spaces.
\begin{defn}[{\cite[Subsection 5.2]{Picard2012_comprehensive_control}}]
 Let $G_{c},D_{c},G$ and $D$ as above. We define 
\begin{align*}
BD(G)\coloneqq & D(G_{c})^{\bot_{D(G)}}=N(1-DG),\\
BD(D)\coloneqq & D(D_{c})^{\bot_{D(D)}}=N(1-GD),
\end{align*}
where $D(G_{c})$ and $D(D_{c})$ are interpreted as closed subspaces
of the Hilbert spaces $D(G)$ and $D(D),$ respectively, equipped
with their corresponding graph norms. Consequently, we have the following
orthogonal decompositions 
\begin{align}
D(G) & =D(G_{c})\oplus BD(G)\label{eq:orth_decomp_bd}\\
D(D) & =D(D_{c})\oplus BD(D).\nonumber 
\end{align}
\end{defn}
\begin{rem}
The decomposition (\ref{eq:orth_decomp_bd}) could be interpreted
as follows: Each element $u$ in the domain of $G$ can be uniquely
decomposed into two elements, one with vanishing boundary values (the
component lying in $D(G_{c})$) and one carrying the information of
the boundary value of $u$ (the component lying in $BD(G)$). In the
particular case of $G=\grad$ a comparison of $BD(G)$ and the classical
trace space $H^{\frac{1}{2}}(\partial\Omega)$ can be found in \cite[Section 4]{Trostorff2013_maxmon_bd}.
\end{rem}
Let $\iota_{BD(G)}:BD(G)\to D(G)$ and $\iota_{BD(D)}:BD(D)\to D(D)$
denote the canonical embeddings. An easy computation shows that $G[BD(G)]\subseteq BD(D)$
and $D[BD(D)]\subseteq BD(G)$ and thus, we may define 
\begin{align*}
\stackrel{\bullet}{G}\coloneqq\iota_{BD(D)}^{\ast}G\iota_{BD(G)}:BD(G) & \to BD(D)\\
\stackrel{\bullet}{D}\coloneqq\iota_{BD(G)}^{\ast}D\iota_{BD(D)}:BD(D) & \to BD(G).
\end{align*}
These two operators share a surprising property.
\begin{prop}[{\cite[Theorem 5.2]{Picard2012_comprehensive_control}}]
 The operators $\stackrel{\bullet}{G}$ and $\stackrel{\bullet}{D}$
are unitary and 
\[
\left(\stackrel{\bullet}{G}\right)^{\ast}=\stackrel{\bullet}{D}\mbox{ as well as }\left(\stackrel{\bullet}{D}\right)^{\ast}=\stackrel{\bullet}{G}.
\]

\end{prop}
Coming back to our original question, when $A\subseteq\left(\begin{array}{cc}
0 & D\\
G & 0
\end{array}\right)$ defines a maximal monotone operator, we find the following characterization.
\begin{thm}[{\cite[Theorem 3.1]{Trostorff2013_maxmon_bd}}]
\label{thm:char_bd_cond} Let $G$ and $D$ be as above. A restriction
$A\subseteq\left(\begin{array}{cc}
0 & D\\
G & 0
\end{array}\right)$ is maximal monotone, if and only if there exists a maximal monotone
relation $h\subseteq BD(G)\oplus BD(G)$ such that 
\[
D(A)=\left\{ (u,v)\in D(G)\times D(D)\,|\,\left(\iota_{BD(G)}^{\ast}u,\stackrel{\bullet}{D}\iota_{BD(D)}^{\ast}v\right)\in h\right\} .
\]
\end{thm}
\begin{example}
$\,$\\
\begin{enumerate}[(a)]

\item In Example \ref{ex:impedance}, the operators $G$ and $D$
are $\grad$ and $\dive$, respectively and the relation $h\subseteq BD(\grad)\oplus BD(\grad)$
is given by 
\[
(x,y)\in h\Leftrightarrow\partial_{0,\nu}^{-1}y=\stackrel{\bullet}{\dive}\iota_{BD(\dive)}^{\ast}a(\mathrm{m})\iota_{BD(\grad)}x.
\]
Indeed, by the definition of the operator $A$ in Example \ref{ex:impedance},
a pair $(v,q)\in D(\grad)\times D(\dive)$ belongs to $D(A)$ if and
only if 
\begin{align*}
a(\mathrm{m})v-\partial_{0,\nu}^{-1}q\in D(\dive_{c}) & \Leftrightarrow\iota_{BD(\dive)}^{\ast}\left(a(\mathrm{m})v-\partial_{0,\nu}^{-1}q\right)=0\\
 & \Leftrightarrow\partial_{0,\nu}^{-1}\iota_{BD(\dive)}^{\ast}q=\iota_{BD(\dive)}^{\ast}a(\mathrm{m})\iota_{BD(\grad)}\iota_{BD(\grad)}^{\ast}v\\
 & \Leftrightarrow\partial_{0,\nu}^{-1}\stackrel{\bullet}{\dive}\iota_{BD(\dive)}^{\ast}q=\stackrel{\bullet}{\dive}\iota_{BD(\dive)}^{\ast}a(\mathrm{m})\iota_{BD(\grad)}\iota_{BD(\grad)}^{\ast}v\\
 & \Leftrightarrow\left(\iota_{BD(\grad)}^{\ast}v,\stackrel{\bullet}{\dive}\iota_{BD(\dive)}^{\ast}q\right)\in h.
\end{align*}
Thus, if we show that $h$ is maximal monotone, we get the maximal
monotonicity of $A$ by Theorem \ref{thm:char_bd_cond}. For doing
so, we have to assume that the vector field $a$ satisfies a positivity
condition of the form 
\begin{equation}
\Re\intop_{-\infty}^{0}\left(\langle\grad u|\partial_{0,\nu}a(\mathrm{m})u\rangle(t)+\langle u|\dive\partial_{0,\nu}a(\mathrm{m})u\rangle(t)\right)e^{-2\nu t}\, dt\geq0\label{eq:impedance}
\end{equation}
for all $u\in D(\partial_{0,\nu})\cap D(\grad).$ In case of a smooth
boundary, the latter can be interpreted as a constraint on the angle
between the vector field $a$ and the outward normal vector field
$N$. Indeed, condition (\ref{eq:impedance_bd}) implies the monotonicity
of $h$ and also of the adjoint of $h$ (note that here, $h$ is a
linear relation). Both facts imply the maximal monotonicity of $h$
(the proof can be found in \cite[Section 4.2]{Trostorff2012_nonlin_bd}).

\item In the theory of contact problems in elasticity we find so-called
frictional boundary conditions at the contact surfaces. These conditions
can be modeled for instance by sub-gradients of lower semi-continuous
convex functions (see e.g. \cite[Section 5]{Migorski_2009}), which
are the classical examples of maximal monotone relations%
\footnote{Note that not every maximal monotone relation can be realized as a
sub-gradient of a lower semi-continuous convex function. Indeed, sub-gradients
are precisely the cyclic monotone relations, see \cite[Theoreme 2.5]{Brezis1971}.%
}.\\
Let $\Omega\subseteq\mathbb{R}^{n}$ be a bounded domain. We recall
the equations of elasticity from Example \ref{ex:visco_elastic}
\begin{equation}
\left(\partial_{0,\nu}\left(\begin{array}{cc}
1 & 0\\
0 & C^{-1}
\end{array}\right)+\left(\begin{array}{cc}
0 & -\Dive\\
-\Grad & 0
\end{array}\right)\right)\left(\begin{array}{c}
v\\
T
\end{array}\right)=\left(\begin{array}{c}
f\\
0
\end{array}\right)\label{eq:elastic-1}
\end{equation}
and assume that the following frictional boundary condition should
hold on the boundary $\partial\Omega$ (for a treatment of boundary
conditions just holding on different parts of the boundary, we refer
to \cite{Trostorff2013_maxmon_bd}):
\begin{equation}
(v,-T\cdot N)\in g,\label{eq:frictional}
\end{equation}
where $N$ denotes the unit outward normal vector field and $g\subseteq L^{2}(\partial\Omega)^{n}\oplus L^{2}(\partial\Omega)^{n}$
is a maximal monotone relation, which, for simplicity, we assume to
be bounded. We note that in case of a smooth boundary, there exists
a continuous injection $\kappa:BD(\Grad)\to L^{2}(\partial\Omega)^{n}$
(see \cite{Trostorff2013_maxmon_bd}) and we may assume that $\kappa[BD(\Grad)]\cap[L^{2}(\partial\Omega)^{n}]g\ne\emptyset.$
Then, according to \cite[Proposition 2.6]{Trostorff2013_maxmon_bd},
the relation 
\[
\tilde{g}\coloneqq\kappa^{\ast}g\kappa=\left\{ (x,\kappa^{\ast}y)\in BD(\Grad)\times BD(\Grad)\,|\,(\kappa x,y)\in g\right\} 
\]
is maximal monotone as a relation on $BD(\Grad)$ and the boundary
condition (\ref{eq:frictional}) can be written as 
\[
(\iota_{BD(\Grad)}^{\ast}v,-\stackrel{\bullet}{\Dive}\iota_{BD(\Dive)}^{\ast}T)\in\tilde{g}.
\]
Thus, by Theorem \ref{thm:char_bd_cond}, the operator 
\begin{align*}
A & \subseteq\left(\begin{array}{cc}
0 & -\Dive\\
-\Grad & 0
\end{array}\right)\\
D(A) & \coloneqq\left\{ (v,T)\in D(\Grad)\times D(\Dive)\,|\,\left(\iota_{BD(\Grad)}^{\ast}v,-\stackrel{\bullet}{\Dive}\iota_{BD(\Dive)}^{\ast}T\right)\in\tilde{g}\right\} 
\end{align*}
is maximal monotone and hence, Theorem \ref{thm:well_posedness} is
applicable and yields the well-posedness of (\ref{eq:elastic-1})
subject to the boundary condition (\ref{eq:frictional}).

\end{enumerate}
\end{example}

\section{Conclusion}

We have illustrated that many (initial, boundary value) problems of
mathematical physics fit into the class of so-called evolutionary
problems. Having identified the particular role of the time-derivative,
we realize that many equations (or inclusions) of mathematical physics
share the same type of solution theory in an appropriate Hilbert space
setting. The class of problems accessible is widespread and goes from
standard initial boundary value problems as for the heat equation,
the wave equation or Maxwell's equations etc. to problems of mixed
type and to integro-differential-algebraic equations. We also demonstrated
first steps towards a discussion of issues like exponential stability
and continuous dependence on the coefficients in this framework. The
methods and results presented provide a general, unified approach
to numerous problems of mathematical physics.

\section{Acknowledgements}

We thank the organizers, Wolfgang Arendt, Ralph Chill and Yuri Tomilov,
of the conference ``Operator Semigroups meet Complex Analysis, Harmonic
Analysis and Mathematical Physics'' held in Herrnhut in 2013 for
organizing a wonderful conference dedicated to Charles Batty's 60th
birthday. We also thank Charles Batty for his manifold, inspiring
contributions to mathematics and of course for thus providing an excellent
reason to meet experts from all over the world in evolution equations
and related subjects.

\end{document}